\documentclass[a4paper]{amsart}

\usepackage{tikz}

\allowdisplaybreaks

\usepackage[misc]{ifsym}
     \makeatletter
     \def\section{\@startsection{section}{1}%
     \z@{.7\linespacing\@plus\linespacing}{.5\linespacing}%
     {\bfseries
     \centering
     }}
     \def\@secnumfont{\bfseries}
     \makeatother
\newtheorem{theorem}{Theorem}[section]
\newtheorem{corollary}[theorem]{Corollary}

\newtheorem{lemma}[theorem]{Lemma}
\newtheorem{proposition}[theorem]{Proposition}

\theoremstyle{definition}
\newtheorem{definition}[theorem]{Definition}
\newtheorem{remark}[theorem]{Remark}

\newtheorem{example}[theorem]{Example}

\usepackage{amsrefs}
\usepackage[utf8]{inputenc}
\usepackage{enumerate}
\usepackage{booktabs}

\usepackage[inline]{enumitem}
\usepackage{times}
\usepackage{multirow}
\usepackage{float}

\usepackage{caption}
\usepackage{subcaption}

\usepackage{extpfeil}

\usepackage[inline]{enumitem}
\usepackage{stmaryrd}
\usepackage[normalem]{ulem} 
\usepackage{stackrel}
\usepackage{graphicx}
\usepackage{mathtools}
\usepackage{url}

\usepackage{algorithm}
\usepackage{appendix}
\usepackage{algpseudocode}

\usepackage{hyperref}
\hypersetup{urlcolor=blue, citecolor=red}

\newcommand{\R}{\mathbb{R}}
\newcommand{\X}{\mathbb{X}}
\newcommand{\U}{\mathbb{U}}

\usepackage{xcolor}
\newcommand{\C}{C}
\newcommand{\F}{F}
\renewcommand{\S}{S}
\newcommand{\T}{T}
\newcommand{\caS}{\mathcal{S}}
\newcommand{\caC}{\mathcal{C}}
\newcommand{\caF}{\mathcal{F}}
\newcommand{\caT}{\mathcal{T}}

\newcommand{\e}{\varepsilon}

\newcommand{\Int}{\mathbf{Int}}

\title[Combinatorial Models for Phylogenetics]{Combinatorial Topological Models for Phylogenetic Networks and the Mergegram Invariant}
\author[Paweł Dłotko, Jan F Senge and Anastasios Stefanou]{}


\thanks{$^*$Corresponding author: Anastasios Stefanou}

\begin{document}

\begin{abstract}
Mutations of genetic sequences are often accompanied by their recombinations, known as phylogenetic networks. These networks are typically reconstructed from coalescent processes that may arise from optimal merging or fitting together a given set of phylogenetic trees. Nakhleh formulated the phylogenetic network reconstruction problem (PNRP): Given a family of phylogenetic trees over a common set of taxa, is there a unique minimal phylogenetic network whose set of spanning trees contains the family?

Inspired by ideas from topological data analysis (TDA), we devise lattice-diagram models for phylogenetic networks and filtrations, the cliquegram and the facegram, both generalizing the dendrogram (filtered partition) model of phylogenetic trees. Both models allow us to solve the PNRP rigorously. The solutions are obtained by taking the join of the dendrograms on the lattice of cliquegrams or facegrams. Furthermore, computing the join-facegram is polynomial in the size and number of the input trees.

Cliquegrams and facegrams can be challenging to work with when the number of taxa is large. We propose a topological invariant of facegrams and filtrations, called the mergegram, by extending a construction by Elkin and Kurlin defined on dendrograms. We show that the mergegram is invariant of weak equivalences of filtrations which, in turn, implies that it is a 1-Lipschitz stable invariant with respect to M\'emoli's tripod distance. The mergegram, can be used as a computable proxy for phylogenetic networks and filtrations of datasets. We illustrate the utility of these new TDA-concepts to phylogenetics, by performing experiments with artificial and biological data.
\end{abstract}

\maketitle

\centerline{\scshape
Pawel Dlotko$^{{\href{mailto:pdlotko@impan.pl}{\textrm{\Letter}}}1}$,
Jan Felix Senge$^{{\href{mailto:janfsenge@uni-bremen.de}{\textrm{\Letter}}}1,2}$
and Anastasios Stefanou$^{{\href{mailto:stefanou@uni-bremen.de}{\textrm{\Letter}}}*2}$}

\medskip

{\footnotesize
 \centerline{$^1$Dioscuri Centre in Topological Data Analysis, Institute of Mathematics, Polish Academy of Sciences, Poland.}
} 

\medskip

{\footnotesize
 \centerline{$^2$Institute for Algebra, Geometry, Topology and its Applications (ALTA), Department of Mathematics, University of Bremen, Germany.}
}

\bigskip

\section{Introduction}
\subsection{Motivation and related work}
\textbf{Dendrogram, Phylogenetic tree, and the Tree of Life}
Let us consider a set $X$ that contains, simply speaking, a collection of genes, also called \emph{taxa}. Their evolution can be tracked back in time, as in the \emph{coalescent model} \cite{kingman1982genealogy}. The model of this type, for any pair of elements in taxa $X$, keeps track of their least common ancestor. 
Doing so, at every time $t$ the tree is represented by a partition of $X$, called a \emph{dendrogram} or a \emph{phylogenetic tree}. 

The mutation-driven temporal evolution of the considered taxa $X$ can be modeled by several stochastic processes (e.g.~Markovian) that can be used for the analysis of the evolution of genetic sequences \cite{kingman1982genealogy}.  
Such an analysis also yields a dendrogram as a summary of the evolution of $X$. 

Dendrograms are also modeled and studied via \emph{tropical geometry} \cite{maclagan2021introduction}. The dendrogram is modeled as a special type of metric space, called an \emph{ultrametric}; the geometry of the space of dendrograms has been studied in \cite{billera2001geometry}. See also \cites{feichtner2006complexes, kozlov2008, kozlov1999complexes, lin2017convexity} for related work on directed trees.
In both cases, a dendrogram serves as a summary of the evolution of a given taxa set. 

In data science, dendrograms are also well known. They are the basic structures of \emph{hierarchical clustering}~\cite{carlsson2010characterization}. In this case, a collection of data points $X$ taken from a metric space will be connected in a dendrogram at the level equal to the diameter of $X$. 
An important distinction between the methods presented here and the hierarchical clustering is that there is typically no canonical taxa set, which may bring additional labelling information into the structure of hierarchical clustering.

A dendrogram is a special case of a \emph{treegram}  over $X$ \cite{smith2016hierarchical}, which is a nested sequence of \emph{subpartitions of $X$} (i.e.~partitions of subsets of $X$). 
In contrast to dendrograms, in treegrams different taxa elements (leaves) may appear at different times, as indicated in \textit{Fig.~\ref{fig:treegram}}.

As shown by Carlsson and F.~M\'emoli~\cite{carlsson2010characterization}, every dendrogram $\theta_X$ over $X$ gives rise to an ultrametric $U_X:X\times X\to \R_{\geq0}$ (meaning a metric, for which the triangle inequality takes the form {\sloppy$U_X(x,z) \leq \allowbreak \max\{U_X(x,y),U_X(y,z)\}$}, for all triples $x,y,z$ in $X$). In particular, they showed that the assignment $\theta_X\mapsto U_X$ yields an equivalence between dendrograms and ultrametrics. 
Thm.~6 in~\cite{smith2016hierarchical} by Z.~Smith, S.~Chowdhury, and F.~M\'emoli, generalizes this equivalence to an equivalence between treegrams $\caT _X$ and \emph{symmetric ultranetworks} $U_X$ (a symmetric real-valued matrix such that {\sloppy$U_X(x,z)\leq \allowbreak \max\{U_X(x,y),U_X(y,z)\}$}). 
For an example illustrating this equivalence, see \textit{Fig.~\ref{fig:treegram}}.

\begin{figure}[ht]
    \centering
    \includegraphics[width=0.7\textwidth]{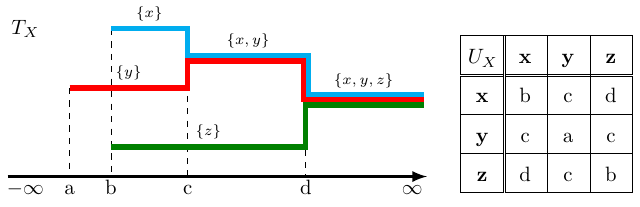}
    \caption[width=\textwidth]{{\bf Example of a treegram and its symmetric ultranetwork.} On the left, a treegram $\caT _X$ build over $X = \{x,y,z\}$, where the horizontal axis represents time (traced backwards in the phylogenetic tree setting). 
    On the right, the corresponding symmetric ultranetwork $U_X$.}
    \label{fig:treegram}
\end{figure}

\textbf{Persistent homology diagrams}
As discussed above, the dendrogram is the main model describing the evolution of species in phylogenetics, and it is heavily used in the hierarchical clustering of datasets (viewed as finite metric spaces). 
Due to their combinatorial nature, dendrograms can neither be directly vectorized nor used in a machine learning pipeline. However, in the case of single-linkage hierarchical clustering there is a process called the \emph{Elder-rule} which extracts a \emph{persistence diagram} from the single linkage dendrogram as in \cite{curry2018fiber}*{Defn.~3.7}. Persistence Diagrams are multisets of points in the extended plane $(\R \cup \{ \infty \})^2$ and they have been the main focus of the research field of Topological Data Analysis (TDA) \cite{chazal2009proximity}.
The persistence diagram obtained in this way corresponds to the $0$-dimensional persistent homology diagram with respect to the \emph{Vietoris-Rips filtration} of the dataset. This is a special case of the more broad method of $p$-dimensional persistent homology diagrams of filtrations, which capture the evolution of $p$-dimensional cycles across the filtration of the dataset \cite{edelsbrunner2022computational}. 

\begin{figure}[ht]
\centering
\includegraphics[width=0.9\linewidth]{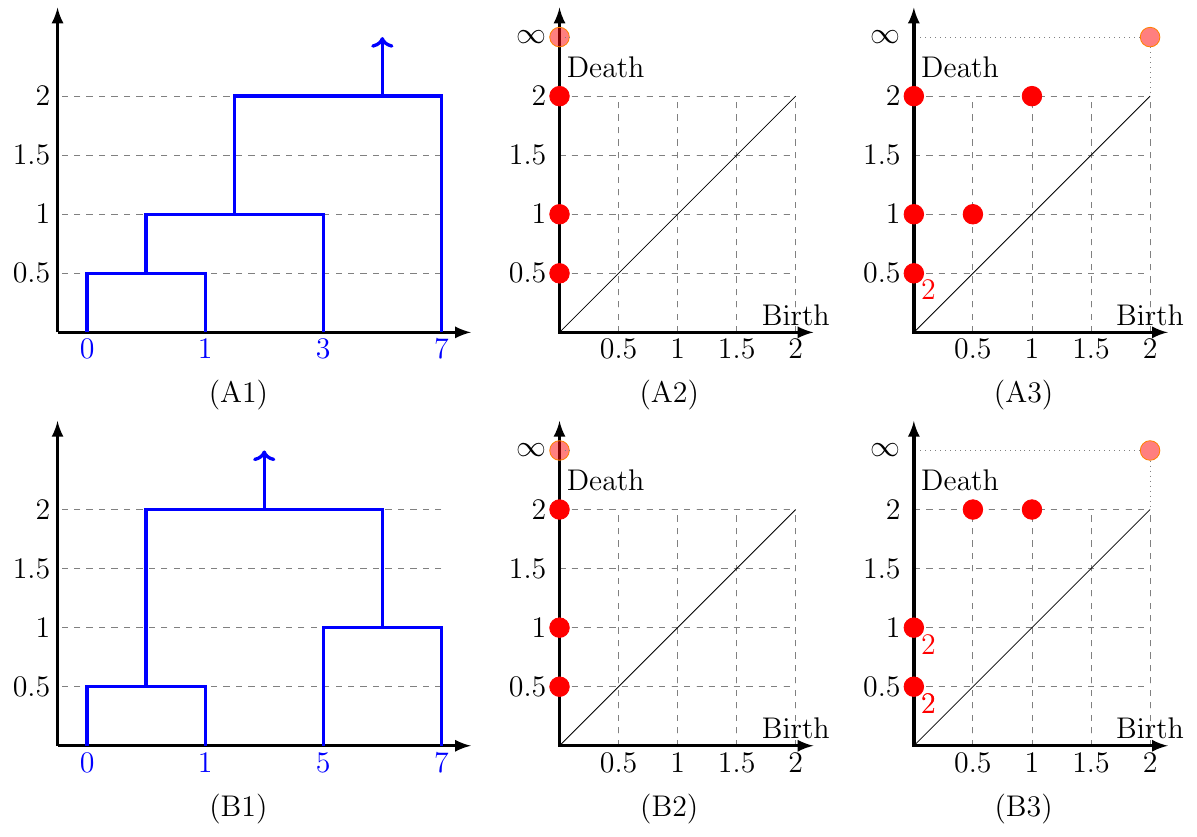}
\caption[]{{\bf Dendrograms, their persistence diagrams and mergegrams.} 
Consider the finite metric subspaces $X=\{0,1,3,7\}$ and $Y=\{0,1,5,7\}$ of $(\R,|\cdot |)$. (A1) shows the single-linkage dendrogram of $X$, (A2) shows the $0$-dimensional persistent homology diagram of $X$, and (A3) shows the mergegram of the single-linkage dendrogram of $X$. (B1) shows the single-linkage dendrogram of $Y$, (B2) shows the $0$-dimensional persistent homology diagram of $Y$, and (B3) shows the mergegram of the single-linkage dendrogram of $Y$. The datasets $X$ and $Y$ have the same $0$-dimensional persistent diagrams but different mergegrams.}
\label{fig:dendrogram_persistence_mergegram}
\end{figure}

\textbf{The mergegram of a dendrogram}
Persistent homology diagrams will not always be able to distinguish different metric spaces (see \textit{Fig.~\ref{fig:dendrogram_persistence_mergegram}}). In \cite{elkin2020mergegram}, the authors considered a new invariant for datasets, called the \emph{mergegram}. The mergegram is defined for every dendrogram as the multiset of interval lifespans of the blocks of the partitions across the dendrogram (in other words, the multiset of the intervals of all the edges in the underlying tree of the dendrogram); in the case of datasets, the mergegram associated with a dataset is the mergegram of the single linkage dendrogram of the dataset. In particular, the authors showed that the mergegram is a stronger invariant than the $0$-dimensional persistent homology diagram of datasets, i.e. if two metric spaces have the same mergegram then they have the same $0$-dimensional persistence diagrams, but not vice versa (see \textit{Fig.~\ref{fig:dendrogram_persistence_mergegram}}). 

By construction of the mergegram, as a multiset of intervals, we can also think of it as a persistence diagram and therefore being equipped with the interleaving and bottleneck distances \cite{chazal2009proximity}. The mergegram -when viewed as persistence diagram- (i) is stable and at the same time (ii)  it is naturally equipped with a family of computable metrics (e.g. Wasserstein and Bottleneck distances) while (iii) still capturing essential information about the dendrograms. Another consequence of viewing mergegrams as persistence diagrams is the possibility of using existing machinery/vectorization of persistence diagrams for a subsequent data analysis pipeline.

\textbf{Loops in the tree of life}
In real world, mutations of genetic sequences are often accompanied by their recombinations. Such joint phenomena are modeled by a \emph{phylogenetic network} that generalizes the notion of a phylogenetic tree. The most common model for visualizing those networks is that of a \emph{rooted directed acyclic graph (rooted DAG)}, \cite{nakhleh2010evolutionary}. 
The recombination events will give rise to loops in the DAG.
These DAGs are typically generated from coalescent processes that may arise from optimal merging or fitting together a given set of phylogenetic trees. This process is known as \emph{the phylogenetic network reconstruction} \cite{nakhleh2010evolutionary}. 
L.~Nakhleh proposed the following formulation which will be referred to in this paper as the \emph{Phylogenetic Network Reconstruction Problem} (PNRP), (pg.~130 in \cite{nakhleh2010evolutionary}):
Given a family $\caT$ of phylogenetic trees over a common set of taxa $X$, is there a unique minimal phylogenetic network $N$ whose set of spanning trees $\T (N)$ \emph{contains} the family $\caT$?
There are different answers to PNRP, due to the different ways of defining what a `minimal network' is (based on different optimization criteria). 
The existing methods and software use DAG-type models to describe the phylogenetic networks \cites{cardona2008perl,huson2008summarizing,huson2012dendroscope,solis2017phylonetworks,than2008phylonet}. 

\subsection{Overview of our contributions}
In this work, we propose two lattice-theoretic models for phylogenetic reconstruction networks. These approaches have several advantages. Firstly, each of these lattice-models provides us with a unique solution to the PNRP (see \textit{Prop.~\ref{prop:inverse}} and \textit{Rem.~\ref{rem:PNRP}}).
Moreover, our models are constructed using tools
from TDA such as networks, filtrations \cite{memoli2017distance}, persistence diagrams \cite{chazal2009proximity}, mergegrams \cite{elkin2020mergegram}, and their associated metrics \cites{bubenik2015metrics,chazal2009proximity}. In addition, those lattice models provide us with interleaving or Gromov-Hausdorff type of metrics for comparing pairs of phylogenetic trees and networks over different sets of taxa \cite{kim2023interleaving}. Below, we give a detailed overview of our results. 
In \textit{Tab.~\ref{tab:correspondence}}, we listed the difference correspondences of lattices.

\textbf{Our first order viewpoint on the phylogenetic reconstruction process}
In the setting of phylogenetic trees, we cannot model recombination phenomena, that would naturally correspond to a loop in the structure (not present in trees). Hence, the simple model of trees (modeled with dendrograms) should be extended to rooted DAGs with an additional structure allowing to keep track of the evolution of taxa. Such a structure is modeled, in this work, by a \emph{cliquegram} (see \textit{Defn.~\ref{dfn:cliquegram}}), that is, a nested family of sets of cliques over $X$, as presented in \textit{Fig.~\ref{fig:cliquegram}}. 
The name cliquegram comes from that graph theoretical notion of (maximal) cliques, i.e. the (largest) subsets of vertices in a graph whose members are all connected by edges in the graph. Note that when the taxa sets decorating the edges of the cliquegram are forgotten, a rooted DAG is recovered.

\begin{figure}[ht]
    \centering
    \includegraphics[width=0.7\textwidth]{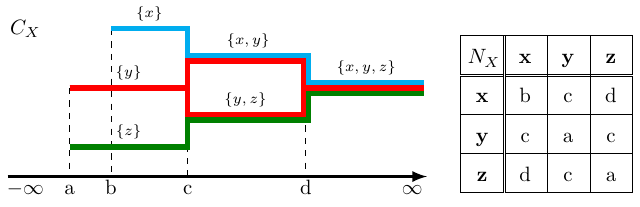}
    \caption[width=\textwidth]{{\bf Cliquegram and its phylogenetic network.}
    On the left there is a cliquegram $\caC _X$, on the right - the associated phylogenetic network $N_X$.}
    \label{fig:cliquegram}
\end{figure}
In the cliquegram, presented in \textit{Fig.~\ref{fig:cliquegram}}, the time axis is represented by a horizontal line. In order to recover the taxa set of biological entities at the time $t$, a vertical line should be placed at time $t$ and the cliques at the intersection of the structure with that line will indicate the taxa set (a.k.a.~taxonomy) of those entities. For instance, for $t \in (c,d)$, we have two entities; the first having a taxa set $\{x,y\}$ and the second having a taxa set $\{y,z\}$. Note that the intersection of their taxonomies, $\{y\}$ (present at times $t<c$), is a common taxon of both entities and creates the loop in the structure. 
That implies the existence of their common predecessor, whereas the union of their taxonomies, $\{x,y,z\}$ (present at times $t > d$), implies the existence of their common ancestor. 

While dendrograms naturally correspond to ultrametrics, cliquegrams also have a corresponding structure as a subclass of symmetric networks, which we call \emph{phylogenetic networks} (see \textit{Defn.~\ref{dfn:phylo-net}}). There is an isomorphism between the lattice $\mathbf{Net}(X)$, of the phylogenetic networks $N_X$ over $X$, and the lattice $\mathbf{Cliqgm}(X)$ of all cliquegrams $\caC_X$ over $X$ explicitly defined in \textit{Prop.~\ref{prop:equivalence}}. See the right part of the \textit{Fig.~\ref{fig:cliquegram}} as an example of the correspondence; Given $a,b \in \{x,y,z\}$, the entry of the matrix at the position $(a,b)$ indicates the time of coalescence of the taxa $a$ and $b$ that can be read out of the cliquegram by taking the time instance when they are merged together. 

One of the main observations of this work is that the lattice structure of cliquegrams over $X$, as opposed to ordinary rooted DAGs, enables us to provide a \emph{unique} answer to the PNRP, meaning that the PNRP has a unique solution in the setting of cliquegrams: Given any family $\caF$ of treegrams over the common set of taxa $X$, there exists a unique minimal cliquegram $\caC_X^{(\caF)}$ (w.r.t.~the partial order $\leq$ in the lattice of all cliquegrams over $X$) containing each treegram in $\caF$. The cliquegram $\caC_X$ is given explicitly as the join of $\caF$, $\vee\caF$ (see \textit{Rem.~\ref{rem:PNRP}} and \textit{Prop.~\ref{prop:clique-dense}}). See also \textit{Fig.~\ref{fig:treegram-dec}}. 
\begin{figure}
    \centering
    \includegraphics[width=1\textwidth]{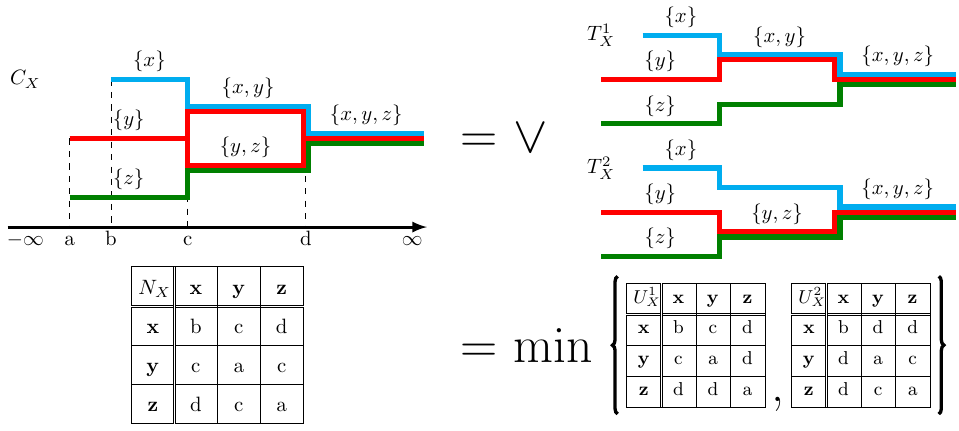}
    \caption{{\bf The join of two treegrams in the cliquegram setting.}
    An example of cliquegram $\caC_X$ over $X$ obtained as the join of two treegrams over $X$. The associated phylogenetic network $N_X$ of the cliquegram $\caC_X$ can be obtained as the pointwise minimum of the associated ultranetworks of the spanning treegrams of $\caC_X$.}
    \label{fig:treegram-dec}
\end{figure}
 This observation allows us to devise a simple algorithm for computing the join-cliquegram from any given family of treegrams. 
 The join-cliquegram $\caC_X$ of the family $\caT$, is then given by the maximal cliques of the Vietoris-Rips complex  of $N_X$. 
 A Python implementation of this algorithm is provided in \cite{ourgitcode}.

Since the cliquegram is obtained by a phylogenetic network $N_X:X\times X\to\R$ which is a weighted graph, and thus a one-dimensional structure, the approach introduced in this section is referred to as \emph{first order approach}.

\textbf{Our higher order viewpoint of the phylogenetic reconstruction process}
The notion of a phylogenetic network $N_X$ (see \textit{Defn.~\ref{dfn:phylo-net}}) extends, in a natural way, to the notion of a \emph{filtration} $\F _X$ (see \textit{Defn.~\ref{dfn:diameter}}). 
In what follows, for the purpose of visualizing the evolution of maximal faces in a filtration, filtrations will be represented as \emph{facegrams} (see \textit{Defn.~\ref{dfn:facegram}}) which are structures that generalize cliquegrams and dendrograms (see Rem.~\ref{rem:facegram generlizes cliquegram}). 
A facegram encodes the evolution of maximal faces across a filtration of simplicial complexes. In the special case where the filtration is the Vietoris-Rips filtration of a phylogenetic network $N_X$ over $X$, the associated facegram is exactly the cliquegram of $N_X$.

The process of constructing a facegram from a given filtration turns out to be an equivalence (in particular, a bijection) between filtrations and facegrams, in the following sense:
Let $X$ be a set of taxa. There is an isomorphism between the lattice $\mathbf{Filt}(X)$ of all filtrations over $X$ and the lattice $\mathbf{Facegm}(X)$ of all facegrams over $X$. The maps between them are explicitly defined in \textit{Thm.~\ref{thm:equivalencetwo}}. For an example of the equivalence, consult Fig.~\ref{fig:facegram-filt}; on the left, the facegram built on taxa $\{x,y,z\}$. On the right, is the isomorphic filtration.
\begin{figure}
    \centering
    \includegraphics[width=0.5\textwidth]{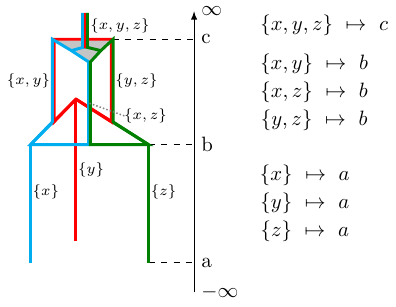}
    \caption[width=\textwidth]{{\bf Facegram and its filtration.} An example of facegram and its associated filtration.}
    \label{fig:facegram-filt}
\end{figure}
By definition, every cliquegram is a particular type of facegram (see Fig.~\ref{fig:facegram}). Hence, in a much broader sense, we can ask whether we can model a more broad phylogenetic reconstruction process as a join operation on the lattice of facegrams, instead of the more restricted setting of the lattice of cliquegrams, see Fig.~\ref{fig:facegram1}.
 \begin{figure}
     \centering
    \includegraphics[width=1\textwidth]{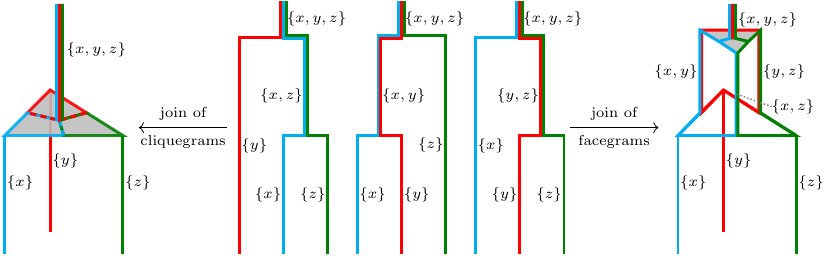}
     \caption{{\bf Comparison of the join of treegrams as cliquegrams and facegrams.} 
     Consider the set of three treegrams in the middle of the picture. The structure on the left is obtained by taking the join operation on those treegrams (viewed as cliquegrams) in the lattice of cliquegrams. On the right, the join of the same treegrams (now viewed as facegrams) in the lattice of facegrams is presented. Note that the cliquegram can be obtained from the facegram on the right by squashing the three loops. Hence, sometimes, the facegram contains reticulation loops
     that are not present in the cliquegram.}
     \label{fig:facegram1}
 \end{figure}
It turns out that in this general setting, where treegrams are viewed as facegrams in place of cliquegrams, the following holds:
\begin{enumerate}
    \item The facegram-join operation yields more information than the cliquegram-join operation, as the cliquegram can be obtained from the facegram (see Fig.~\ref{fig:facegram1}). 
      \item In the setting of facegrams, the PNRP still admits a unique solution, in the sense that for any family $\caF$ of treegrams (viewed as facegrams) over $X$, there exists a unique minimal facegram $\caC_X$ (w.r.t.~the partial order $\leq$ of facegrams) that contain each treegram in $\caF$. Specifically, that facegram $\caF _X$ is given by the join-span $\vee\caF$ (see \textit{Thm.~\ref{thm:filtration-irreducible}}). In fact, \textit{Thm.~\ref{thm:filtration-irreducible}} states that treegrams are precisely the join-irreducible facegrams. 
    \end{enumerate}

\noindent \textbf{Face-Reeb graph and mergegram invariants}
The facegram can become very difficult to visualize when the number of taxa is large. 
Hence, it is important for the applications to study invariants of facegrams that are easy to describe and visualize as well as equipped with a metric structure. This allows comparison of those descriptors (and therefore, indirectly, to allow for comparison of lattice-diagrams) which can be used for statistical inference.

Furthermore, for most machine learning algorithms we need the structure of a Hilbert space and thus a possibility to vectorize the facegram in a sensible way. Due to the weighted-graph nature of the facegram, such a representation should be sensitive to the edges and their weights.
 For that purpose, we propose two novel topological invariants of facegrams (and thus of filtrations too), (i) the face-Reeb graph of a facegram and (ii) the mergegram of a facegram (by extending a construction by Elkin and Kurlin defined on dendrograms \cite{elkin2020mergegram}).

We then show the mergegram is a 1-Lipschitz stable invariant, 
 while the face-Reeb graph is not stable. Our main result is that the mergegram is invariant of weak equivalences of filtrations, a stronger form of homotopy equivalence (see \textit{Thm.~\ref{thm:invariance of mergegram}}). The mergegram 
 can be used as a computable proxy for phylogenetic networks and also, more broadly, for filtrations of datasets, which might also be of interest to TDA.

\textbf{Motivation for introducing cliquegram, facegram and  mergegram}
In this work, we introduce the cliquegram and the facegram in order to obtain an analogous construction to the dendrogram of an ultrametric for the general case of a phylogenetic network and more broadly to a filtration, respectively. Dendrograms are in one-to-one correspondence with ultrametrics. However, dendrograms are useful (i) for visualizing ultrametrics and (ii) for obtaining invariants from ultrametrics, i.e. (a) the $0$-dimensional sublevel set persistence of a metric space $(X,d_X)$ can be obtained by applying the Elder rule on the dendrogram of the single linkage ultrametric of $(X,d_X)$, and (b) the mergegram of a dendrogram \cite{elkin2020mergegram}. In order to obtain analogous constructions to dendrograms for networks and filtrations, we first observed that partitions are a special case of clique-sets and clique-sets are a special case of face-sets over $X$. Combined with the fact that dendrograms are filtered partitions, the preceding argument motivated us to define cliquegrams  as filtered clique-sets and facegrams as filtered face-sets, respectively. Except for generalizing dendrograms, cliquegrams, and facegrams are also useful, since:
\begin{enumerate}
\item[(i)] (Biological motivation) Through those lattices we can model the join-span of a set of treegrams (viewed as cliquegrams or facegrams respectively), and thus obtain a phylogenetic structure that can accommodate loops.
\item[(ii)] (Mathematical motivation) By viewing filtrations as facegrams, we can craft new invariants of filtrations by crafting invariants of facegrams, such as the face-Reeb graph and the mergegram, those are invariants of filtrations which we couldn't obtain directly from the filtration but only through its facegram.
\end{enumerate}

\textbf{Computational complexity and experiments}
In the last section, we give polynomial upper bounds on the computational complexity of the phylogenetic network reconstruction process that is modeled as the join-facegram of a set of treegrams (see \textit{Thm.~\ref{thm:complexity}}). The same complexity upper bound applies when computing the mergegram invariant of the join-facegram. We provide implementations of our algorithms in Python. Next, we apply our algorithms to certain artificial datasets and to a certain benchmark biological dataset.

\subsection{Relation of our work with other research works}

\vspace{1em}
\textbf{Applying TDA for the analysis of phylogenetic networks.} 
Phylogenetic trees and networks have recently been injected to TDA. 
Most TDA approaches for analyzing phylogenetic trees/networks take either the approach of considering vertices as points in high-dimensional space (and use the distance in this space), e.g.~\cites{rabadan2019topological} or use the tree distances directly to construct Vietoris-Rips complexes to capture connected components as well as cycles in the networks for the analysis of the data, e.g.~\cites{TDA_topology_camara, TDA_topology_viral}. 
Other connections of phylogenetic networks with other topological constructions like Merge Trees \cites{munch2019ell, gasparovic2019intrinsic, yan2019structural}, Reeb graphs \cite{stefanou2020tree}, or use other graph representations from which topological information is extracted \cite{lesnick2020quantifying}.
In this work, on the other hand, a new diagrammatic model is constructed for phylogenetic networks, namely the facegram, and mathematical properties of this representation are established in terms of Nakhleh's phylogenetic reconstruction problem. Subsequent analysis is done on this model structure in terms of its mergegram invariant.

\textbf{Relation of facegrams with dendrograms and treegrams}
For a given partition $P_X$ of the set $X$, to every block in the partition corresponds a total complex, and thus, $P_X$ gives rise to a simplicial complex, given by the disjoint union of these total complexes. Hence, dendrograms (and treegrams) can be viewed as filtered complexes or equivalently filtrations. The corresponding facegram of that filtered complex is our dendrogram (resp.~treegram). However, both the dendrogram and the treegram model do not allow recombinations that would introduce loops to our tree of life.

\textbf{Relation of facegrams with formigrams}
 An extension of treegrams which allows for the formation of loops is the \emph{formigram} \cites{kim2018formigrams,kim2023interleaving}. 
By a formigram over $X$, here we mean a filtered subpartition $\S_X:\Int\to\mathbf{SubPart(X)}$ over the poset of closed intervals of $\R$, $\Int$. 
It should be noted that the PNRP also admits a solution in the lattice of formigrams, if we consider the join-operation on a given family of treegrams. However, the join of a set of treegrams (when viewed as formigrams) would still be a treegram, so no recombination events can be recorded/represented in this join-structure.
Facegrams, on the other hand, do accommodate recombinations (loops) when realized as joins of treegrams while maintaining an $\R$-filtration model for our tree of life. This point of view provides a generalization of the concept of dendrograms and treegrams that accommodates loops and at the same time it links facegrams to the machinery of filtrations from TDA (i.e.~$\R$-filtrations). 
Finally, we note that although the formigram is not the same as the facegram, these structures are embedded into a common generalized construction: Firstly, we observe that the notion of facegrams indexed over the poset of $\mathbb{R}$,  generalizes over an arbitrary poset $\mathbf{P}$. For the case of the poset  $\mathbf{P}=\mathbf{Int}$, given that the lattice of face-sets contains the lattice of  subpartitions, facegrams over the poset $\mathbf{Int}$, would be a construction that generalizes both the formigram model and our model of facegrams over the poset $\mathbb{R}$. 

\textbf{Relation to join-decompositions of poset maps} 
In this work, we also utilize the machinery of join-decompositions of order-preserving maps and their associated interleaving distances as developed in \cite{kim2023interleaving}. Specifically, we consider maps from the poset $(\R,\leq)$ to the lattice of face-sets which includes the lattice of subpartitions $\mathbf{SubPart}(X)$ (see \textit{Defn.~\ref{dfn:facegram}}), and we establish certain decompositions of our structures into indecomposables (see \textit{Thm.~\ref{thm:filtration-irreducible}}, \textit{Cor.~\ref{cor:filtration-irreducible}} and \textit{Prop.~\ref{prop:clique-dense}}, \textit{Cor.~\ref{cor:cliq-irred}}).

\textbf{Relation of mergegrams with other invariants}
Formigrams can be viewed as functors from $\mathbf{Int}$ to $\mathbf{SubPart}$, see \cite{kim2023interleaving}. In \cite{kim2024extracting} the authors  introduced the \emph{maximal group diagram} and the \emph{persistence clustergram}. In our work, we are studying the facegram construction as a functor from $\mathbb{R}$ to $\mathbf{Face}$ (which is different from the formigram functor) and we introduce three invariants of that construction, the face-Reeb graph, the mergegram, and the labeled mergegram. Note that dendrograms (as well as treegrams) can be viewed as facegrams and also as formigrams at the same  time. However, even in that case, our invariants are different from the ones used in \cite{kim2024extracting}. Indeed, suppose one has a dendrogram with two leaves $\{x,y\}$ merging at  some time instance $t$. Then, the maximal group diagram of the dendrogram will contain three  intervals $[0,\infty)$, $[0,\infty)$, $[t,\infty)$, corresponding to $\{x\}$, $\{y\}$, $\{x,y\}$, respectively, the persistent  clustergram will contain the three intervals $[0,t)$, $[0,t)$, $[0,\infty)$, corresponding to $\{x\}$, $\{y\}$, $\{x,y\}$,  respectively, whereas the labeled mergegram will contain the intervals $[0,t)$, $[0,t)$, $[t\infty)$,  corresponding to $\{x\}$, $\{y\}$, $\{x,y\}$, respectively. Although those invariants are different, the formalism is quite similar. It seems possible that one can define invariants for facegrams over the poset $\Int$ that may generalize all those invariants, but this topic was out of scope of this paper, so we didn't develop such ideas further.

\textbf{Appendix}
In the appendix we go into some more details for some concepts we just mention briefly in the main part. Namely, we showcase the underlying graph structure of the tree-, clique and facegram models, explain the data structure for their computation as well some more algorithmic and computational connections. Furthermore, we relegate some of the proofs of the next sections to the appendix.

\begin{figure}
\includegraphics[width=0.95\textwidth]{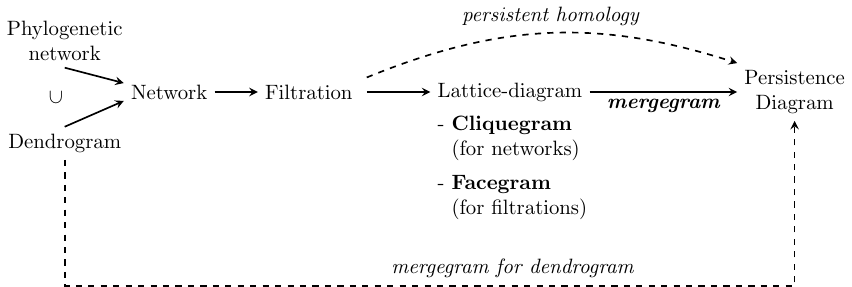}
\caption{{\bf Concepts overview.}
Overview of concepts examined and their relations to each other. Those in bold font are introduced in this work.}
\label{fig:total_overview}
\end{figure}

\section{Diagrammatic models}
In this chapter we recall certain results from the theory of lattices \cite{kim2023interleaving}, then we develop two different lattice-diagram models for visualizing phylogenetic reconstruction networks and more broadly filtrations, the \emph{cliquegram} and the \emph{facegram}, respectively, and we show that in each of those settings, the PNRP admits a unique solution.
The facegram, except being a model for phylogenetic networks, more broadly, it is also related to the notion of \emph{filtered simplicial complexes} that are used in applied topology. 

For an overview of the upcoming concepts and their relations as lattices, see \textit{Tab.~\ref{tab:correspondence}}.
\begin{table}[ht]
\centering
\resizebox{\textwidth}{!}{
\begin{tabular}{lrlll}
\multicolumn{2}{c}{\textbf{Lattice of ...}}
    && \textbf{Lattice of ...} & \textbf{References}
    \\[-0.2em]
\hline
\vspace{-1em}\\
\multicolumn{2}{l}{Dendrograms, $\mathbf{Dendgn}(X)$} 
    & $\rightleftarrows$
    & ultrametrics on $X$
    & {\tiny\itshape (Def.~\ref{dfn:partition})}, {\tiny\itshape (Ex.~\ref{ex:symmetricultranetwork})}, {\tiny\itshape \cite{carlsson2010characterization}}
    \\
\multirow{1}{*}{\rotatebox[origin=t]{270}{$\lhook\joinrel\rightarrow$}}
    & Partitions, $\mathbf{Part}(X)$
    & $\rightleftarrows$
    & equivalence relations on $X$
    & {\tiny\itshape (Def.~\ref{dfn:partition})}, {\tiny\itshape\cite{gratzer2011lattice}}
    \\[0.2cm]
\multicolumn{2}{l}{Treegrams, $\mathbf{Treegm}(X)$ }
    & $\rightleftarrows$ 
    & symmetric ultranetworks on $X$
    & {\tiny\itshape (Def.~\ref{dfn:partition})}, {\tiny \cite{smith2016hierarchical}}, {\tiny\itshape (Ex.~\ref{ex:common-distances})}
\\
\multirow{1}{*}{\rotatebox[origin=t]{270}{$\lhook\joinrel\rightarrow$}}
 & Subpartitions, $\mathbf{SubPart}(X)$ 
 & $\rightleftarrows$ 
 & (sub-)equivalence relations on $X$
 & {\tiny\itshape (Def.~\ref{dfn:partition})}, {\tiny \cite{smith2016hierarchical}}, {\tiny\itshape\cite{gratzer2011lattice}}
    \\[0.2cm]
\multicolumn{2}{l}{Cliquegrams, $\mathbf{Cliqgm}(X)$} 
    & $\rightleftarrows$ 
    & phylogenetic networks $\mathbf{Net}(X)$
    & {\tiny\itshape (Def.~\ref{dfn:cliquegram})}, {\tiny\itshape (Prop.~\ref{prop:equivalence})}, {\tiny\itshape (Def.~\ref{dfn:phylo-net})}
    \\
\multirow{1}{*}{\rotatebox[origin=t]{270}{$\lhook\joinrel\rightarrow$}} 
    & clique-sets $\mathbf{Cliq}(X)$ 
    & $\rightleftarrows$
    & all graphs $\mathbf{Grph}(X)$ 
    & {\tiny\itshape (Def.~\ref{defn:clique-set})}, {\tiny\itshape (Prop.~\ref{prop:cliq-clut-iso})}, {\tiny\itshape (Prop.~\ref{prop:cliq-clut-iso})}
    \\[0.2cm]
\multicolumn{2}{l}{Facegram, $\mathbf{Facegm}(X)$ }  
    & $\rightleftarrows$ 
    & filtrations $\mathbf{Filt}(X)$ 
    & {\tiny\itshape (Def.~\ref{dfn:facegram})}, {\tiny\itshape (Prop.~\ref{thm:equivalencetwo})}, {\tiny\itshape (Def.~\ref{dfn:filtration})}
    \\
&  face-sets, $\mathbf{Face}(X)$ 
    & $\rightleftarrows$ 
    & simplicial complexes $\mathbf{Simp}(X)$ 
    & {\tiny\itshape (Def.~\ref{dfn:face-set})}, {\tiny\itshape (Prop.~\ref{prop:clut-simp})}, {\tiny\itshape (Prop.~\ref{prop:clut-simp})}
    \\[0.5cm]
\multicolumn{5}{c}{\bfseries Connections for Filtrations}\\[-0.2em]
\hline
\vspace{-1em}\\
\multicolumn{5}{r}{%
    tree filtration (\textit{\small Def.~\ref{dfn:diameter}})
    \hfill$\lhook\joinrel\longrightarrow$\hfill
    vietoris-rips filtration (\textit{\small Def.~\ref{dfn:diameter}})
    \hfill$\lhook\joinrel\longrightarrow$\hfill
    filtration (\textit{\small Def.~\ref{dfn:filtration}})
}\\
\multicolumn{5}{r}{%
    (symmetric ultranetwork)
    \hfill
    (phylogenetic network)
    \hfill
    \phantom{filtration (\textit{\small Def.~\ref{dfn:filtration}})}
}
\end{tabular}
}
\caption{Correspondences listed in the papers and their definitions with references where to find them. In general, we treat the different concepts as lattices and show isomorphisms there. Due to the inclusions for filtrations it direclty follows that treegrams are cliquegrams and cliquegrams are facegrams but in general not vice versa. 
}
\label{tab:correspondence}
\end{table}

\subsection{Lattice structures} 
In this section, we review the basic notions and results of the theory of \emph{lattices} \cites{erne2003posets,roman2008lattices} as well as some results from \cite{kim2023interleaving} in order to help us properly formulate the PNRP and explain why it always admits a solution in the lattice setting. Finally, we consider the notion of \emph{diagram over a lattice} from \cite{kim2023interleaving}, which we refer to simply as the \emph{lattice-diagram}.

A \emph{lattice} $\mathcal{L}=(\mathcal{L},\leq,\vee,\wedge)$ is a poset 
$(\mathcal{L},\leq)$ (a binary relation $\leq$ on $\mathcal{L}$ which is reflexive, transitive, and antisymmetric) that admits all finite, nonempty joins $\vee$ and meets $\wedge$. For every $X,Y$ in $\mathcal{L}$, we say that \emph{$Y$ is a join of $X$ and $Z$} if $Y=X\vee Z$, then $X, Z$ are called \emph{joinands (join summands)} of $Y$. 
For every $X,Y$ in $\mathcal{L}$, we denote by $[X,Y]:=\{Z\in \mathcal{L}\mid X\leq Z\leq Y\}$ the interval between $X$ and $Y$ in the lattice $\mathcal{L}$. 
A lattice that admits all joins and meets is called a \emph{complete lattice} and is denoted by $(\mathcal{L},\leq,\vee,\wedge,0,1)$ (where $0:=\wedge \mathcal{L}$ is the bottom element and $1:=\vee \mathcal{L}$ is the top element). Any finite lattice is complete, since it admits a least upper bound, the meet of all the elements, and greatest lower bound, the join of all the elements, which are both finite and thus exists. 
A subset $\mathcal{B}\subset \mathcal{L}$ (possibly infinite) of a lattice $\mathcal{L}$ is said to be a \emph{join-dense subset of $\mathcal{L}$} if every $X$ in $\mathcal{L}$ is the join of a finite number of elements in $\mathcal{B}$. 
 \begin{example}
\label{ex:irred in the finite case are join-dense}
     An element $X$ of $\mathcal{L}$ is \emph{join-irreducible} if it is not a bottom element, and, whenever $X=Y\vee Z$, then $X=Y$ or $X=Z$. If the lattice $\mathcal{L}$ is finite, then the join-irreducible elements of $\mathcal{L}$ are join-dense.
 \end{example}
\begin{proposition}[{\cite{kim2023interleaving}*{Prop 3.2}}]
\label{prop:join-dec}
    Let $\mathcal{L}$ be a lattice that admits intital element $0$ and let $X$ in $\mathcal{L}$. Then, $X=\bigvee (\mathcal{B}\cap [0,X])$.
\end{proposition}
\begin{definition}
\label{dfn:spanning}
We will call the (maximal) elements of the set $\mathcal{B}(X):=\mathcal{B}\cap [0,X]$ 
\emph{the (maximal) spanning $\mathcal{B}$-elements of $X$}.
\end{definition}
\begin{example}
Let $X$ be a finite set. Let $\mathcal{G}_X=(\mathcal{G}_X,\subset,\cup,\cap)$ be the lattice of all connected simple graphs $(V,E)$ with vertex set $V=X$. Then, the set $\caT _X$ of all tree graphs with vertex set $X$ is a join-dense subset of $\mathcal{G}_X$. In particular, for any given graph $(X,E)$ in $\mathcal{G}_X$ the spanning $\caT _X$-elements of $(X,E)$ are exactly the spanning trees of $(X,E)$. 
\end{example}

\begin{proposition}[Minimal join-reconstruction in a complete lattice setting]
\label{prop:inverse}\hfill~
Let $\mathcal{L}=(\mathcal{L}, \leq, \vee, \wedge,0,1)$ be a complete lattice and let $\mathcal{B}$ be a join-dense subset of $\mathcal{L}$. 
Any element $X$ in $\mathcal{L}$ is equal to the join $\vee \mathcal{B}(X)$  
 of its spanning $\mathcal{B}$-elements.
In addition, for any family $\caF$ of elements in $\mathcal{B}$, there exists a unique $\leq$-minimal element $X$ in $\mathcal{L}$ such that $\caF\subset \mathcal{B}(X)$, and moreover, that unique such element $X$ is simply given by the join $\vee\caF$.
\end{proposition}
\begin{proof} The first part of the proof is equivalent to \textit{Prop.~\ref{prop:join-dec}}. 
Now, for the second part, consider 
$$X:=\bigwedge\{Y\in \mathcal{L}\mid \caF\subset \mathcal{B}(Y)\}$$
This is the unique $\leq$-minimal element $X$ such that $\caF\subset \mathcal{B}(X)$, which exists since the lattice $\mathcal{L}$ is complete. Now, we claim that $X=\vee \caF$. 
\begin{itemize}
\item Note that $X$ is a non-empty meet: for $Y:=\vee \caF$, $\caF\subset \mathcal{B}(Y)$ since $\mathcal{B}(Y)=\mathcal{B}\cap [0, Y]\supset \caF$.
\item By the subsequent argument $\vee \caF$ lies in the set that meet-spans $X$. Since the meet $X$ is always less or equal to all the elements in that set, then $X\leq \vee \caF$. 
\item We claim that $\vee \caF\leq X$: Let $Y\in \mathcal{L}$ such that $\caF\subset \mathcal{B}(Y)$ and let $Z\in \caF$. It suffices to check that $Z\leq Y$. Indeed, $Z\in \caF\Rightarrow Z\in \mathcal{B}(Y)$. Thus, $Z\leq \vee \mathcal{B}(Y)=Y$.
\item Finally, by combining all the previous arguments, we obtain that the element $X=\vee F$ is the unique $\leq$-minimal element in $\mathcal{L}$ such that $\caF\subset \mathcal{B}(X)$.
\end{itemize}
\end{proof}

\begin{remark}
\label{rem:PNRP}
If we model properly phylogenetic networks over a fixed finite set of taxa as a finite lattice (and thus complete lattice in particular) and treegrams over that set of taxa as a join-dense subset of that lattice, then by \textit{Prop.~\ref{prop:inverse}} the PNRP will admit a unique solution in this lattice. In the next section we craft two different lattice-diagram models for phylogenetic networks, for each of which the collection of all treegrams is a join-dense subset, and thus in each of these lattice-diagram models, the PNRP admits a solution in particular.
\end{remark}

\textbf{Diagrams over lattices}
Let $\mathcal{L}=(\mathcal{L},\leq, \vee,\wedge, 0, 1)$ be a finite lattice, which we fix here. 
An $\leq$-filtered diagram $X^\bullet:=\ldots \leq X^i\leq X^{i+1} \leq \ldots$
of elements in $\mathcal{L}$ is called a \emph{$\mathcal{L}$-diagram}. Equivalently, an $\mathcal{L}$-diagram is an order-preserving map of posets $X^\bullet:(\R,\leq)\to (\mathcal{L},\leq)$. 
The collection of all $\mathcal{L}$-diagrams forms a lattice on its own. The partial order is given by $X^\bullet\leq Y^\bullet\Leftrightarrow X^t\leq Y^t$, for all $t\in \R$.

\begin{definition}[Dendrograms as lattice-diagrams; partitions of $X$]
\label{dfn:partition}\label{dfn:dendrogram}
    Let $X$ be a finite set. 
    \begin{itemize}
        \item A \emph{partition of $X$} consists of a collection $P_X=\{B_1,\ldots,B_n\}$ of subsets $B_i$ of $X$, $i=1,\ldots,n$, called \emph{blocks}, such that $\cup_{i=1}^{n} B_i=X$ and $B_i\cap B_j=\emptyset$, for all $i\neq j$. Let $\mathbf{Part}(X)$ be the lattice of partitions $P_X$ on $X$ where the partial order is given by
    $$P_X\leq P'_X\Leftrightarrow \forall S\in P_X\text{, }\exists S'\in P'_X\text{, such that }S\subset S'.$$
\item A \emph{dendrogram over $X$} 
is a map $\mathcal{P}_X:\R\to\mathbf{Part}(X)$, such that:
\begin{enumerate}
    \item $\mathcal{P}_X$ is an order-preserving map, i.e.~$\mathcal{P}_X(t)\leq \mathcal{P}_X(s)$, for all $t\leq s$, 
    \item there exists a $t_1\in \R$, such that $\mathcal{P}_X(t)=\{X\}$, for $t\geq t_1$, $\mathcal{P}_X(0)=\{\{x\}\mid x\in X\}$, and $\mathcal{P}_X(t)=\emptyset$, for $t< 0$.
\end{enumerate}  
\item We denote by $\mathbf{Dendgm}(X)$ the lattice of all dendrograms over $X$, whose partial order, meets and joins are given by point-wise inequalities, meets, and joins, respectively.
\end{itemize}
\end{definition}
\begin{definition}[Treegrams as lattice-diagrams; subpartitions of $X$]
\label{dfn:subpartition}\label{dfn:treegram}
    Let $X$ be a finite set. 
    \begin{itemize}
        \item A \emph{subpartition of $X$} consists of a collection $P_X=\{B_1,\ldots,B_n\}$ of subsets $B_i$ of $X$, $i=1,\ldots,n$, called \emph{blocks}, such that $B_i\cap B_j=\emptyset$, for all $i\neq j$. Let $\mathbf{SubPart}(X)$ be the lattice of subpartitions $P_X$ on $X$ where the partial order is given by
    $$P_X\leq P'_X\Leftrightarrow \forall S\in P_X\text{, }\exists S'\in P'_X\text{, such that }S\subset S'.$$
\item A \emph{treegram over $X$} 
is a map $\mathcal{P}_X:\R\to\mathbf{SubPart}(X)$, such that:
\begin{enumerate}
    \item $\mathcal{P}_X$ is an order-preserving map, i.e.~$\mathcal{P}_X(t)\leq \mathcal{P}_X(s)$, for all $t\leq s$, 
    \item there exists a $t_0,t_1\in \R$, such that $\caC _X(t)=\{X\}$, for $t\geq t_1$, and $\caC _X(t)=\emptyset$, for $t\leq t_0$, 
    
\end{enumerate}  
\item We denote by $\mathbf{Treegm}(X)$ the lattice of all treegrams over $X$, whose partial order, meets and joins are given by point-wise inequalities, meets, and joins, respectively.
\end{itemize}
\end{definition}
Note that because we are working over a finite set $X$, our definition of dendrogram and treegram are equivalent to the ones in  \cite{carlsson2010characterization} and \cite{smith2016hierarchical}, respectively.

\begin{example}\label{ex:cliquegramNotTreegram}~\\[-1\baselineskip]
\begin{enumerate}
\item For the taxa $X=\{ a, b, c\}$ a dendrogram of $X$ is given by
\[
P_X^\bullet(x) = 
\begin{cases}
\emptyset & \text{for~} x<0 \\
\{ \{ a\}, \{ b \}, \{ c \}\} = P_0  & \text{for~} x \in [0,2)\\
\{ \{ a, b \}, \{ c \}\} = P_1  & \text{for~} x \in [2,3) \\
\{ \{ a, b,c \}\} = X = P_2  & \text{for~} x \geq 3 
\end{cases}\]

where $P_0 \leq P_1 \leq P_2$ is the partial order given in $\mathbf{Part}(X)$. 

\item Changing the previous map such that $P_X^\bullet(x) = \emptyset$ for $x<-1$ and $P_X^\bullet(x)=P_0$ for $x\in [-1, 2)$, is not a dendrogram but a valid treegram.
\item Changing the previous map such that $P_X^\bullet(x) = \{ \{ a,b\}, \{a,c\} \}$ for $x\in [2,3)$, cannot be modelled as a treegram, since $\{ \{ a,b\}, \{a,c\} \}$ is not a (sub-)partition.
\end{enumerate}
\end{example}

\subsection{The cliquegram of a phylogenetic network}
In this section, for a set of taxa $X$, we will start from a \emph{phylogenetic network} $N_X:X\times X\to\R$ (see \textit{Defn.~\ref{dfn:phylo-net}}) and then associate to it a 
lattice-diagram $\caC _X$ using the maximal cliques of the network called the \emph{cliquegram 
of $N_X$} (see \textit{Defn.~\ref{dfn:cliquegram}}). 
The assignment $N_X\mapsto \caC _X$ lifts to a one-to-one correspondence between phylogenetic networks and cliquegrams, which generalizes the one-to-one correspondence between symmetric ultranetworks and treegrams \cite{smith2016hierarchical} and the one-to-one correspondence between ultrametrics and dendrograms \cite{carlsson2010characterization}. 
Furthermore, we devise an algorithm for reconstructing a cliquegram from a given set of phylogenetic trees. An implementation of this algorithm can be found in \cite{ourgitcode}.

\textbf{Phylogenetic networks}
In the field of \emph{mathematical phylogenetics}, we are interested in studying evolutionary relationships among biological entities. 
Each biological entity is called a \emph{taxon}.
The set of all taxa is denoted by $X$ and let us fix it from now on.
Relationships among the taxa of $X$ are modeled by a weighted DAG on $X$
called a \emph{phylogenetic network}.
More concretely, a phylogenetic network $N_X$ over $X$ can be identified with a real-valued matrix $N_X$ over $X$, namely, a map $N_X:X\times X\to\R$: By viewing the poset $(\R,\leq)$ as time traced backwards, each value $N_X(x,x')$, $x,x'\in X$, of the matrix represents the time, traced backward, when the taxa $x,x'$ were mutated from their \emph{most recent common ancestor}, and it is called \emph{time to coalescence of $x$ and $x'$}. If $x=x'$, the interpretation of $N_X(x,x)$ is that it is the time, traced backwards, when the taxon $x$ was first observed in the phylogenetic network, called \emph{time of observation of $x$}. Because of this biological interpretation for $N_X$ and the consideration that time is traced backward, for any $x,x'$ in $X$: 
\begin{enumerate}
    \item The matrix $N_X$ must be symmetric in $x,x'$
    \item The time $N_X(x,x)$ of observation of $x$ and the time $N_X(x',x')$ of observation of $x'$ must be less or equal to the time $N_X(x,x')$ to coalescence of $x$ and $x'$.
\end{enumerate}
\begin{definition}
\label{dfn:phylo-net}
A \emph{phylogenetic network over $X$} 
is a map $N_X:X\times X\to\R$ such that
\begin{enumerate}
    \item $N_X(x,x')=N_X(x',x)$, for all $x,x'\in X$,
    \item $\max\{N_X(x,x),N_X(x',x')\}\leq N_X(x,x')$, for all $x,x'\in X$.
\end{enumerate}
We denote by $\mathbf{Net}(X)$ the lattice of all phylogenetic networks over $X$, whose partial order, meets and joins are given by entry-wise reverse inequalities, maximums, and minimums, respectively.
\end{definition}

\begin{example}
\label{ex:common-distances}\label{ex:symmetricultranetwork}
Let $X$ be a set of taxa. A matrix $D_X:X\times X\to \R$ such that for all $x,x',x''\in X$:
\begin{enumerate}
     \item $D_X(x,x')=D_X(x',x)$,
     \item $D_X(x,x'')\leq D_X(x,x')+D_X(x',x'')$,
     \item $D_X(x,x')\geq0$,
      \item $D_X(x,x)=0$,
     \item $D_X(x,x')=0\Rightarrow x=x'$,
\end{enumerate}
     is said to be a \emph{metric}. If the matrix is satisfying only the first four, then it is called a \emph{pseudometric}. Every pseudometric is also a phylogenetic network, but not vice versa.
     If $D_X$ satisfies the first condition and also \sloppy$D_X(x,x'')\leq \allowbreak \max\{D_X(x,x'),D_X(x',x'')\}$, then $D_X$ is said to be a \emph{symmetric ultranetwork} \cite{smith2016hierarchical}. 
     If it is also a metric, then it is called an \emph{ultrametric}.
We easily check that in all cases the matrices satisfy the two axioms of a phylogenetic network. From now on, we will also call any symmetric ultranetwork a \emph{phylogenetic tree}.
\end{example}

\noindent We are now investigating a lattice-diagram
representation for phylogenetic networks, visualized by a certain DAG, which generalizes the treegram representation of symmetric ultranetworks; a richer structure that allows for undirected loops in that DAG (note that there will not be any directed loops in that DAG structure, but rather, there would be loops in the underlying graph of that DAG).

Recall that a treegram consists of an increasing sequence of subpartitions of $X$. Now, to allow our representation to have loops we introduce a generalization of the notion of subpartition of a set $X$, i.e.~a family of `blocks' (subsets) of $X$, such that 
\begin{enumerate}
    \item At each level, each block is maximal in the sense that no block is contained in another block.
    \item Intersections between the blocks of $X$ are allowed so that loops can be formed in the network, as time is traced backward.
\end{enumerate}

If we think of a subpartition of $X$ as an equivalence relation $R$ on a subset $X_R$ of $X$, then condition 2. above, simply proposes dropping the "transitivity property" on $R$, and thus it prompts for considering only symmetric and reflexive relations $R$ on the subset $X_R$ of $X$ (not necessarily transitive). 
A symmetric and reflexive relation $R$ on $X_R\subset X$ is simply a graph $G_R$ on the vertex subset $X_R$ of $X$. An analog of a `block' in the setting of graphs, is a maximal connected subgraph $B$ of the graph $G_R$, called a \emph{maximal clique of a graph} (note that when $R$ is an equivalence relation on $X_R\subset X$, the maximal cliques of the associated graph $G_R$ of $R$ are precisely the equivalence classes of $R$ on $X_R$).

Below, we provide an abstract notion of a set $\caC$ of cliques of $X$, without a reference to a graph.

\begin{definition}\label{defn:clique-set}
 A \emph{clique-set of $X$} is a set $\caC_X$ of subsets of $X$, called \emph{cliques}, such that 
    \begin{enumerate}
    \item[(i)] no subset in $\caC_X$ is contained in another subset in $\caC_X$, and 
    \item[(ii)] 
    every set $X$ of vertices in which all pairs are contained in some clique in $\caC_X$ is also contained in a clique in $\caC _X$, i.e.~: for any $Y\subset X$, if for any $y,y'\in Y$, there exists $C_{y,y'}\in \caC _X$ such that $\{y,y'\}\subset C_{y,y'}$, then $Y\subset C$, for some $C\in \caC _X$.
   \end{enumerate}
    
     We denote by $\mathbf{Cliq}(X)$ the lattice of all clique-sets over $X$, whose partial order is given by $\caC _X\leq \caC '_X$ if and only if, for any clique $\C\in \caC _X$, there exists a clique $\C '\in\caC '_X$ such that $\C\subset \C '$.
\end{definition}

\begin{example}\label{ex:facegramNotCliquegram}
\label{ex:subpartition}
Below, we give some (non-)examples of clique-sets.
\begin{enumerate}
   \item If $G_X$ is a graph over $X$, then the set $\C _X$ of all maximal cliques of $G_X$ is a clique-set of $X$.
 \item A subpartition $P_X=\{B_1,\ldots, B_n\}$ of $X$ is a clique-set over $X$. 

\item 
For the taxa set $X=\{ a,b,c\}$, the set $C=\{ \{ a, b\}, \{ a, c\}, \{ b,c\}\}$ is not a clique-set, since all pairs in $Y = \{ a, b, c\}$ are contained in some clique in $C$ but $Y \not\subset C$.
\end{enumerate}
\end{example}

\begin{remark} Note that the maximal cliques of a graph give rise to a natural one-to-one correspondence between graphs and clique-sets. See \textit{Prop.~\ref{prop:cliq-clut-iso}} in Appendix. 
\end{remark}

\textbf{Cliquegrams} 
We show that a phylogenetic network can be represented by a diagram of clique-sets, called the \emph{cliquegram} (a portmanteau of `clique-set-diagram'), a representation which generalizes the treegram representation of symmetric ultranetworks \cite{smith2016hierarchical}.
  Let $N_X$ be a phylogenetic network over a set of taxa $X$. For any $t\in\R$, we can define a graph $G_t^X=(V_t^X,E_t^X)$, given by 
\begin{itemize}
    \item $V_t^X:=\{x\in X\mid N_X(x,x)\leq t\}$,
    \item $E_t^X:=\left\{\{x,x'\}\subset X\mid N_X(x,x')\leq t\text{ and }x\neq x'\right\}.$
\end{itemize}
\begin{remark}
Note that the graph $G_t^X=(V_t^X,E_t^X)$ is well-defined since $N_X$ satisfies the second assumption in \textit{Defn.~\ref{dfn:phylo-net}}. $G_t^X$ is well-defined only for phylogenetic networks. Indeed, if $N_X$ is not a phylogenetic network, but rather, a symmetric network that does not satisfy the second assumption in \textit{Defn.~\ref{dfn:phylo-net}}, then $G_t^X$ will contain an edge that does not contain both of its incident vertices.
\end{remark}

Let $\caC _X(t)$ denote the set of all maximal cliques of $G_t^X$.
We call the map $\caC _X:\R\to\mathbf{Cliq}(X)$, $t\mapsto \caC _X(t)$, the \emph{cliquegram representation of $N_X$}.

It is easy to check that (i) $\caC _X$ forms an order-preserving map $\caC _X:(\R,\leq)\to\mathbf{Cliq}(X)$, and (ii) there exists a $t_0\in \R$ such that $\caC _X(t)=\{\{X\}\}$, for $t\geq t_0$, and $\caC _X(t)=\emptyset$, for $t\leq-t_0$.
\begin{remark}
Note that, if $N_X$ is a metric space, then for any time $t\in\R$, the clique-set $\caC _X(t)$ is equal to the set of all maximal simplices of the Vietoris-Rips complex of $N_X$ at time $t$. 
Also, note that if $N_X$ is a symmetric ultranetwork or an ultrametric,
then the associated cliquegram $\caC _X$ would be a treegram or a dendrogram, respectively. 

Hence, the cliquegram representation of phylogenetic networks generalizes the treegram representation of a symmetric ultranetwork \cite{smith2016hierarchical} and the dendrogram representation of an ultrametric \cite{carlsson2010characterization}, respectively.  
\end{remark}

\begin{algorithm}[ht]
\caption{The cliquegram of a phylogenetic network}\label{alg:cliquegram}
\label{alg:cliquegramfromnetwork}
\begin{algorithmic}[1]
\State \textbf{Input: }A phylogenetic network $N_X$. Order all the different values of the coefficients $N_X(x,x')$ of the distance matrix in an increasing way. Let $M_{N_X}$ be that ordered set of real numbers. 
\State $CL \leftarrow \emptyset$
\ForAll{$t \in M_{N_X}$ taken in an increasing order}
    \State $G_t$ $\leftarrow$ Build the graph $G_t$ with vertex set $V(G_t):=X$ \\
    \phantom{ $G_t$ $\leftarrow$ }and edge-set $E(G_t):=\left\{(x,x')\in X\times X\mid N_X(x,x')\leq t \right\}.$
    \State $CL \leftarrow CL \cup \{t; $ \texttt{FindMaxCliques}($G_t$) $\}$
\EndFor
\State{Return $CL$ (A nested sequence of clique-sets of $X$, i.e.~the cliquegram associated to $D$.)}
\end{algorithmic}
\end{algorithm}

Finding the maximal cliques \texttt{FindMaxCliques} in a graph is implemented using the Bron-Kerbosch algorithm \cite{BronKerbosch73} and subsequent improvements, see e.g. \cite{tomita_maxcliques_2006}. It's a recursive algorithm but there exist iterative implementations avoiding issues with recursion depth. 
Assume we have two successive element $t_i$, $t_{i+1}$ in $M_{N_X}$. 
Given the new edges as $E(G_{t_{i+1}}) \backslash E(G_{t_i})$, all the maximal cliques in $G_{t_{i+1}}$ which are not already in $G_{t_{i}}$ need to contain the two vertices in at least one of the new edges. 
Hence, if there are only few new edges, we can find the new maximal cliques faster by using the version (for each new edge separately) and  restricted to only include the maximal cliques containing the chosen new edge. 
For the runtime, see \textit{Sec.~\ref{sec:mergegram_cliquegram}} for a further discussion as well as the code repository in \cite{ourgitcode} for a comparison.

Next, we give an abstract notion of a cliquegram without a reference to a phylogenetic network.

\begin{definition}
\label{dfn:cliquegram}
Let $X$ be a finite set of taxa. 
\begin{itemize}
\item A \emph{cliquegram over $X$} 
is a map $\caC _X:\R\to\mathbf{Cliq}(X)$, such that:
\begin{enumerate}
    \item $\caC_X$ is an order-preserving map, i.e.~$\caC _X(t)\leq \caC _X(s)$, for all $t\leq s$, 
    \item there exists a $t_0,t_1\in \R$, such that $\caC _X(t)=\{X\}$, for $t\geq t_1$, and $\caC _X(t)=\emptyset$, for $t\leq t_0$, 
    
\end{enumerate}  
\item We denote by $\mathbf{Cliqgm}(X)$ the lattice of all cliquegrams over $X$, whose partial order, meets and joins are given by point-wise inequalities, meets, and joins, respectively.
\end{itemize}
\end{definition} 

\begin{remark}\label{rmk:criticalset}
Since we assume that $X$ is finite, (1) and (2) imply that there exists an ordered set of non-negative real numbers $(a_1<\ldots<a_s)$ and of minimal cardinality, called a \emph{critical set}, such that $\caC _X(t)=\caC _X(s)$, for all $a_i\leq t<s<a_{i+1}$, $i=1,\ldots,s$. 
In particular, when the cliquegram is a treegram, \textit{Defn.~\ref{dfn:cliquegram}} agrees with the definition of treegram in \cite{kim2018formigrams}.
\end{remark}
\begin{remark}By definition, the cliquegram can be visualized by a rooted DAG, as in Fig.~\ref{fig:cliquegram3}.
\end{remark}

The cliquegram representation of phylogenetic networks yields an equivalence between cliquegrams and phylogenetic networks. See Fig.~\ref{fig:cliquegram3}. This equivalence guarantees that all the information encoded in the phylogenetic networks is also encoded in the cliquegram and vice versa.

\begin{proposition}
\label{prop:equivalence}
Let $X$ be a set of taxa. There is an isomorphism
 $$\mathbf{Net}(X)\stackrel[\Psi]{\Phi}{\rightleftarrows}\mathbf{Cliqgm}(X)$$
 between the lattice $\mathbf{Net}(X)$ of all phylogenetic networks over $X$ and the lattice $\mathbf{Cliqgm}(X)$ of all cliquegrams over $X$. 
 The isomorphism maps $\Phi,\Psi$ are defined as follows
 \begin{itemize}
     \item For any phylogenetic network $N_X$, the cliquegram $\Phi(N_X)$ over $X$ is given for any $t\in\R$ by 
$$\Phi(N_X)(t):=\bigvee\left\{\{x,x'\}\subset X\mid N_X(x,x')\leq t\right\}.$$
     \item For any cliquegram $\caC _X$, the phylogenetic network $\Psi(\caC _X)$ is given for any $x,x'\in X$ by
$$\Psi(\caC _X)(x,x'):=\min\{t\in\R\mid x, x'\text{
 belong to some clique in }\caC _X(t)\}.$$
 \end{itemize}
\end{proposition}
\begin{example} For the taxa set $X=\{ a,b,c \}$, a cliquegram is given by the assignment 
\[\caC(x)=
\begin{cases}
    \emptyset & \text{for~} x < 0 \\
    \{ \{ a\}, \{ b\}, \{ c\}\} & \text{for~} x \in [0, 2) \\
    \{ \{ a, b\}, \{ a, c\}\} & \text{for~} x \in [2, 3) \\
    \{ \{ a, b, c\}\} & \text{for~} x \geq 3
\end{cases}.\]
By \textit{Prop.~\ref{prop:equivalence}} this cliquegram is isomorphic to the phylogenetic network given by phylogenetic network
\begin{center}
\begin{tabular}{c||c|c|c}
$N_X$& a & b & c \\
\hline
\hline
a & 0 & 2 & 2 \\
b & 2 & 0 & 3 \\
c & 3 & 2 & 0
\end{tabular}
\end{center}
This is the same situation as in \textit{Ex.~\ref{ex:cliquegramNotTreegram}(2)}, which could not be represented as a treegram.
\end{example}

\begin{figure}[ht]
    \centering
    \includegraphics[width=0.7\textwidth]{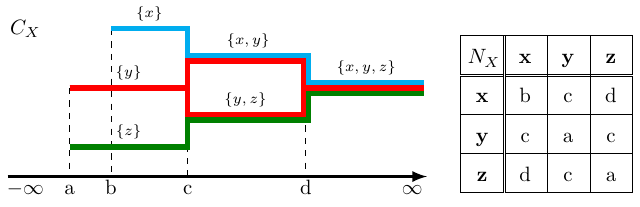}
    \caption{On the left there is a cliquegram $\caC _X$, on the right is the associated phylogenetic network $N_X$.}
    \label{fig:cliquegram3}
\end{figure}

By definition of a cliquegram and Ex.~\ref{ex:subpartition}, it is easy to check that any treegram is also a cliquegram. However, the converse is not true. 
Thus, \textit{Prop.~\ref{prop:equivalence}} can be viewed as a generalization of the equivalence of dendrograms with ultrametrics \cite{carlsson2010characterization}, and of treegrams with symmetric ultranetworks (symmetric ultranetworks) \cite{smith2016hierarchical}.

\begin{remark}
The biological interpretation of \textit{Prop.~\ref{prop:equivalence}} is that there is a phylogenetic network reconstruction process which can be modelled as a lattice-join operation $\vee$ on the collection of cliquegrams over $X$, i.e.~any family $\caT$ of treegrams will span a unique minimal cliquegram, given by the join-span $\vee \caT$. 
\end{remark}

\begin{remark}
It is interesting to note that because of \textit{Prop.~\ref{prop:equivalence}} and the algebraic viewpoint of the minimum operation on $\R$ as a tropical addition on $\R$ (and therefore the entry-wise minimum of matrices), we can realize the join-cliquegram as a tropical addition on treegrams. We recommend \cite{maclagan2021introduction} for a thorough introduction to tropical geometry.
\end{remark}

\begin{remark}[Computing the join-cliquegram of a set of treegrams]\label{remark:join_cliquegram}
Given a family $\caT$ of treegrams, we can compute the join-cliquegram $\vee\caT$ of  $\caT$ by (i) first computing the minimizer network matrix of the associated symmetric ultranetworks of the trees in $\caT$ (this is the entry-wise minimum of these matrices) 
 and then (ii) taking the associated cliquegram of that network, following \textit{Alg.~\ref{alg:cliquegramfromnetwork}}.
\end{remark}

\subsection{The facegram of a filtration}
In this section, we define a higher order generalization of a phylogenetic network, called a \emph{filtration}, which is originated in TDA \cite{memoli2017distance}. We associate a graphical/DAG
representation for a filtration, called a \emph{facegram}, which generalizes the 
cliquegram representation of a phylogenetic network (see Rem.~\ref{rem:facegram generlizes cliquegram}).

\textbf{Filtrations}
Let $X$ be a set of taxa. 
Phylogenetic networks $N_X$ over $X$ are normally considered modeling evolutionary relationships among taxa, by assigning a certain time instance $N_X(x,x')$ to each pair $\{x,x'\}$ of taxa in $X$, called \emph{time to coalescence}, if $x\neq x'$, and \emph{time to observation}, if $x=x'$. 
From a combinatorial (topological) viewpoint, perhaps a more natural extension of these coalescence times to a higher order is the following: Let $S=\{x_1,\ldots,x_n\}$ be a subset of $X$. Consider a function $\F _X:\mathbf{pow}(X)\to \R$ representing \emph{for any $S\subset X$ the time $\F _X(S)$ to coalescence of all the taxa in $S$}, where $\mathbf{pow}(X)$ is the set of all non-empty subsets of $X$. 

 By viewing $(\R,\leq)$ as time traced backward, each value $\F _X(S)$, $S\in \mathbf{pow}(X)$, represents the time, traced backward, when the taxa $S\subset X$ were mutated from their \emph{most recent common ancestor}. Because of this biological interpretation for $\F _X:\mathbf{pow}(X)\to\R$ and the consideration that time is traced backwards, for any $S\subset T\subset X$: the time to coalescence of $S$ must be less or equal than the time to coalescence of $T$, i.e.~$\F _X(S)\leq \F _X(T)$. 

\begin{definition}[{\cite{memoli2017distance}*{Sec.~3}}]
\label{dfn:filtration}
Let $X$ be a finite set of taxa.
\begin{enumerate}
    \item A \emph{filtration over $X$} is an order-preserving map $\F _X:(\mathbf{pow}(X),\subset)\to(\R,\leq)$ of posets.
    \item We denote by $\mathbf{Filt}(X)$ the lattice of filtrations, whose partial order, joins and meets are given by point-wise reverse inequalities, minimums and maximums, respectively.
\end{enumerate}
\end{definition}

Below is an example of a filtration, which is given by extending the pairwise-defined \emph{time to coalescence model} from the setting of phylogenetic networks to the setting of filtrations. 
\begin{definition}
\label{dfn:diameter}
Let $N_X:X\times X\to\R$ be a phylogenetic network over a set of taxa $X$. We define the \emph{Vietoris-Rips} or \emph{diameter filtration $\mathbf{VR}_{X}:\mathbf{pow}(X)\to\R$ of $N_X$} element-wisely for each $\sigma\in\mathbf{pow}(X)$, as follows:
$$\mathbf{VR}_{X}(\sigma):=\max_{x,x'\in \sigma}N_X(x,x').$$
By definition, the map $N_X\mapsto\mathbf{VR}_{X}$ is injective.

We call the diameter filtration $\mathbf{VR}_{X}$ of a symmetric ultranetwork $U_X$, a \emph{tree filtration}.
\end{definition}

\begin{remark}
The lattice $\mathbf{Net}(X)$ embeds into the lattice $\mathbf{Filt}(X)$ via the Vietoris-Rips filtration construction. Thinking of phylogenetic networks as weighted or filtered graphs, and thus one-dimensional structures, and filtrations as filtered simplicial complexes, then a filtration can be seen as a higher order generalization of phylogenetic networks. 
\end{remark} 

\textbf{Facegrams}
Now, we want to devise a representation for phylogenetic filtrations which (i) generalizes the 
cliquegram representation of phylogenetic networks and (ii) that can be used for modeling the phylogenetic filtration reconstruction from a given family of phylogenetic trees, if they are viewed as filtrations. 
From the mathematical-biology point of view, 
the join-operation on filtrations 
is expected to be a much richer structure, capturing more reticulations/recombination cycles, than the join-operation on phylogenetic networks. 
This would be more clear in Ex.~\ref{fig:facegram}. To achieve that, we consider a generalization of a clique-set of a set $X$, simply by dropping the second restriction in the definition of a clique-set.
First, let's recall the notion of a simplicial complex.
\begin{definition}
Let $X$ be a finite set of taxa. A \emph{simplicial complex over $X$} is a collection $\S_X\subset \mathbf{pow}(X)$ of subsets of $X$, called \emph{simplices} or \emph{faces}, which is closed under taking subsets. 
\end{definition}

By definition, a simplicial complex over $X$, $\S_X$, is a downset in the lattice $(\mathbf{pow}(X),\subset,\cup,\cap)$ of non-empty subsets of $X$. Every finite downset of a given poset is completely determined by its maximal elements. In the case of a simplicial complex $\S_X$, those are the maximal faces of $\S_X$.
\begin{definition}
\label{dfn:face-set}
Let $X$ be a finite set of taxa.
\begin{enumerate}
\item A \emph{face-set of $X$}
is a family $\caS_X$ of subsets of $X$ that are pairwise incomparable with respect to the inclusion of sets. The elements of a face-set are called \emph{faces}.
 \item 
     We denote by $\mathbf{Face}(X)$ the lattice of face-sets of $X$, whose inequalities are given by: $\caS_X\leq \caS'_X$ if and only if, for any face $C\in \caS_X$, there exists a face $\C '\in\caS'_X$ such that $C\subset \C '$.
\end{enumerate}
\end{definition}
 \begin{example}
 If $\S_X$ is a simplicial complex over $X$, then the set $\C _X$ of all maximal faces $\sigma$ of $\S_X$ forms a face-set of $X$.
 \end{example}

\begin{remark}
\label{rem:joinstwo}
Note that the maximal faces of a simplicial complex give rise to a natural one-to-one correspondence between simplicial complexes and face-sets. See \textit{Prop.~\ref{prop:clut-simp}} in Appendix. In addition, because of that equivalence, the join $\C _X\vee \C '_X$ of a pair of face-sets of $X$ is equal to the set of maximal faces of the union simplicial complex $\S_X\cup S'_X$ of the simplicial complexes $\S_X$ and $S'_X$ associated to the face-sets $\C _X$ and $\C '_X$, respectively.
\end{remark}

Let $\F _X$ be a filtration over a set of taxa $X$. For any $t\in\R$, we define the simplicial complex $\caS_X(t)$, given by
$$\caS_X(t):=\{\sigma\in\mathbf{pow}(X)\mid \F _X(\sigma)\leq t\}.$$
Let $\caC _X(t)$ denote the face-set of all maximal faces of $\caS_X(t)$.
We call the map $\caC _X:\R\to\mathbf{Face}(X)$, $t\mapsto \caC _X(t)$ the \emph{facegram representation of $\F _X$}. 
It is easy to check that (i) $\caC _X$ forms an order-preserving map $\caC _X:\R\to\mathbf{Face}(X)$, and (ii) there exists a $t_0\in \R$ such that $\caC _X(t)=\{X\}$, for any $t\geq t_0$.

\begin{remark}
\label{rem:facegram generlizes cliquegram}
Let $N_X$ be a phylogenetic network and let $\mathbf{VR}_{X}:\mathbf{pow}(X)\to\R$ be its associated Vietoris-Rips filtration. Then, the facegram representation $\caF _X$ of the filtration $\mathbf{VR}_{X}$ coincides with the cliquegram representation of $N_X$. Therefore, the facegram representation of a filtration is a generalization of the cliquegram representation of a phylogenetic network. 
\end{remark}

Below, we define the abstract notion of a facegram over $X$ without a reference to a filtration. 

\begin{definition}
\label{dfn:facegram}
Let $X$ be a finite set of taxa. 
A \emph{facegram over $X$} 
is a map $\caF _X:\R\to\mathbf{Face}(X)$, such that:
\begin{enumerate}
    \item $\caF _X$ is an order-preserving map, i.e.~$\caF _X(t)\leq \caF _X(s)$, for all $t\leq s$, 
    \item there exists a $t_0, t_1\in \R$, such that $\caF _X(t)=\{X\}$, for any $t\geq t_1$, and 
    $\caF _X(t)=\emptyset$, for $t\leq t_0$,
\end{enumerate}

We denote by $\mathbf{Facegm}(X)$ the lattice of all facegrams over $X$, whose partial order, meets and joins are given by point-wise inequalities, meets, and joins, respectively.
\end{definition} 

\begin{remark}
Similar to Rem.~\ref{rmk:criticalset}, there exists an ordered set of non-negative real numbers $(a_1<\ldots<a_n)$ with minimal cardinality, called a \emph{critical set}, such that $\caC _X(t)=\caC _X(s)$, for all $a_i\leq t<s<a_{i+1}$, $i=1,\ldots,n$.
\end{remark}

\begin{remark}By definition, the facegram can be visualized as in Fig.~\ref{fig:facegram-filt2}.
\end{remark}

\begin{figure}[ht]
    \centering
    \includegraphics[width=0.5\textwidth]{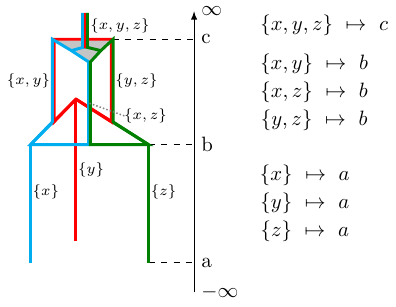}
    \caption[width=\textwidth]{An example of a facegram and its associated filtration.}
    \label{fig:facegram-filt2}
\end{figure}

We show that facegrams and filtrations are naturally equivalent (for the proof, see Prop.~\ref{thm:equivalencetwo}).
\begin{proposition}
\label{thm:equivalencetwo}
Let $X$ be a set of taxa. There is an isomorphism
$$\mathbf{Filt}(X)\stackrel[\Psi]{\Phi}{\rightleftarrows}\mathbf{Facegm}(X)$$
between the lattice $\mathbf{Filt}(X)$ of all filtrations over $X$ and the lattice $\mathbf{Facegm}(X)$ of all facegrams over $X$, where the maps are defined as follows
\begin{itemize}
    \item For any filtration $\F _X$, the facegram $\Phi(\F _X)$ over $X$ is given for any $t\in\R$ by 
$$\Phi(\F _X)(t):=\bigvee\left\{S\subset X\mid \F _X(S)\leq t\right\}.$$
    \item For any facegram $\caF _X$, the filtration $\Psi(\caF _X)$ is given for any $S\subset X$ by
$$\Psi(\caF _X)(S):=\min\{t\in\R\mid S\text{
is contained in some face in }\caF _X(t)\}.$$
\end{itemize}
\end{proposition}

\begin{example}
For a taxa set $X = \{ a, b, c\}$, a facegram over X is given by the map  
\[ \caF _X (x) = 
\begin{cases}
\emptyset & \text{for~} x < 0 \\
\{ \{a\}, \{b\}, \{c\} \} & \text{for~} x \in [0,2) \\
\{\{a,b\}, \{a,c\}, \{b, c\}\} & \text{for~} x \in [2, 3) \\
\{a, b, c\} & \text{for~} x \geq 3, \\
\end{cases}\]
which is isomorphic to the filtration assigning $0$ to the singletons, $2$ to the tuples and $3$ to the triplet in $X$ (see \textit{Prop.~\ref{thm:equivalencetwo}}). 
Furthermore, comparing it to \textit{Ex.~\ref{ex:facegramNotCliquegram}}, we can see we cannot model it as a cliquegram and thus the facegram is a more general construction.
\end{example}

Note also that treegrams over $X$ are special cases of facegrams over $X$. 
In fact, we show that treegrams -when viewed as facegrams- are precisely the join-irreducible facegrams. 
We can now connect the facegram with the PNRP, in particular we specify what the join-irreducible facegrams are.

\begin{theorem}
\label{thm:filtration-irreducible}
A facegram over $X$ is join-irreducible if and only if it is a treegram over $X$. In particular, treegrams are join-dense in the lattice of facegrams. By Rem.~\ref{rem:PNRP}, if we use the facegram as a model of phylogenetic reconstruction, then in this setting the PNRP will admit a solution.
\end{theorem}

\begin{proof}
($\Leftarrow$) Because in every treegram over $X$ all the taxa in the end are joined to form a single block, each $x\in X$ appears as a leaf in the treegram, so we can restrict to dendrograms. Our claim is equivalent to check that dendrograms over $X$ are join-irreducible in the sublattice of face sets with vertex set exactly equal to $X$: Let $\C _X$ be a dendrogram, i.e.~$\C _X:(\R,\leq)\to\mathbf{Part}(X)$. 
To show that claim above, it suffices to check it pointwisely. 
Indeed, at each time $t\in\R$, the face-set $\C _X(t)$ is a partition of $X$ and thus a face-set of vertex set equal to $X$. If the face-set is not join-irreducible in the lattice of face-sets with vertex set $X$, there would be two other face-sets $\C '_X,\C ''_X$ such that $\C _X(t)=\C '_X(t)\vee \C ''_X(t)$ with $\C '_X(t),\C ''_X(t)\neq \C _X(t)$, and each with vertex set equal to $X$. 
That means that if $\S_X,\S'_X,\S''_X$ are the simplicial complexes generated by those face-sets, respectively, then $\S_X=S'_X\cup S''_X$ while both $\S'_X$, $\S''_X$ would be strict subcomplexes of $\S_X$. 
So, the maximal faces of $S'_X$ and the ones for $\S''_X$ are strictly contained in the maximal faces of their union $\S_X$, and so is their union. That implies that $\S'_X\cup \S''_X$ is a strict subcomplex of $\S_X$, which is clearly a contradiction. 
Hence, the partition $\C _X(t)$, viewed as a face-set with vertex set $X$ is indeed join-irreducible in the lattice of face-sets with vertex set equal to $X$.
 
($\Rightarrow$) If $\C _X$ is join-irreducible and $\C _X$ is not a treegram, then there is a minimum value $t_1\in \R$ such that the clique-set $\C _X(t_1)$ is not a subpartition. However, $\C _X(t_1)$ would be a union of finitely many subpartitions. Pick any of them, denoted by $\S_X(t_1)$, and then observe that as $t>t_1$ the blocks of that subpartition are merged more and more together, creating at times $t>t_1$ other (perhaps more than one) choices of subpartitions $\S_X(t)$ that are coarser than $\S_X(t_1)$ until they become a single block $\{X\}$. For each time instance, we make a choice of such a subpartition $\S_X(t)$. By definition of $t_1$, for $t<t_1$ $\C _X(t)=:\S_X(t)$ was a subpartition of $X$, and thus the map $\S_X:t\mapsto \S_X(t)$ is a treegram. Moreover, by definition, the join of all those treegrams $\S_X$ that are constructed is equal to $\C _X$ (where they are finitely many of them since $X$ is finite), and thus (i) $\C _X$ is a finite join of at least two treegrams and (ii) each of those treegrams is different from $\C _X$, where both (i) and (ii) follow by the assumption that $\C _X$ itself is not a treegram. Hence, $\C _X$ is join-reducible, a contradiction. 

Lastly, note that treegrams are also join-dense in the lattice of facegrams. Indeed, if $\C _X$ is a facegram, then there would be a finite set $S$ of critical points. If we take the sublattice of facegram with the same critical set $S$, then since $X$ is finite too, that sublattice would be finite, thus the facegrams would be equal to a join-span of join-irreducible, i.e.~treegrams (see Ex.~\ref{ex:irred in the finite case are join-dense}).
\end{proof} 
In addition, we obtain the following interesting decomposition, connecting filtrations with \emph{tree filtrations} $\T _X$, i.e.~the Vietoris-Rips filtrations of symmetric ultranetworks $U_X$.  

\begin{corollary}
\label{cor:filtration-irreducible}
    Let $\F _X:\mathbf{pow}(X)\to\R$ be a filtration. Because of \textit{Thm.~\ref{thm:filtration-irreducible}} and \textit{Defn.~\ref{dfn:spanning}} $\F _X$ admits the following join decomposition into spanning tree filtrations, i.e.:
    $$\F _X = \min_{\T _X\geq \F _X} \T _X.$$
    This decomposition is not unique among join decompositions of $\F _X$, since in the above decomposition of $\F _X$, one can only consider the $\vee$-maximal spanning elements (tree filtrations $\T _X$) of $\F _X$ in the lattice of filtrations $(\mathbf{Filt}(X),\geq,\min,\max)$ (note that the join of that lattice is the minimum of the filtration values because the poset-order is given by the reverse inequality). 
\end{corollary}

\begin{example}\label{ex:filtration_decomposition}
Given taxa set $X=\{ a, b, c, d\}$, let the following filtration $\F _X$ be given by the  following assignments:
\begin{align*}
\{ a \} & \mapsto 0 & \{ b \} & \mapsto 0 & \{ c \} & \mapsto 0 & \{ d \} & \mapsto 0 \\
\{ a, b \} & \mapsto 1 & \{ b, c \} & \mapsto 1 & \{ c, d \} & \mapsto 1\\
\{ a, b, c \} & \mapsto 2 & \{ b, c, d \} & \mapsto 2 & \{ a, b, c, d \} & \mapsto 3
\end{align*}

By \textit{Defn.~\ref{dfn:diameter}} a tree filtration is the diameter filtration of a symmetric ultranetwork. Let us list the following tree filtrations via their pairwise distances

\begin{center}{\Small
\begin{tabular}{c|cccc}
$U^1_X$ & a & b & c & d \\
\hline
a & 0 & 1 & 2 & 3 \\ 
b & 1 & 0 & 2 & 3 \\
c & 2 & 2 & 0 & 3 \\
d & 3 & 3 & 3 & 0
\end{tabular}
\quad
\begin{tabular}{c|cccc}
$U^2_X$ & a & b & c & d \\
\hline
a & 0 & 3 & 3 & 3 \\ 
b & 3 & 0 & 2 & 2 \\
c & 3 & 2 & 0 & 1 \\
d & 3 & 2 & 1 & 0

\end{tabular}
\quad
\begin{tabular}{c|cccc}
$U^3_X$ & a & b & c & d \\
\hline
a & 0 & 3 & 3 & 3 \\ 
b & 3 & 0 & 1 & 2 \\
c & 3 & 1 & 0 & 2 \\
d & 3 & 2 & 2 & 0
\end{tabular}
\quad
\begin{tabular}{c|cccc}
$U^4_X$ & a & b & c & d \\
\hline
a & 0 & 2 & 2 & 3 \\ 
b & 2 & 0 & 1 & 3 \\
c & 2 & 1 & 0 & 3 \\
d & 3 & 3 & 3 & 0
\end{tabular}
}\end{center}

which correspond to the following assignments, where $\T ^i_X$ is the tree filtration of the symmetric ultranetwork $U_X^i$ for all $i=1,2,3,4$:
\begin{align*}
\T ^i_X(\{ x \}) &= 0 &  \text{for } x \in X && \T ^i_X(\{ a, b, c, d \}) &= 3\\
T^1_X( \{ a, b \})&=1 & T^1_X( \{ a, b, c \})&=2, 
& T^2_X( \{ c,d  \})&=1 & T^2_X( \{ b,c,d \})&=2 \\
T^3_X( \{ b,c \})&=1 & T^3_X( \{ b,c,d \})&=2,
&T^4_X( \{ b,c \})&=1 & T^4_X( \{ a,b,c \})&=2
\end{align*}
Any elements of the power set not listed above for the different tree filtrations are set to 3.

Then we get the following join-decompositions:
\[
F_X 
= T^1_X \vee T^2_X \vee T^3_X \vee T^4_X 
= T^1_X \vee T^2_X \vee T^3_X 
= T^1_X \vee T^2_X \vee T^4_X.
\]

Note that, $F_X \neq T^1_X \vee T^2_X$.
Furthermore, the decomposition in \textit{Cor.~\ref{cor:filtration-irreducible}} is not a minimal decomposition in the sense that we can find fewer trees whose join yields the filtration, as in this example we have three trees whose join give the filtration.
\end{example}

\begin{remark}
\label{rem:epi}
    Any cliquegram is a facegram, but not vice versa. In addition, there is a surjective lattice morphism from facegrams to cliquegrams: for any given facegram $\caC _X:(\R,\leq)\to\mathbf{Face}(X)$ we can construct a cliquegram $\caC '_X:(\R,\leq)\to\mathbf{Cliq}(X)$ as follows: for each $t\in\R$, we define the $\caC '_X(t)$ to be the set of maximal cliques of the $1$-dimensional skeleton (underlying graph) of the simplicial complex generated by the face-set $\caC _X(t)$. See Fig.~\ref{fig:facegram}.
\end{remark}
\begin{figure}[ht]
    \centering
    \includegraphics[width=0.95\textwidth]{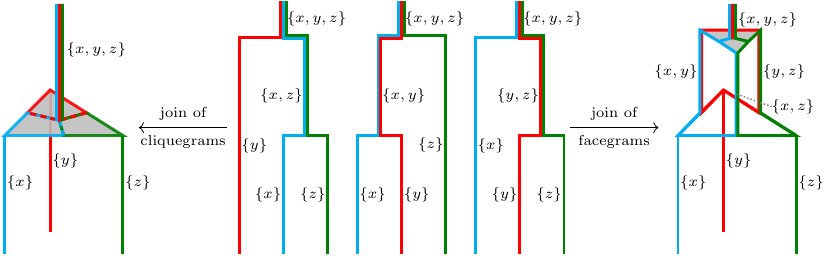}
    \caption{The left shows the join operation on cliquegrams. On the right, there is the join operation on the richer structure of facegrams. The cliquegram can be obtained by the facegram by squashing the three loops. So, the facegram contains reticulation cycles that are not seen by the cliquegram.}
    \label{fig:facegram}
\end{figure}

We also obtain the following results on cliquegrams.
\begin{proposition}
\label{prop:clique-dense}
Treegrams may not be join-irreducible in the lattice of cliquegrams. However,  treegrams are join-dense in the lattice of cliquegrams.  
Thus, by Rem.~\ref{rem:PNRP} if we use the cliquegram as a model of phylogenetic networks, the PNRP still admits a solution in this setting.
\end{proposition}
\begin{proof}
Firstly, to see why treegrams may not be join-irreducible, it suffices to restrict to the static (non-filtered) case of clique-sets. The partition $\{\{x,y,z\}\}$ of $X=\{x,y,z\}$ is a clique-set which is not join-irreducible because one has 
$$\{\{x,y,z\}\}=\{\{x,y\},\{z\}\}\vee \{\{y,z\},\{x\}\}\vee \{\{z,x\},\{y\}\}.$$ 
See also Fig. \ref{fig:facegram}. 

Next, we show that treegrams are join-dense inside the lattice of cliquegrams: Let $\caC _X$ be a cliquegram and let $\S$ be the critical set of $\caC _X$ (the finite set of all time instances $t$ where the clique-sets $\caC _X(t)$ change through time). Now, think of the cliquegram as a facegram. Consider the sublattice of facegrams with critical set $\S$. Then, this sublattice is finite (since $\S$ and $X$ are finite). Thus, as we have seen in Ex.~\ref{ex:irred in the finite case are join-dense}, the facegram $\C _X$ can be written as a join-span of some join-irreducible elements. By 
Thm.~\ref{thm:filtration-irreducible}, those are treegrams. Then, we apply the surjective lattice morphism in Rem.~\ref{rem:epi} from the lattice of facegrams to the lattice of cliquegrams. Since this is a lattice morphism it respects the joinands (the treegrams), and hence, we can write the cliquegram as a join-span of treegrams inside the lattice of cliquegrams now. That means treegrams are join-dense inside the lattice of cliquegrams (although not join-irreducible).
\end{proof}
\begin{corollary} \label{cor:cliq-irred}
Symmetric ultranetworks over $X$ are join-dense  inside the lattice of phylogenetic networks over $X$, $\mathbf{Net}(X)$. In particular, in the setting of $\mathbf{Net}(X)$, PNRP admits a solution.  
\end{corollary}
In addition, we obtain the following interesting decomposition, connecting phylogenetic networks (including the case of finite metric spaces) with symmetric ultranetworks.
\begin{corollary}
    Let $N_X:X\times X\to\R$ be a phylogenetic network (e.g.~a finite metric space). Because of Prop.~\ref{prop:clique-dense} and \textit{Defn.~\ref{dfn:spanning}}, $N_X$ admits the following decomposition into spanning symmetric ultranetworks, i.e.:
    $$N_X=\min_{U_X\geq N_X} U_X.$$
    This decomposition is not unique among join decompositions of $N_X$, since in the above decomposition of $N_X$, one can only consider the $\vee$-maximal spanning elements (symmetric ultranetworks $U_X$) of $N_X$ i,n the lattice of phylogenetic networks $(\mathbf{Net}(X),\geq,\min,\max)$ (the join of that lattice is the entrywise minimum of the network matrix because the poset-order is given by the reverse inequality). 
\end{corollary}
    
\begin{example}
Given taxa set $X=\{ a, b, c, d\}$, the filtration given in \textit{Ex.~\ref{ex:filtration_decomposition}} is equivalent to the phylogenetic network

\begin{center}{\Small
\begin{tabular}{c|cccc}
$N_X$ & a & b & c & d \\
\hline
a & 0 & 1 & 2 & 3 \\ 
b & 1 & 0 & 1 & 2 \\
c & 2 & 1 & 0 & 1 \\
d & 3 & 2 & 1 & 0
\end{tabular}.
}\end{center}
Hence, the same decompositions apply here.
\end{example}

\subsection{Metrics for filtrations and facegrams}

A prominent notion of distance on persistence modules and broadly categories with a flow is the interleaving distance \cites{chazal2009proximity, desilva2018theory}. 

For facegrams over the same taxa set $X$ we define the interleaving metric between pairs of facegrams over $X$ and we show that it is isometric to the interleaving metric of their associated filtrations.
We recall the notion of topological equivalence, weak equivalence, and homotopy equivalence of filtrations over different vertex sets as in \cites{kim2023interleaving,memoli2017distance}. 
Then, we define the corresponding metrics for comparison of filtrations up to these three types of equivalences.
Because of the isometry between interleaving metrics of filtrations and facegrams, and because of the natural equivalence between filtrations and facegrams, the corresponding metrics for comparison of facegrams up to equivalence and up to homotopy equivalence are omitted, since we will not need them: it suffices to focus on metrics on filtrations since all the results about the metrics on filtrations can be transferred to these metrics too.

\textbf{Metrics for filtrations and facegrams over the same vertex sets}
First, we define the interleaving metrics for comparing a pair of filtrations, and facegrams, respectively.
\begin{definition}
    Let $\F _X,\F '_X:\mathbf{pow}(X)\to\R$ be two filtrations over $X$ and let $\caF _X, \caF '_X:\R\to\mathbf{Face}(X)$ be two facegrams over $X$. We consider the following interleaving distances:
    $$d_{\mathrm{I}}(\caF _X,\caF '_X):=\inf\{\e\geq0\mid \caF _X(t)\leq \caF '_X(t+\e) \text{ and }\caF '_X(t)\leq \caF _X(t+\e),\text{ for all }t\in\R\}$$
    $$d_{\mathrm{I}}(\F _X,\F '_X):=\max_{\emptyset\neq\sigma\subset X}|\F _X(\sigma)-\F '_X(\sigma)|$$
\end{definition}

\begin{proposition}
\label{prop:facegram-filtration-equiv}
Let $\F _X,\F '_X$ be two filtrations over $X$, and let $\caF _X,\caF '_X$ be the corresponding facegrams over $X$. Then we have the following interleaving isometry:
$$d_{\mathrm{I}}(\caF _X,\caF '_X)=d_{\mathrm{I}}(\F _X,\F '_X).$$
\end{proposition}
\begin{proof}
    By viewing facegrams equivalently as filtered simplicial complexes, then the proof follows the same way as in the proof of \textit{Thm.~4.8} in \cite{kim2023interleaving}.
\end{proof}
\textbf{Metrics for filtrations and facegrams over different vertex sets}
Quite often in applications of topology, we are interested in studying topological features of filtrations that remain invariant under a weaker notion of equivalence than just a topological equivalence. Such a notion is a weak equivalence of filtrations (also called \emph{pullbacks} \cite{memoli2017distance}), which is a special type of filtered homotopy equivalence yielded by the pullbacks of a given pair of filtrations (the name pullback appears in \cite{memoli2017distance}*{Sec.~4.1} and later the name weak equivalence appears in \cite{scoccola2020locally}*{Sec.~6.8}).

\begin{definition}\label{defn:weakeq}
Let $X$ and $Y$ be sets.
Let $\C _X$ and $C_Y$ be facegrams over $X$ and $Y$, respectively, and let $\F _X, \F _Y$ be filtrations over $X$ and $Y$, respectively. Then we define the following notions of similarity:
\begin{enumerate}
\item The facegrams $\C _X,C_Y$ are said to be \emph{isomorphic} if there exists a bijection $\varphi:X\to Y$ such that 
$C_Y(t)=\{\varphi(\sigma)\mid \sigma \in \C _X(t)\},$ 
for any $t\in\R$. Such a bijection is called an \emph{isomorphism}.
    \item The filtrations $\F _X, \F _Y$ are said to be \emph{isomorphic}, if there exists a bijection $\varphi:X\to Y$ such that $\F _Y\circ \varphi=\F _X$. Such a bijection is called an \emph{isomorphism}.
    \item Let $\F _X,\F _Y$ be filtrations over the vertex sets $X,Y$, respectively. A pair of surjections $X\xtwoheadleftarrow{\varphi_X} Z\xtwoheadrightarrow{\varphi_Y}Y$, also called a \emph{tripod}, is said to be a \emph{weak equivalence of $\F _X, \F _Y$} if 
  their pullbacks along $\varphi_X$ and $\varphi_Y$ are equal, i.e.~$\varphi_X^* \F _X=\varphi_Y^* \F _Y$. In that case, $\F _X$, $\F _Y$ are called \emph{weakly equivalent}.
\end{enumerate}
\end{definition}

Note that the term \emph{weak equivalence} was used in further generality as an equivalence between objects in certain types of categories \cite{scoccola2020locally}. Here, we define this type of equivalence for the case of filtrations.

\begin{example}
    Any isomorphism of a pair of filtrations $\F _X$, $\F _Y$ gives rise to an isomorphism of their associated facegrams $\caF _X$, $\caF _Y$, and vice versa.
\end{example}

Because of Prop.~\ref{prop:facegram-filtration-equiv},
we will only focus on the definition of distance for filtrations, since a distance for facegrams can be defined analogously. 
 
\begin{definition}[{\cite{memoli2017distance}*{Sec.~4.1.1.}}]
Let $\F _X$ and $\F _Y$ be filtrations over the sets $X$ and $Y$, respectively.
The \emph{tripod distance of $\F _X,\F _Y$} is defined as
$$d_{\mathrm{T}}(\F _X,\F _Y):=\inf_{X\xtwoheadleftarrow{\varphi_X} Z\xtwoheadrightarrow{\varphi_Y}Y}d_{\mathrm{I}}(\varphi_X^* \F _X,\varphi_Y^* \F _Y).$$
\end{definition}

For comparing phylogenetic networks, we utilize the \emph{network distance} of Chowdhury et al.~\cite{chowdhury2022distances}, which is a generalization of the Gromov-Hausdorff from the setting of metric spaces to the setting of networks. Below we are considering the equivalent definition of the network distance which uses tripods instead of correspondences \cite{chowdhury2022distances}*{Sec.~2.3}. 
\begin{definition}[{\cite{chowdhury2022distances}*{Defn.~2.27}}]
Let $N_X$ and $N_Y$ be phylogenetic networks over the sets $X$ and $Y$, respectively.
The \emph{Gromov-Hausdorff distance of $N_X,N_Y$} is defined as
$$d_{\mathrm{GH}}(N_X,N_Y):=\inf_{X\xtwoheadleftarrow{\varphi_X} Z\xtwoheadrightarrow{\varphi_Y}Y}\max_{z,z'\in Z}|N_X(\varphi_X(z),\varphi_X(z'))-N_Y(\varphi_Y(z),\varphi_Y(z'))|.$$
\end{definition}

\begin{remark}
\label{rem:GH realization remark}
    Let $(X,N_X)$ and $(Y,N_Y)$ be phylogenetic networks (e.g.~finite metric spaces) over $X$ and $Y$, respectively, and let $\mathbf{VR}_{X},\mathbf{VR}_{Y}$ be their associated Vietoris-Rips filtrations.
    Then 
    $$d_{\mathrm{GH}}((X,N_X),(Y,N_Y))=d_{\mathrm{T}}(\mathbf{VR}_{X},\mathbf{VR}_{Y}).$$
    This equality was shown in \cite{memoli2021quantitative}*{Prop.~2.8} for the case of finite metric spaces. It is easy to verify it for the more general case of phylogenetic networks, too.
\end{remark}

\section{Invariants and stability} 
 We develop the following invariants of filtrations and facegrams:
(i) the \emph{face-Reeb graph},
     (ii) the \emph{mergegram}, 
     and (iii) its refinement called the \emph{labeled mergegram}.  
We then investigate their stability and computational complexity. We show that the face-Reeb graph is not a stable invariant of facegrams over the same taxa set, but the mergegram invariant is. 
The labeled mergegram is also not stable, but will be needed for faster computations of the unlabeled mergegram in the case of the join-facegram. 
In addition, we show that the mergegram is (i) invariant of weak equivalences of filtrations and that (ii) it is a stable invariant of filtrations (and thus facegrams) over different vertex sets with respect to M\'emoli's tripod distance \cite{memoli2017distance}.
These results might be of independent interest to TDA as well.

\begin{remark}
The reader can safely skip the following subsection on face-Reeb graphs and go directly to Subsec.~\ref{subsec:mergegram}, since the face-Reeb graph is shown to be an unstable invariant of facegrams and cliquegrams and is not used in the upcoming applications. Nevertheless, it is an important natural motivation for the subsequent study of invariants of facegram.
\end{remark}

\subsection{The face-Reeb graph}
Every dendrogram has an underlying structure of a merge tree. Analogously, the facegram has an underlying structure of a rooted Reeb graph. To see this, first, we need to recall the more technical definition of rooted Reeb graphs as in \cite{stefanou2020tree} which depends on choosing a critical set. See also \cite{de2016categorified} for the standard definition.

\begin{definition}[{\cite{de2016categorified}*{Not.~2.5}}]
\label{def:Reeb def}
An \emph{$\R$-space} $(\X,f)$ is a space $\X$ together with a real-valued continuous map $f:\X\to\R$.
An $\R$-space $(\X,f)$ is said to be a \emph{rooted Reeb graph} if
it is constructed by the following procedure, which we call \emph{a structure on $(\X,f)$}:

 Let $S=\{a_1<\ldots<a_{n}\}$ be an ordered subset of $\R$ called a \emph{critical set} of $\X$.

\begin{itemize}
    \item For each $i=1,\ldots,n$ we specify a set $V_i$ of vertices which lie over $a_i$,
    \item For each $i=1,\ldots,n-1$ we specify a set of edges $E_i$ which lie over $[a_i,a_{i+1}]$
    \item For $i=n$ we specify a singleton set $E_n$ that lie over $[a_n,\infty)$,
    \item For $i = 1,\ldots,n$, we specify a down map $\ell_i:E_i \to V_i$
    \item For $i = 1,\ldots,n-1$, we specify an upper map $r_i: E_i \to V_{i+1}$.
\end{itemize}
The space $\X$ is the quotient $\U/\sim$ of the disjoint union
$$\U=\coprod_{i=1}^{n}(V_i\times\{a_i\})\coprod_{i=1}^{n-1}(E_i\times [a_i,a_{i+1}])\coprod(E_n\times [a_n,\infty))$$
with respect to the identifications $(\ell_i(e),a_i)\sim(e,a_i)$ and $(r_i(e),a_{i+1})\sim(e,a_{i+1})$, with the map $f$ being the projection onto the second factor.
\end{definition}
\begin{definition}
     Let $\caF _X$ be a facegram with critical set $a_1<\ldots<a_n$. The underlying \emph{face-Reeb graph} $R_X$ of $\caF _X$ is determined by $\left(\{V_{i}\}_{i=1}^n,\{E_i\}_{i=1}^{n},\{\ell_i:E_i\to V_i\}_{i=1}^{n},\{r_i:E_i\to V_{i+1}\}_{i=1}^n\right)$, where,
 \begin{align*}
     E_i&:=\pi_0(\caF _X(a_i),\subset)=\caF _X(a_i)\\
     V_i&:=\pi_0\left(\caF _X(a_{i-1})\cup \caF _X(a_{i}),\subset\right)\\
     \ell_i&:=\pi_0\left( \caF _X(a_i)\hookrightarrow \caF _X(a_{i-1})\cup \caF _X(a_{i})\right)\\
     r_i&:=\pi_0\left(\caF _X(a_i)\hookrightarrow \caF _X(a_{i})\cup \caF _X(a_{i+1})\right).
 \end{align*}
\end{definition}
Thus, every facegram has an associated Reeb graph structure. However, the Reeb graph structure is not a stable invariant of facegrams as shown in the figure below.

\begin{figure}[ht]
\centering
\includegraphics[width=0.95\textwidth]{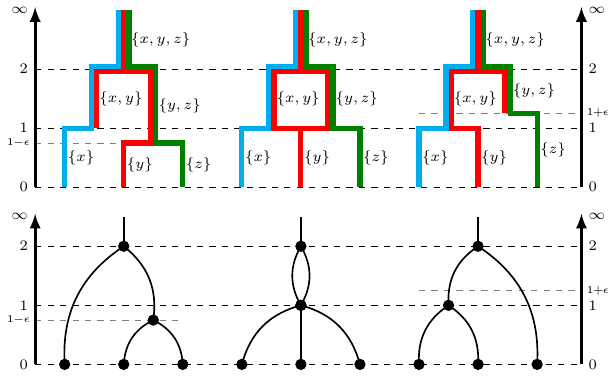}
\caption{{\bf Facegrams and their face-Reeb graphs for small perturbations.} 
On the top there are three facegrams and on the bottom their associated face-Reeb graphs. On the top of the figure, on the left and to the right of the middle facegram there are two facegrams that are $\e$-close to it, in the interleaving distance. However, for small $\e\geq0$, at least for $\e$-values that are smaller than half the height of the loop of the middle facegram, the associated Reeb graphs of those facegrams are not $\e$-close to the middle facegram, since their difference is at least as large as half of the height of the loop of the middle facegram. That means that a small perturbation of the facegrams can lead to a large difference in the face-Reeb graph of the facegrams.}
\label{fig:instability}
\end{figure}
 This example motivated us to look for stable invariants of facegrams (our main model of phylogenetic reconstruction).  
Such an invariant of facegrams should be (i) stable under perturbations of facegrams, (ii) computable, (iii) while capturing as much information as possible. 
This will be discussed below.

\subsection{The mergegram invariant}\label{subsec:mergegram}

In this section, we extend the notion of a mergegram by Elkin, Kurlin \cite{elkin2020mergegram} from the setting of dendrograms to the setting of facegrams. 
Furthermore, in the next sections we show that the mergegram in the setting of facegrams is still stable. Additionally, we show that the mergegram is invariant of weak equivalence of filtrations (a stronger notion of homotopy equivalence of filtration and weaker than isomorphism).

\begin{definition}
Let $\caF _X:(\R,\leq)\to\mathbf{Face}(X)$ be a facegram over $X$. We define the \emph{mergegram of $\caF _X$} as the multiset (i.e.~set with repeated elements)
of intervals
$$\mathbf{mgm}(\caF _X):=\left\{\left\{I_\sigma\mid \sigma\subset X
\right\}\right\}\text{, where }I_\sigma:=\caF _X^{-1}(\sigma)=\{t\in\R\mid \sigma\in \caF _X(t)\}.$$
Given a $\sigma\subset X$, $I_\sigma\subset\R$ is either an interval or $\emptyset$, and it is called \emph{the lifespan of $\sigma$ in $\caF _X$}.  
\end{definition}

 For an example of a mergegram see Fig.~\ref{fig:labeledmergegram}.~It is easy to check that the mergegram of a facegram corresponds to the edges of the face-Reeb graph of the facegram, viewed as intervals of $\R$. Thus, the face-Reeb graph is more informative than the mergegram.
Just as in the case of dendrograms, the mergegram and the Reeb graph of a facegram are topological invariants of the facegram. 

\begin{proposition}
The mergegram of a facegram is invariant to isomorphisms of facegrams. 
\end{proposition}
\begin{proof}
Suppose that $f:X\to Y$ is an equivalence between $\caF _X$ and $C_Y$. Since also $f:X\to Y$ is a bijection, for any $\sigma\subset X$, $\sigma$ and $f(\sigma)$ have the same cardinalities. In particular, $\sigma\in \caF _X$ if and only if $f(\sigma)\in \caF _Y$. Thus, $I_\sigma=I_{f(\sigma)}$ whenever $I_\sigma\neq\emptyset$ and $\sigma\neq\emptyset$. Therefore, $\mathbf{mgm}(\caF _X)=\mathbf{mgm}(\caF _Y)$.  
\end{proof}

\textbf{A refinement of the mergegram} In general, the mergegram is not a complete invariant of facegrams, meaning that there exist different facegrams with the same mergegram. For instance, if we consider a non-tree like facegram and then consider any of its spanning treegrams (viewed as facegrams). These spanning treegrams will contain the same lifespans of edges, and hence the same mergegrams, as the given facegram. We can strengthen the mergegram invariant by considering labeling the lifespans of the edges by their associated simplices. We define this formally as follows.
\begin{definition}
Let $\caF _X:(\R,\leq)\to\mathbf{Face}(X)$ be a facegram over $X$. We define the \emph{labeled mergegram of $\caF _X$} as the set
$$\mathbf{mgm}^{*}(\caF _X):=\{(\sigma,I_\sigma)\mid \emptyset\neq\sigma\subset X\text{ and }I_\sigma\neq\emptyset \}.$$
\end{definition}

  We have the following decomposition result.
  \begin{proposition}
  \label{prop:decomp-clut}
  Let $\caF _X:(\R,\leq)\to \mathbf{Face}(X)$ be a facegram. For any $t\in \R$ we have 
  $$\caF _X(t)=\bigvee_{\emptyset\neq\sigma\subset X} \caF _X^{(\sigma)}(t),$$
  where the face-set $\caF _X^{(\sigma)}(t)$ is given by
  \[\caF _X^{(\sigma)}(t):=
   \begin{cases} 
      \{\sigma\}, & t\in I_\sigma \\
        \emptyset, & \text{otherwise.} 
   \end{cases}
\]
  \end{proposition}
  \begin{proof}
       The proof follows directly by the fact that every simplicial complex is determined by its clique set of maximal faces. So, we apply this fact point-wisely in the facegram. 
  \end{proof}
\begin{corollary}
The labeled mergegram $\caF _X\mapsto \mathbf{mgm}^*(\caF _X)$ is an injection. Hence, it is a complete invariant of facegrams over $X$.
\end{corollary}
\begin{proof}
Let $\caF _X,\caF '_X$ be two facegrams over $X$.
Suppose that $\mathbf{mgm}^{*}(\caF _X)=\mathbf{mgm}^{*}(\caF '_X)$. Let $I_\sigma^{(\caF _X)},I_\sigma^{(\caF '_X)}$, $\emptyset\neq\sigma\subset X$, be the lifespans of $\caF _X,\C '_X$, respectively. Then, by definition of $\mathbf{mgm}^*$, we obtain $(\sigma,I_\sigma^{(\caF _X)})=(\sigma,I_\sigma^{(\caF '_X)})$, for all $\emptyset\neq\sigma\subset X$. Hence, $\caF _X^{(\sigma)}(t)={\caF '}_X^{(\sigma)}(t)$, for all $\emptyset\neq\sigma\subset X$ and all $t\in\R$. By Prop.~\ref{prop:decomp-clut}, we obtain $\caF _X(t)=\caF '_X(t)$, for all $t\in\R$. Therefore, $\caF _X=\caF '_X$.
\end{proof}

\begin{figure}[ht]
\centering
\includegraphics[width=1\textwidth]{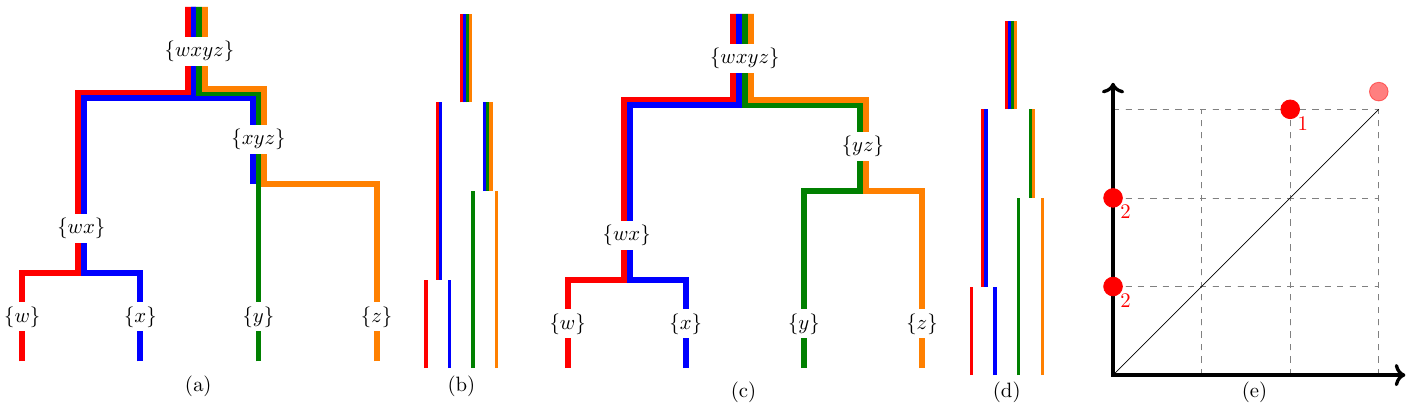}
\captionsetup{width=.95\textwidth}
\caption{{\bf Visualizations of two facegrams with the same mergegram but different labeled mergegram.} 
Visualizations (a), (c) of two facegrams having isomorphic face-Reeb graphs (if we forget the labels of their edges) and thus the same mergegrams, however differently labeled. 
The mergegrams correspond in this visualization to the vertical bars without their labels; (b), and (d) show the bars/intervals $I_\sigma$ together with their label; these are the labeled mergegrams of facegrams (a) and (c), respectively. The coloring associated with each interval is the colors of taxa in the clique. Interpreting the mergegram as a persistence diagram gives plot (e), where the multiplicity of the points is given by the numbers next to the points. The point with lower opacity is the point with infinite lifetime. 
We can see that the mergegrams are the same in (e) or in (b) and (d) when ignoring the coloring of the bars. 
On the other hand, the labeled mergegram (and the associated coloring of the bars) shows that the two facegrams differ in one face-set, namely $(\{x,y,z\}, [2,3))$ and $(\{y,z\}, [2,3))$. Finally, note that: the facegram (c) is a dendrogram and in particular a treegram. However, for the facegram (a) its associated face-Reeb graph is a merge tree and yet that facegram is neither a dendrogram nor a treegram. This implies in particular that the facegram (a) is the join span of at least two different treegrams (viewed as facegrams).
Moreover, by definition, the face-Reeb graph of a facegram would not be a merge tree if and only if there exists a taxon that can be joined with the root by at least two different \emph{consecutive} paths of maximal faces through time, and this is clearly not the case for the facegram (a) since $\{x\}$ cannot be joined by a consecutive path of inclusions of maximal faces with $\{x,y,z\}$ through time. 
}
\label{fig:labeledmergegram}
\end{figure}

\textbf{The mergegram of a filtration}
In this paragraph, we show how the mergegram of a given facegram can be expressed in terms of the filtration associated with that facegram.~This implies that the mergegram can be used as a new topological signature of filtrations without reference to the facegram, suggesting using mergegrams as a new topological signature of filtrations of datasets that can complement the existing toolkit of persistence invariants in TDA.

  \begin{proposition}
  \label{prop:mgm-formula}
    Let $\caF _X$ be a facegram over $X$ and let $\F _X:=\Psi(\caF _X)$ be the filtration over $X$ associated to $\caF _X$ via the one-to-one correspondence $(\Phi,\Psi)$ in \textit{Thm.~\ref{thm:equivalencetwo}}. Let $$\mathbf{s}(\caF _X):=\left\{\emptyset\neq\sigma\subset X\mid \sigma\in \caF _X(t)\text{, for some }t\in\R\right\}.$$ 
    Given a $\emptyset\neq\sigma\subset X$, $\sigma\in \mathbf{s}(\caF _X)$ if and only if $I_\sigma$ is a non-empty interval of $\R$. Moreover, for any $\sigma\in \mathbf{s}(\caF _X)$, $I_\sigma$ is given by the formula
$$I_\sigma=\left[\F _X(\sigma),\min_{\sigma\subsetneqq\tau\subset X} \F _X(\tau) \right).$$

 Moreover, for any $\emptyset\neq\sigma\subset X$, we have: $$\min_{\sigma\subsetneqq\tau\subset X} \F _X(\tau) =\min_{\sigma\subsetneqq\tau\in\mathbf{s}(\caF _X)} \F _X(\tau).$$ 
 Hence, the mergegram of $\caF _X$ is given by the following multiset
 of intervals:
 $$\mathbf{mgm}(\caF _X)=\left\{\left\{\left [\F _X(\sigma),\min_{\sigma\subsetneqq\tau\in\mathbf{s}(\caF _X)} \F _X(\tau) \right)\mid \sigma\in \mathbf{s}(\caF _X)\right\}\right\}.$$
 \end{proposition}
  \begin{proof}
The proof is straightforward by definition of $\F _X$, $I_\sigma$ and $\mathbf{s}(\caF _X)$.
\end{proof}
We obtain the following useful corollary.
\begin{corollary}
\label{cor:S-cor}
Let $\caF _X$ be a facegram over $X$.
For any collection $S$ of simplices such that $\mathbf{s}(\caF _X)\subset S\subset \mathbf{pow}(X)$, we have:
$$\mathbf{mgm}(\caF _X)=\left\{\left\{\left [\F _X(\sigma),\min_{\sigma\subsetneqq\tau\in S} \F _X(\tau) \right)\mid \sigma\in S\right\}\right\}.$$
\end{corollary}
Because of Prop.~\ref{prop:mgm-formula} which shows the direct expression of the mergegram in terms of the filtration alone, and because of the equivalence between facegrams and filtrations, from now we can also talk about the mergegram of a filtration. 

For an easier notation later on, we restate the above with the equivalent definition below (formally they differ only by adding empty intervals and so, as multisets, they are the same).  

\begin{definition}
   Let $\F _X:\mathbf{pow}(X)\to\R$ be a filtration over $X$. 
The \emph{mergegram of the filtration $\F _X$} is defined as 
$$\mathbf{mgm}(\F _X):=
\left\{\left\{\left [\F _X(\sigma),\min_{\sigma\subsetneqq\tau\in\mathbf{pow}(X)} \F _X(\tau) \right)\mid
\sigma\in \mathbf{pow}(X)
\right\}\right\}.$$
Similarly, we can define the labeled mergegram of a filtration, and denote it by $\mathbf{mgm}^*(\F _X)$.
\end{definition}

\subsection{Invariant properties of mergegrams}
 Although an isomorphism invariant, the mergegram of a filtration is not invariant up to homotopy of filtrations (see Rem.~\ref{rem:facts about max faces invariances}). However, it turns out that we can consider a particular form of homotopy equivalence of filtrations called \emph{weak equivalence}, (see \textit{Defn.~\ref{dfn:weak equiv}}) under which the mergegram remains invariant (see Prop.~\ref{thm:invariance of mergegram}).

\begin{definition}[{\cite{memoli2017distance}*{Sec.~4.1}}]
\label{dfn:weak equiv}
    Let $\F _X:\mathbf{pow}(X)\to\R$ be a filtration over $X$ and let $\varphi_X:Z\twoheadrightarrow X$ be a surjection. Then, we define the \emph{pullback filtration $\varphi_X^* \F _X$ induced by $\F _X$ and $\varphi_X$}, given by 
    \begin{align*}
        \varphi_X^* \F _X:\mathbf{pow}(Z)&\to\R\\
        \kappa&\mapsto \F _X(\varphi_X(\kappa))
    \end{align*}
  \end{definition}

\begin{lemma}
\label{lem:max-faces-lemma} Let $(X,\S_X)$ be a simplicial complex, and let $\varphi_X:Z\twoheadrightarrow X$ be a surjection. Let $(Z,\varphi_X^* \S_X)$ be the pullback complex of $(X,\S_X)$ given by 
    $$\varphi_X^* \S_X:=\{\kappa\subset Z\mid \varphi_X(\kappa)\in \S_X\}.$$
Let $\varphi_X(\cdot):\mathbf{pow}(Z)\to\mathbf{pow}(X)$, 
$\sigma \mapsto \varphi_X(\sigma) := \cup _{z \in \sigma}\varphi_X(z)$ 
be the map induced by $\varphi_X$ and taking the union.
Then, the restriction of this map from the set of maximal faces of the complexes $(Z,\varphi_X^* \S_X)$ to the set of maximal faces of the complex $(X,\S_X)$ is a bijection. 
\end{lemma}
\begin{proof}
    Consider the restriction map of $\varphi_X(\cdot):\mathbf{pow}(Z)\to\mathbf{pow}(X)$ that maps maximal faces of $\varphi^*_X \S_X$ to maximal faces of $\S_X$. 
    This map is well-defined because of the surjectivity of $\varphi_X$. 
    By definition, it follows that the pre-image of any maximal face of $\S_X$ is a maximal face of the pullback complex, and that the map is clearly surjective.\\
    Now, we claim that this restriction map of $\varphi_X(\cdot)$ is also injective. Let $\kappa_1,\kappa_2$ be two maximal faces $\varphi_X^* \S_X$ such that $\varphi(\kappa_1)=\varphi_X(\kappa_2)=:\sigma\in \S_X$. 
    Consider the union $\kappa_3:=\kappa_1\cup \kappa_2$. 
    Then, $\varphi_X(\kappa_3)=\varphi_X(\kappa_1\cup \kappa_2)=\varphi_X(\kappa_1)\cup\varphi_X(\kappa_2)=\sigma\in \S_X$. Hence, $\kappa_3\in \varphi_X^* \S_X$. 
    Because of $\kappa_1,\kappa_2\subset \kappa_3$ and because $\kappa_1,\kappa_2$ are maximal faces of the complex $\varphi_X^* \S_X$, we conclude that $\kappa_1=\kappa_2=\kappa_3$. Therefore, our surjective map is also injective, and thus a bijection.
    \end{proof}
\begin{remark}
    Note that although the pullback complex $\varphi^*_X \S_X$ is different from the complex $\S_X$, using Quillen's theorem A in the simplicial category \cite{quillen2006higher}, it was shown by M\'emoli that these complexes are homotopy equivalent \cite{memoli2017distance}*{Cor.~2.1}.
    
\end{remark}

\noindent
Recall Def.~\ref{defn:weakeq} for weak equivalences of filtrations. Then we have the following.

\begin{theorem}
    \label{thm:invariance of mergegram}
    The mergegram is invariant of weak equivalences of filtrations. 
\end{theorem}
\begin{proof}
Let $\F _X$ and $\F _Y$ be filtrations over the vertex sets $X$ and $Y$, respectively.
Let $\caF _X, \caF _Y,\varphi_X^* \caF _X$ be the associated facegrams of $\F _X,\F _Y,\varphi_X^* \caF _X$, respectively.
Suppose that $\F _X$ and $\F _Y$ are weakly equivalent. 
That is, there is a tripod $X\xtwoheadleftarrow{\varphi_X} Z\xtwoheadrightarrow{\varphi_Y}Y$, such that 
\begin{equation}
\label{eqn:pullback}
    \varphi_X^* \F _X=\varphi_Y^* \F _Y.
\end{equation}
We claim that $\mathbf{mgm}(\F _X)=\mathbf{mgm}(\F _Y)$.
Because of Eqn.~(\ref{eqn:pullback}), it suffices to show the equality $\mathbf{mgm}(\F _X)=\mathbf{mgm}(\varphi^* \F _X)$. 
By surjectivity of $\varphi_X$, it follows that those mergegrams can only be different on the multiplicities of some of their intervals. 
Now, we need to craft a bijection between those multisets, that preserves their multiplicities.
We consider the map 
\begin{align*}
    \overline{\varphi}_X:\mathbf{mgm}(\varphi_X^* \F _X)&\mapsto\mathbf{mgm}(\F _X) \\
    I_\kappa&\mapsto I_{\varphi_X(\kappa)}
\end{align*}
This map is surjective counting multiplicities, since $\varphi_X$ is surjective.
By \textit{Lem.~\ref{lem:max-faces-lemma}}, the map of maximal faces, $k\mapsto \varphi_X(k)$, is bijective. Hence, the map $I_\kappa\mapsto I_{\varphi_{X}(\kappa)}$ is a bijection which preserves multiplicities.
Thus, the mergegrams are equal.

\end{proof}
\begin{remark}
\label{rem:facts about max faces invariances}
    It follows by \textit{Lem.~\ref{lem:max-faces-lemma}} that the number of maximal faces of a simplicial complex is invariant of pullbacks. However, the number of maximal faces of a complex is not invariant up to homotopy equivalence of complexes: to see this, consider two complexes $\S_X$ and $S_Y$, where $\S_X$ is a triangle and $S_Y$ is the complex obtained by gluing two triangles along a common edge. These complexes are clearly homotopy equivalent, however, they have a different number of maximal faces ($1$ and $2$ respectively). This implies that the mergegram, although invariant of weak equivalences, is not invariant up to homotopy equivalence of filtrations. 
\end{remark}
\begin{corollary}
    Let $(X,N_X),(Y,N_Y)$ be two phylogenetic networks (e.g.~finite metric spaces). Let $\mathbf{VR}_{X}:\mathbf{pow}(X)\to\R$ and $\mathbf{VR}_{Y}:\mathbf{pow}(Y)\to\R$ be the associated Vietoris-Rips filtrations. If there exists an isometry $\varphi:X\to Y$ between those networks, then
    $\mathbf{mgm}(\mathbf{VR}_{X})=\mathbf{mgm}(\mathbf{VR}_{Y})$.
\end{corollary}
\subsection{Stability of the mergegram}
 We define the bottleneck distance for comparing mergegrams and an $\ell^\infty$-type of metric for comparing labeled mergegrams. 
Then, we show that the mergegram invariant of filtrations is $1$-Lipschitz stable in the perturbations of the tripod distance on filtrations.

Since mergegrams are multisets of intervals, they can be thought of as persistence diagrams. Thus, we can consider the bottleneck distance as a metric for the comparison of mergegrams. 
\begin{definition}[Bottleneck distance]
Let $A,B$ be a pair of multisets of intervals. Let $\e\geq0$. An \emph{$\e$-matching} of $A,B$ is a bijection $\varphi:A'\to B'$ between a subset $A'\subset A$ and $B'\subset B$ such that
\begin{enumerate}
    \item for every $I=[a,b]\in A'$ and $\varphi(I)=[c,d]\in B'$, $||(a,b)-(c,d)||_\infty\leq \e\text{, and }$ 
    \item for every $[a,b]\in (A\setminus A')\cup (B\setminus B')$, $|a-b|\leq2\e$. 
 \end{enumerate}
 The \emph{bottleneck distance of $A,B$} is defined as 
$$d_{\mathrm{B}}(A,B):=\inf\{\e\geq0\mid \text{there exists an }\e\text{-matching of }A,B\}.$$ 
\end{definition}
We consider an $\ell^\infty$-type of metric for comparing labeled mergegrams over $X$.
\begin{definition}
Let $\caF _X,\caF '_X$ be facegrams over $X$, respectively. Their $\ell^\infty$-distance is
$$d_{\infty}(\mathbf{mgm}^*(\caF _X),\mathbf{mgm}^*(\caF '_X)):=\max_{\sigma\in\mathbf{pow}(X)}d_{\mathrm{B}}(I_\sigma^{\caF _X},I_\sigma^{\caF '_X}).$$
\end{definition}
We can also consider a tripod-like distance for comparing labeled mergegrams over different sets, however, we are not intending to use such a distance in this paper, so we omit such a definition.

\begin{proposition}
\label{prop:stability-mgm}
Let $X$ be a taxa set. Let $\F _X,\F '_X$ be two filtered simplicial complexes over $X$ and let $\caF _X$ and $\caF '_X$ be their associated facegrams. 
Then, we have the following bounds
\begin{align*}
d_{\mathrm{B}}(\mathbf{mgm}(\F _X),\mathbf{mgm}(\F '_X))
&\leq  d_\infty(\mathbf{mgm}^*(\caF _X),\mathbf{mgm}^*(\caF '_X)) \\
&\leq d_{\mathrm{I}}(\caF _X, \caF '_X)=d_{\mathrm{I}}(\F _X,\F '_X).
\end{align*}
\end{proposition}
\begin{proof}
 The equality follows directly from the fact that the lattice of face-sets is isomorphic to that of simplicial complexes and applying the property of the interleaving distance that the associated persistent structures are interleaving isometric (see Bubenik et al.~\cite{bubenik2015metrics}*{Sec.~2.3}). 
Assume that $d_I(\caF _X,\caF '_X)\leq\e$. That means that $\caF _X(t)\leq \caF '_X(t+\e)$ and $\caF '_X(t)\leq \caF _X(t+\e)$. This in turn implies that each of the lifespans of the intervals in the labeled mergegram is of length at most $\e$. This explains the bound with the labeled mergegram. Since the $d_\infty$ distance of labeled mergegrams is an $\ell^\infty$-type of metric, this bound also implies that there is a canonical $\e$-matching $I_\sigma^{(\caF _X)}\mapsto I_\sigma^{(\caF '_X)}$ between the intervals of the ordinary mergegrams $\mathbf{mgm}(\caF _X)$ and $\mathbf{mgm}(\caF '_X)$.
So, the bottleneck distance of the ordinary mergegrams is at most $\e$ too. 
\end{proof}

Next, we establish the stability of the mergegram by following the same strategy as in \cite{memoli2017distance} and other works of M\'emoli et al., e.g.~\cite{kim2023interleaving}. More precisely, now that we have established the invariance of mergegrams under weak equivalences of filtrations (Thm.~\ref{thm:invariance of mergegram}), we show the stability of the mergegram invariant by invoking stability for the invariant when the underlying sets are the same.

\begin{proposition}[$1$-Lipschitz stability of the mergegram]
\label{prop:hausd-stability-mgm}
Let $X$ and $Y$ be two finite sets.
\begin{enumerate}
\item Let $\F _X,\F _Y$ be a pair of filtrations over $X$ and $Y$, respectively. Then
$$d_{\mathrm{B}}(\mathbf{mgm}(\F _X),\mathbf{mgm}(\F _Y))\leq d_{\mathrm{T}}(\F _X,\F _Y).$$
    \item Let $(X,N_X),(Y,N_Y)$ be a pair of phylogenetic networks (e.g.~finite metric spaces) and let $\mathbf{VR}_{X},\mathbf{VR}_{Y}$ be their Vietoris-Rips filtrations. 
    Then
    $$d_{\mathrm{B}}(\mathbf{mgm}(\mathbf{VR}_{X}),\mathbf{mgm}(\mathbf{VR}_{Y}))\leq d_{\mathrm{GH}}((X,N_X),(Y,N_Y)).$$
\end{enumerate}
\end{proposition}

\begin{proof}
    Let $X\xtwoheadleftarrow{\varphi_X} Z\xtwoheadrightarrow{\varphi_Y}Y$ be a tripod between $X$ and $Y$. 
    We have
    \begin{align*}
            d_{\mathrm{B}}(\mathbf{mgm}(\F _X),\mathbf{mgm}(\F _Y))&=d_{\mathrm{B}}(\mathbf{mgm}(\varphi_X^* \F _X),\mathbf{mgm}(\varphi_Y^* \F _Y))\\
    &\leq d_{\mathrm{I}}(\varphi_X^* \F _X,\varphi_Y^* \F _Y).
    \end{align*}
    Since the tripod was arbitrary, by taking infimums over all tripods, we obtain
    $$d_{\mathrm{B}}(\mathbf{mgm}(\F _X),\mathbf{mgm}(\F _Y))\leq d_{\mathrm{T}}(\F _X,\F _Y).$$
    The second claim follows directly by Rem.~\ref{rem:GH realization remark}.
\end{proof}

\begin{example}
Consider the three facegrams from Fig.~\ref{fig:instability}.
We can see in the figure below that the mergegrams of those facegrams are at most $\e$-close to each other in the bottleneck distance.
\begin{figure}
\centering
\includegraphics[width=0.95\textwidth]{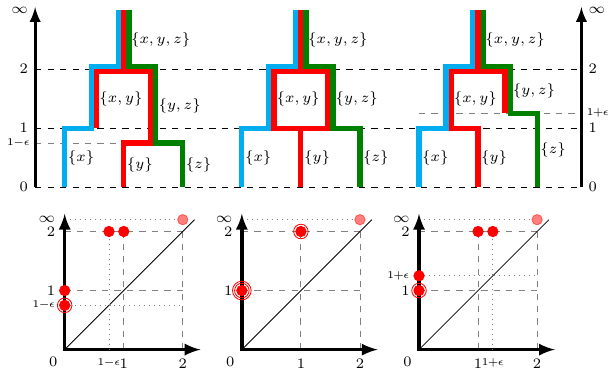}
\caption{{\bf Perturbed facegrams and their mergegrams.} Facegrams and their respective mergegrams. Points with multiplicity $2$ and $3$ are shown by drawing multiple rings around a point.}
\label{fig:stability}
\end{figure}
\end{example}

\section{Computational aspects}
In this section, we investigate (i) the complexity of computing the join-span of a family of treegrams in the lattice of cliquegrams and in the lattice of facegrams, respectively, and the (ii) complexity of computing the mergegram for these cases.
As we describe in the 
Figure \ref{tab:overview_algorithms}, we have two different classes of inputs (the top one which is motivated by phylogenetic reconstruction and the bottom one which is motivated by TDA) for which we want to compute the mergegrams of the respective cliquegram or facegram structures. In the following, we will either describe how to pass along each arrow and what to be careful of or give pseudo-code for the computations. As can be seen, there are 4 different paths to be taken from the inputs to their mergegrams.

\begin{figure}
\centering
\includegraphics[width=0.8\textwidth]{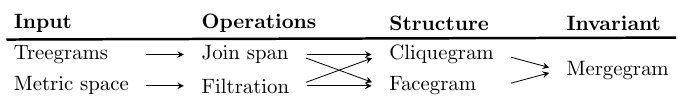}
\caption{Overview over the algorithmic structures. To compute the join of treegrams we can use the join operation of the lattice of cliquegrams or of facegrams and then calculate the respective mergegram invariants. Similarly, we can compute the mergegram for metric spaces
either using cliquegrams or facegrams. The arrows represent the different paths to be taken and represent different algorithms to be used.}
\label{tab:overview_algorithms}
\end{figure}
For both cliquegram and facegram, we provide two (slightly) different algorithms. 
The second algorithms both focus on using the matrix representation to compute birth and death times directly and independently from other steps in the iteration throught the critical values. Thus, those can be more readily adapted in a parallel and somewhat distributed framework of computations in the future. 

In the case of cliquegram we exploit the representation via a phylogenetic network for computational efficiency, i.e. a $n\times n$-matrix for $n$ being the number of common taxa, since there exists a one-to-one correspondence due to \textit{Prop.~\ref{prop:equivalence}}. 
Furthermore, the different values in the matrix correspond to the different values in the critical set of the associated cliquegram. 

For facegrams such a representation via matrices does not exist, and we need to use the face-sets with their respective filtration values. In the case of the join-facegram of a family of treegrams, we can speed-up computation using the fact that the information of maximal faces is completely encoded in the treegrams, see \textit{Alg.~\ref{alg:join_tree_facegram}}. Using the information of the associated ultranetworks of the treegrams, we can compute it differently, as done in \textit{Alg.~\ref{alg:join_tree_facegram_second}}.  
\subsection{The mergegram of a cliquegram}\label{sec:mergegram_cliquegram}

Given a set $\{U^1_X,\ldots,U_X^\ell\}$ of $\ell$ symmetric ultranetworks over a set of taxa $X=\{x_1,\ldots,x_n\}$, the join-cliquegram $\caC _X:=\vee_{i=1}^\ell \caT _X^i$ of their associated treegrams $\caT _X^i:=\varphi(U_X^i)$, $i=1,\ldots,\ell$ can be computed as mentioned in Remark \ref{remark:join_cliquegram}. Hence, computing the join-cliquegram is the same as computing any cliquegram in general: 

\begin{enumerate}
    \item Take the minimizer matrix (entry-wise minimum of the symmetric ultranetworks over all associated treegrams). The minimizer matrix is a symmetric network, more specifically a phylogenetic network $N_X$, which is unique due to \textit{Prop.~\ref{prop:equivalence}}.
    \item Compute the associated cliquegram for the minimizer matrix - this is the join-cliquegram - using Algorithm \ref{alg:cliquegramfromnetwork}. The different values in the minimizer matrix will be the different values of the critical set of the cliquegram.
    \item In general, we can now compute the mergegram of that cliquegram, via Algorithm \ref{alg:mergegram_cliquegram}.
\end{enumerate}
The pseducode for this algorithm is given in \textit{Alg.~\ref{alg:mergegram_cliquegram}}.\par

\begin{algorithm}
\caption{The mergegram of a cliquegram}\label{alg:mergegram_cliquegram}
\begin{algorithmic}[1]
\State \textbf{Input: } Given a cliquegram $\caC _X : \mathbb{R}\to\mathbf{Cliq}(X)$, 
with critical set $\varepsilon_1 < \varepsilon_2 < \ldots <\varepsilon_s$.\\
\phantom{\textbf{Input: }} Let $\varepsilon_0$ be an arbitrary value $<\varepsilon_1$. Note that $\caC (\varepsilon _0) = \emptyset$.

\State \textbf{Initialize} $LM ~\leftarrow ~ \{ \}$

\ForAll{$k \in \{ 1, ..., s \}$}
    \ForAll{$\sigma \in \caC _X(\varepsilon_{k}) ~\backslash ~ \caC _X(\varepsilon_{k-1})$} 
    \Comment{All these cliques appear when $\varepsilon_{k-1}$$\curvearrowright \varepsilon_{k}$.}
        \State $M ~ \leftarrow ~ M \cup \{ (\sigma, \varepsilon_{k}, \infty)\}$
        \Comment{Add newly born clique (it appeared in $\varepsilon _{k}$)}
    \EndFor\\
    \ForAll {$\sigma \in \caC _X(\varepsilon_{k-1}) ~\backslash ~ \caC _X(\varepsilon_{k})$}
    \Comment{All these cliques disappear when $\varepsilon_{k-1}$$\curvearrowright$ $\varepsilon_{k}$.}
    \State Let $(\sigma, \varepsilon, \infty)$ be the existing entry in $LM$ starting with $\sigma$.
        \State $M ~\leftarrow ~ M \setminus \{(\sigma, \varepsilon, \infty)\}$
            \Comment{Remove entry}
        \State $M ~ \leftarrow ~ M \cup \{ (\sigma, \varepsilon, \varepsilon_{k})\}$
            \Comment{Add the entry updated with its death time.}
    \EndFor
\EndFor
\State $LM ~\leftarrow ~ \{ (X, \varepsilon_s, \infty)\}$ \Comment{If we want to include the infinite component}
\State Return $M$  \Comment{Return the labeled mergegram}
\end{algorithmic}
\end{algorithm}

We can also use a more space-efficient algorithm by using the information in the phylogenetic network to say when cliques are born and when they die. Here we mean the Vietoris-Rips filtration of the phylogenetic network. This stems from the fact, that we do not need a list of maximal cliques but can directly compute the contribution to the mergegram for a given critical value $\varepsilon_k$ and the maximal cliques found in the graph $G_{\varepsilon_k}$.
\begin{enumerate}
    \item Get all maximal cliques for a specific value $\varepsilon$ in the critical set.
    \item For these maximal cliques $\sigma$ compute the 
    \begin{enumerate}
        \item maximal values in $N_X(\sigma, \sigma)$, this is the birth-time $b_\sigma$ of $\sigma$. If it's the same as $\varepsilon _i$, $\sigma$ has been born at $\varepsilon_i$, it it's smaller $\sigma$ has been a maximal clique before.
        \item the minimal value of all the column-wise maxima in $N_X(\sigma, X \backslash \sigma)$, this is the death-time $d_\sigma$ of $\sigma$.
    If $b_\sigma = \varepsilon_i$, add $(b_\sigma, d_\sigma)$ to the mergegram.
    \end{enumerate}
\end{enumerate}

\begin{algorithm}
\caption{The mergegram of a phylogenetic network}\label{alg:mergegram_cliquegram_phylonetwork}
\begin{algorithmic}[1]
\State\textbf{Input}: Given a phylogenetic network $N_X$.
\State \textbf{Initialize} $LM ~\leftarrow ~ \{ \}$
\ForAll{$\varepsilon \in \{$unique values in $N_X \}$}
\State $G_{\varepsilon} $ $\leftarrow$ Build the graph $G_{\varepsilon}$ with vertex set $V(G_{\varepsilon}):=X$ \\
    \phantom{ $G_{\varepsilon}$ $\leftarrow$ }and edge-set $E(G_{\varepsilon}):=\left\{(x,x')\in X\times X\mid N_X(x,x')\leq {\varepsilon} \right\}.$ 

\ForAll{$\sigma ~ \in ~\texttt{FindMaxCliques}(G_{\varepsilon})$}
\State $b ~ \leftarrow ~ \max_{i, j \in \sigma} (N_X(i, j))$
\If{$b = {\varepsilon}$} \Comment{Maximal clique appeared at this value}
\State \textbf{}$LM ~\leftarrow ~ LM \cup \{ \left(\sigma, ~b, ~
    \max _{x \in \sigma} \min_{y \in X \backslash \sigma} N_X(x, y) \right)\}$
    \label{alg:mergegram_cliquegram_phylonetwork_birthdeath}
\EndIf
\EndFor
\EndFor
\State $LM ~\leftarrow ~ \{ (X, \varepsilon_s, \infty)\}$ \Comment{If we want to include the infinite component}
\State Return $M$  \Comment{Return the labeled mergegram}
\end{algorithmic}
\end{algorithm}

Now let us consider the complexity of these two algorithms. 
Let $M_1, M_2, \ldots, M_s$ the clique-sets corresponding to each value in the critical set $\varepsilon_1 < \ldots < \varepsilon_s$ of $\caC _X$. By the definition of critical set, the clique-sets are trivial, i.e. $\caC _X(t) = \emptyset$, for $t<\varepsilon_1$ and $\caC _X(t) = X$ for $t\geq \varepsilon_s$.
Given the ultranetwork representations of treegrams as input, the asymptotic complexity of calculating the minimizer matrix is $O(s n^2)$, since it depends linearly on the number $s$ of all different critical points of all the trees (this corresponds to the number of all different values of the entries of all the ultranetworks) and quadratic 
on the number of taxa, $n$ (since the minimizer matrix is the $n\times n$-matrix obtained by taking entry-wise minima over all the entries of all the ultranetworks of the treegrams). 

The calculation of the maximal cliques for each step in the intermediate graph $G_t$ in step $2$ is the costly step in this pipeline: 
In terms of the number of taxa, $n:=|X|$ alone as input parameter the computation is certainly costly: variants of the Bron-Kerbosch algorithms \cite{BronKerbosch73} are still the best general case running cases for finding all maximal cliques in such a graph at the different threshold steps, resulting in a worst-case runtime of $O(s3^{n/3})$ for $s$ being the size of the critical set of the cliquegram and $n$ the number of vertices in the graph. 
To be more precise, $3^{n/3}$ are the most maximal cliques which can occur at one step in the critical set. Graphs of such a structure are called Moon-Moser Graphs \cite{moon_cliques_1965}. The optimal worse-case time complexity of the Bron-Kerbosch algorithm is given by $O(3^{n/3})$. In practice, the number of maximal cliques is often-times observed to be far smaller, see \cites{tomita_maxcliques_2006, conteComplexity2021} for an extended time complexity analysis. 
In general the number of maximal cliques over all values of the critical set is $O(2^n)$ for the cliquegram, see \textit{Prop.~\ref{prop:maximalnumbercliquegram}} and thus the number for maximal faces is also $O(2^n)$ for the facegram. 
As can be seen in \textit{Tab.~\ref{tab:runtimes_long}} for the sizes of the cliquegram and facegram, the cliquegram will have a far larger number of maximal cliques when considering the join-cliquegram opposed to the join-facegram.

In terms of the set of all cliques as a complexity parameter, the computation of the mergegram of the cliquegram is done in $O(s|M|)$ where $|M|$ is the number of maximal cliques over all values of the critical set which is bounded above by $2^{|X|}$, i.e. $M := \bigcup M_i$.
Specifically, we have these bounds for the two inner loops and computing the set-differences in \textit{Alg.~\ref{alg:mergegram_cliquegram}} as 
$$O\left((s-1)2|M| + \sum_{i=1}^{s-1} 2\min (|M_i|, |M_{i+1}|)\right)
= O(s|M| + (s-1)|M|)
= O(s|M|).$$

Opposed to this, \textit{Alg.~\ref{alg:mergegram_cliquegram_phylonetwork}} has in addition to the runtime of the \texttt{FindMaxCliques} part, the look-up of the maximal values in the $|\sigma|^2$-matrix as well as the maximal-minimal computations if $\sigma$ is a new maximal clique

\par\medskip
\textbf{Space complexity}
For both algorithms we need to compute the maximal cliques of a graph on $n=|X|$ vertices. The Bron-Kerbosch algorithm uses a recursive call stack with $O(n)$ and the set of (current clique, potential nodes to add, nodes already processed) each recursion call, which can be implemented with a space complexity of $O(n)$ each time, resulting in a space usage of $O(n)$. Depending on the graph representation via adjacency matrix or adjacency list we have a space complexity of $O(n^2)$ or $O(n+m)$ (for $m$ being the number of edges in $G$). For simplification, we always use the adjacency matrix, hence $O(n^2)$ is the space complexity of the maximal clique finding algorithm for one specific $G_t$ and time $t$. 
In addition, the algorithm uses $O(|\caC_X(\varepsilon _i)|)$ space for the list of maximal cliques.

Suppose we compute all maximal cliques for each value in the critical set and save them in the sets $\caC_X(\varepsilon _i)$ as in \textit{Alg.~\ref{alg:mergegram_cliquegram}}. 
While the checks to see what maximal cliques are born or die are quick then, there is a large space complexity of $O(3*\sum_{i=1}^s\caC_X(\varepsilon _i))= O(s\caC_{X, max})$ for the list of all maximal cliques as well as $O(3*\sum_{i=1}^s M_i)) = O(sM_{max})$ for the labeled mergegram itself. Hence, the total space complexity is 
\[
O(n^2 + s\caC_{X, max} + sM_{max}) = O(n^2 + s\caC_{X, max})
\]

We can decrease this, to $2\caC_{X,max}$, since we only need to keep track of two time instances of the cliquegram. If we do not need keep track of the maximal cliques from before at all during the computation of the mergegram, as done in \textit{Alg.~\ref{alg:mergegram_cliquegram_phylonetwork}}, we can reduce the space complexity of the algorithm to be only $O(2\sum |M_i|)$ for the unlabeled mergegram. 
Furthermore, since we do not need to to save all the maximal cliques at a specific time-step due to the fact that we directly check their birth and possibly death time and only add their birth-death-tuple to the unlabeled mergegram, the total space complexity is given as  
\[
O(n^2 + 2\sum |M_i|) = O(n^2 + s M_{max}).
\]
Asymptotically there is no difference between the above three versions, because of the dominating factor which are all $O(2^n)$ in the worst case, see \textit{Prop.~\ref{prop:maximalnumbercliquegram}}.
In the actual implementation only consecutive steps for maximal cliques are saved, thus decreasing the space complexity.

\begin{remark}[Data representation]
Using optimized data-structures for the representation of the cliquegram the theoretical complexity for computing the cliquegram associated to a distance matrix is the same as for computing the Rips complex for the \emph{Critical Simplex Diagram (CSD)} in \cite{boissonnat2018CSD} where it is stated to be $O(|\sum |M_i||^{2.38})$. 
In general, the CSD is a good candidate for an efficient representation of the clique-/facegram with fast base operations.
\end{remark}

\textbf{Cliquegram of a metric space}
Given a 
metric space
with $n$ points, we can consider the $n\times n$-distance matrix of the distance of each of the points from each other. Since any such distance matrix is also a phylogenetic network (with the diagonals being zero), we can compute the associated cliquegram and its mergegram as we did above, without having to compute the minimizer matrix.

\subsection{The mergegram of a facegram and its complexity}
\label{sec:complexity}
\textbf{Join-facegram: complexity}

We show that the complexity of computing the join-facegram of a set of treegrams and then its associated mergegram invariant is
polynomial in the input parameters.

First, we need the following technical results.
\begin{proposition}
\label{prop:expl-formula for join of face-sets}
Let $\C _X^k$, $k=1,\ldots,\ell$ be a finite collection of face-set over $X$. Let $\C _X:=\bigvee_{k\in \{1,\ldots, \ell\}} \C _X^k$. Then, 
$$\C _X=\left\{\sigma\subset \bigcup_{k\in\{1,\ldots,\ell\}}\C _X^k\mid \sigma\nsubseteq \tau, \text{ for all }\tau \in\left( \bigcup_{k\in\{1,\ldots,\ell\}}\C _X^k\right)\setminus \{\sigma\}\right\}.$$ 
In other words, the join of the $\C _X^k$'s consist of the set of maximal elements in the union $\bigcup_{k\in\{1,\ldots,\ell\}}\C _X^k$.
In particular, $\C _X\subset \bigcup_{k\in\{1,\ldots,\ell\}}\C _X^k$.
\end{proposition}

\begin{proof}
Because of the lattice isomorphism between simplicial complexes and face-set in \textit{Prop.~\ref{prop:clut-simp}}, the proof follows by recasting the problem in the setting of simplicial complexes; see Rem.~\ref{rem:joinstwo}. Hence, the maximal faces of a union of complexes are computed by the above formula.
\end{proof}
\begin{corollary}
Given a set of $\ell\geq2$ face-sets over $X$ $\{C^1_X,\ldots,C^\ell_X\}$, the join-face-set $\C _X:=\bigvee_{i\in\{1,\ldots,\ell\}} C^i_X$ can be computed in time at most $O\left(\left(\sum_{i=1}^{\ell}|C^i_X|\right)^2\max |C^i_X|\right)$ .
\end{corollary}
\begin{proof}
Consider $C \in \bigcup_{k\in\{1,\ldots,\ell\}}\C _X^k$. In the construction of $\bigvee_{k\in \{1,\ldots, \ell\}}$ we need to check if $C$ is not contained in another element in $\bigcup_{k\in\{1,\ldots,\ell\}}\C _X^k$. To check this, we need to perform $\sum_{i=1}^{\ell}|C^i_X|$ comparison of $C$ to other elements of the union. The cost of each comparison is bounded by $max |C^i_X|$. Since there are $\sum_{i=1}^{\ell}|C^i_X|$ elements $C$ for which the check need to be done, the computational complexity of the procedure is $O\left(\left(\sum_{i=1}^{\ell}|C^i_X|\right)^2\max |C^i_X|\right)$. 
\end{proof}
\begin{theorem}
\label{thm:complexity}
Given a set $\{U^1_X,\ldots,U_X^\ell\}$ of $\ell$ symmetric ultranetworks over a set of taxa $X=\{x_1,\ldots,x_n\}$, 
(i) the join-facegram $\caF _X:=\vee_{i=1}^\ell \caT _X^i$ of their associated treegrams $\caT _X^i:=\varphi(U_X^i)$, $i=1,\ldots,\ell$, (ii) the labeled mergegram $\mathbf{mgm}^*(\caF _X)$, (iii) the Reeb graph and (iv) the unlabeled mergegram $\mathbf{mgm}(\caC _X)$ of $\caC _X$, can all be computed in time at most $O(n^4\cdot\ell^2)$. 
\end{theorem}
\begin{proof}
Since the mergegram can directly be obtained by the labeled mergegram (by projecting its elements to the second coordinate), and since the labeled mergegram is complete invariant for facegrams, among (i), (ii) and (iv), it suffices to show our complexity claim only for computing the labeled mergegram. Also for (iii) as well: computing the Reeb graph of a facegram is not more complicated than the descriptive complexity of facegram and of the labeled mergegram, having an upper bound on the computational complexity of the labeled mergegram, will yield an upper bound on the complexity of computing the Reeb graph.  

\begin{figure}
    \centering
    \includegraphics[width=0.4\textwidth]{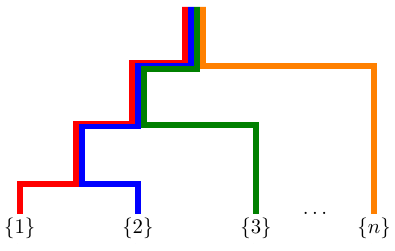}
    \caption{{\bf Worst case scenario for a tree in terms of run-time in the algorithm.} 
    Example of a tree with the maximal number of different height values which are $2n-1$ for $n$ taxa/leaves.}
    \label{fig:proof_trees}
\end{figure}

     First, observe that for each treegram $\caT _X^i$, viewed as a facegram, the number of different simplices in it, will be no more than the number of its edges, which is one number less than the number of its vertices. Because of the tree structure in $\caT _X^i$, the number of its vertices is upper bounded by: the product of 
     \begin{itemize}
     \item[(i)] the number of different distance values in the matrix $U_X^i$, which is no larger than $2n-1$: because of the tree structure of $\caT ^i_X$, there are at most $2n-1$ different height values for the tree; the value $n$ can be attained when the treegram is as in Fig.~\ref{fig:proof_trees} (for this type of treegrams, there are $n$ leaf nodes and there are $(n-1)$ tree nodes, so $(2n-1)$ in total), and 
     \item[(ii)] the number of all possible vertices of a fixed height, which is no larger than $n$: because vertices of fixed height are at most as many as the cardinality of the level set of a fixed height which is a partition of the taxa set $X$ and the cardinality of it is at most $n$ because the finest partition is $\{\{x\}\mid x\in X\}$ which has cardinality same as $|X|=n$.  
     \end{itemize}
     Hence, the number of all different simplices in each treegram is at most $n\cdot(2n-1)$. 
     Now, consider the set $S$ of all different simplices appearing in the treegrams $\{\caT  _X^i\}_{i=1}^\ell$ (viewed as facegrams). By the subsequent argument $|S|\leq n\cdot (2n-1)\cdot\ell$.

     The join–filtration $\F _X$ of the tree filtrations associated to the ultrametrics
 $U^k_X$, $k=1,\ldots,\ell$, is given point-wise for every $\sigma\in\mathbf{pow}(X)$ by the formula:
     $$\F _X(\sigma)=\min_{k\in\{1,\ldots,\ell\}}\left(\max_{x_i,x_j\in\sigma}U^k_X(x_i,x_j)\right).$$
   By the second conclusion of \textit{Prop.~\ref{prop:expl-formula for join of face-sets}}, the definition of $S$ and that $\caF _X=\bigvee_{i=1}^{\ell} \caT ^i_X$, we obtain that $S\supset \mathbf{s}(\C _X)$. By \textit{Prop.~\ref{prop:mgm-formula}} and \textit{Cor.~\ref{cor:S-cor}}, the mergegram can be computed by using the set $S$ of simplices. Hence, there are at most $|S|$ non-empty intervals to compute to decide if they are non-empty, which is $O(n^2\cdot\ell)$. Now, once we have a simplex $\sigma\in S$, to decide if 
   $$\min_{\sigma\nsubseteq\tau\in S}\F _X(\tau)> \F _X(\sigma)$$
     takes $|S|$-steps (taking all possible $\tau\in S$) which is again $O(n^2\cdot\ell)$. 
     Therefore, computing the labeled mergegram $\mathbf{mgm}^*(\caF _X)$ and thus the join-facegram $\caF _X$ from the given set of ultranetworks, take time at most $O(n^4\cdot\ell^2)$.
\end{proof}

\begin{remark}
    The intuition behind \textit{Thm.~\ref{thm:complexity}}, is that by \textit{Rem.~\ref{rem:joinstwo}} the information of all maximal simplices in the join of the facegrams is already given by the collection of maximal simplices of the facegrams of the trees. 
This is not the case for cliquegrams where we need to find maximal cliques for the different filtration levels, either in the case of constructing the join-cliquegram from (cliquegrams of) the treegrams or when squashing the loops in the facegram. 
Thus, the cliquegram construction and therefore the calculation of the mergegram of cliquegram might need exponential time.
\end{remark}

\textbf{Join-facegram: algorithms}~\\
In the following, we list two algorithms for the mergegram of the join-facegram. 
Algorithm \ref{alg:join_tree_facegram} is still close to the actual implementation; we could just say: make a large union of all $(\sigma, I_\Sigma)$ and then just take the smallest filtration value of $\sigma$ over all such pairs. 
A bit more precisely,
the algorithm to calculate the labeled mergegram first collects all different pairs $(\sigma , I_\sigma)$ for a fixed $\sigma$ in all the labeled mergegrams of the treegrams to get the minimal birth time $\min _{i} F_{\caT ^i}(\sigma)$ and minimal death time $\min_{i} F_{\caT ^i}(\tau_i)$ for $\tau_i \supset \sigma$. Then the labeled mergegram for the join of the treegrams is directly given by this collection.
\begin{algorithm}
\caption{The mergegram of the join-facegram of the facegram of treegram (Version 1)}
\label{alg:join_tree_facegram}
\begin{algorithmic}[1]
\State Let $\mathbf{mgm}^{*}(\caT ^i_X)$ be the labeled mergegrams of the treegrams for $i \in \{ 1, \ldots , l\}$.
\State Let $S = \emptyset$,~$LM = \emptyset$
\ForAll{$i \in \{ 1, \ldots , l\}$}
    \ForAll{$(\sigma, I_\sigma) \in \mathbf{mgm}^{*}(\caT ^i_X)$}
        \If{$\sigma \not\in S$}
        
            \State $S \leftarrow S \cup \{ \sigma \}$
            \State $LM \leftarrow LM \cup \{ (\sigma, I_\sigma) \}$
        
        \Else
            
            \State Let $(\sigma, [b, d)])$ be the labeled interval in $L$ and $I_\sigma = [b_\sigma, d_\sigma)$.
            \State $LM \leftarrow LM \backslash \{(\sigma, I)\} \cup \{(\sigma, [\min(b, b_\sigma), \min(d, d_\sigma))\}$ \Comment{{Update entry}}
        \EndIf
    \EndFor
\EndFor

\ForAll{$(\sigma, [b_\sigma, d_\sigma)) \in LM$} \Comment{Update }
\State $b \leftarrow \min \{ b_\tau \mid \tau \supset \sigma \text{ and }(\tau, (b_\tau, d_\tau)) \in LM\}$
\If{$b_\sigma \geq b $} \Comment{Remove non-maximal faces in $LM$}
    \State $LM \leftarrow LM \backslash \{ (\sigma, [b_\sigma, d_\sigma)) \}$
\EndIf
\If{$b < d_\sigma$} \Comment{Update to earlier death-time}
    \State $LM \leftarrow LM \backslash \{ (\sigma, [b_\sigma, d_\sigma)) \} \cup \{(\sigma, [b_\sigma, b)) \}$
\EndIf
\EndFor

\State{Return $LM$}
\Comment{Return the labeled mergegram}
\end{algorithmic}
\end{algorithm}
The algorithm relies on the fact that for the associated ultranetworks for the tree distances for each $(\sigma, I_\sigma)$ in a treegram, the distance between taxa in the face-set $\sigma$ are smaller than the distances between points in $\sigma$ and outside of it, that is
\[ 
    \max_{y \in \sigma} \max_{x \in \sigma} U^i_X(x,y) 
    \leq 
    \min_{y \not\in \sigma} \max_{x \in \sigma} U^i_X(x,y),
\]
since this would otherwise violate the condition of the lattice structure of $\mathbf{Facegm}(X)$.
Therefore, to get the mergegram of the join-facegram, we have to consider all the labeled intervals in all the treegrams.
\begin{algorithm}
\caption{The mergegram of the join-facegram of a set of treegrams (Version 2)}
\label{alg:join_tree_facegram_second}
\begin{algorithmic}[1]
\State Let $\mathbf{mgm}^{*}(\caT ^i_X)$ be the labeled mergegrams of the treegrams for $i \in T:=\{ 1, \ldots , l\}$ and $U^i_X$ their corresponding ultranetworks.
\State and $LM \leftarrow \emptyset$ \Comment{Labeled intervals of the joined facegram.}
\ForAll{$\sigma \in \{\tau ~| ~\exists i\in T ~\exists I\text{ interval s.t. } (\tau, I) \in \mathbf{mgm}^{*}(\caT ^i_X)\}$}

        \State $b \leftarrow \min_{i \in T} \max_{y \in \sigma} \max_{x \in \sigma} U^i_X(x,y)$
        \Comment{get birth time}
        \State $d \leftarrow \min_{i \in T} \min_{y \not\in \sigma} \max_{x \in \sigma} U^i_X(x,y)$
        \Comment{get death time}
        \vspace{0.5\baselineskip}
        \If{$b \neq d$}\Comment{By definition of the treegrams, we always have $b\leq d$.}
            \State $LM \leftarrow LM \cup \{ (\sigma, ~[b, d)~) \}$
        \EndIf
\EndFor
\State Return $LM$
\Comment{Return the labeled mergegram}
\end{algorithmic}
\end{algorithm}

In general, the possibly computationally expensive part of the computation of the mergegram of a facegram is the construction / size of the facegram itself. 
In the case of the facegram arising as the join of treegrams, all the necessary information is already encoded in the collection of all the labeled mergegrams of the treegrams, and we can compute it in polynomial time. 
Furthermore, for a given facegram representation, we can calculate the mergegram of it in at most $O(M^2)$ time, where $M$ is the number of different faces.

\begin{remark}
    It should be noted that for the mergegram each simplex can only contribute to at most one non-trivial interval to the mergegram. 
\textit{Alg.~\ref{alg:join_tree_facegram}} takes this into account by explicitly matching the simplices with an appropriate interval. 
In \textit{Alg.~\ref{alg:join_tree_facegram_second}} this matching is done implicitly using just the appropriate minimal/maximal values in the ultramatrices of the treegrams. This leads to major speed-ups in the algorithm, in particular, since we do not need a separate check for maximality in the end as in \textit{Alg.~\ref{alg:join_tree_facegram}}.
\end{remark}

\begin{remark}{Space complexity.}
    By the comment in the proof for \textit{Thm.~\ref{thm:complexity}} we have that the maximal number of different simplices appearing in the labeled mergegram of a treegram is $2|X|-1$. 
    Hence, the the join-mergegram contains at most $\min(l(2n-1), 2^n)$ different birth-death tuples. 
    Furthermore, during computation, we only need to save all actual maximal faces in the different treegrams as well as the minimizer matrix, which is the same number as tuples. 
    Hence, 
    \[
    O(2\min(l(2n-1), 2^n) + n^2) 
    .
    \]
\end{remark}

\textbf{Metric space reconstruction via facegram}
Instead of representing a metric space as a cliquegram, we can also use the facegram representation. 
Opposed to the cliquegram - which gives the same result as the Rips-complex in terms of maximal cliques and their filtration values - we have the flexibility of choosing any filtration on the metric space to build our facegram.

\begin{enumerate}
\item Given a finite metric space $M$ (e.g.~a point cloud) with $n$ points, label these with $1, 2, \ldots, n$.
\item Choose a filtration on the metric space, e.g. the Alpha complex \cite{edelsbrunner2022computational}. 
\item Sort the different filtration values and obtain the maximal faces for each of the filtration values, to obtain the facegram of the metric space.
\item Compute the mergegram of that facegram.
\end{enumerate}

For an implementation of these steps, see \cite{ourgitcode} for a more detailed explanation. 
In general, the hurdles for the implementation are to find a quick implementation for the desired filtration which lends itself to fast construction of the facegram. Obtaining the mergegram from the facegram can then be done in the same fashion as for the cliquegram using Algorithm \ref{alg:mergegram_cliquegram}.

\subsection{Experiments with datasets}\label{subsec:experiments}
We examine the mergegram of the join-facegram, first to an artificial dataset and then to a benchmark biological dataset.

\textbf{Artificial dataset}

 Let us consider two randomly chosen dendrograms given in Fig. \ref{fig:example_dendrogram}. There exist several ways to construct random dendrograms, one is to choose a random distance matrix and use UPGMA (unweighted pair group method with arithmetic mean) hierarchical clustering method. For random treegrams other hierarchical clustering algorithms like single-linkage clustering can be used. 
 The respective mergegrams of the dendrograms constructed are shown in \textit{Fig.~\ref{fig:example_mergegram_same})}. 

\begin{figure}[ht!]
    \centering
    \includegraphics[width=0.7\textwidth]{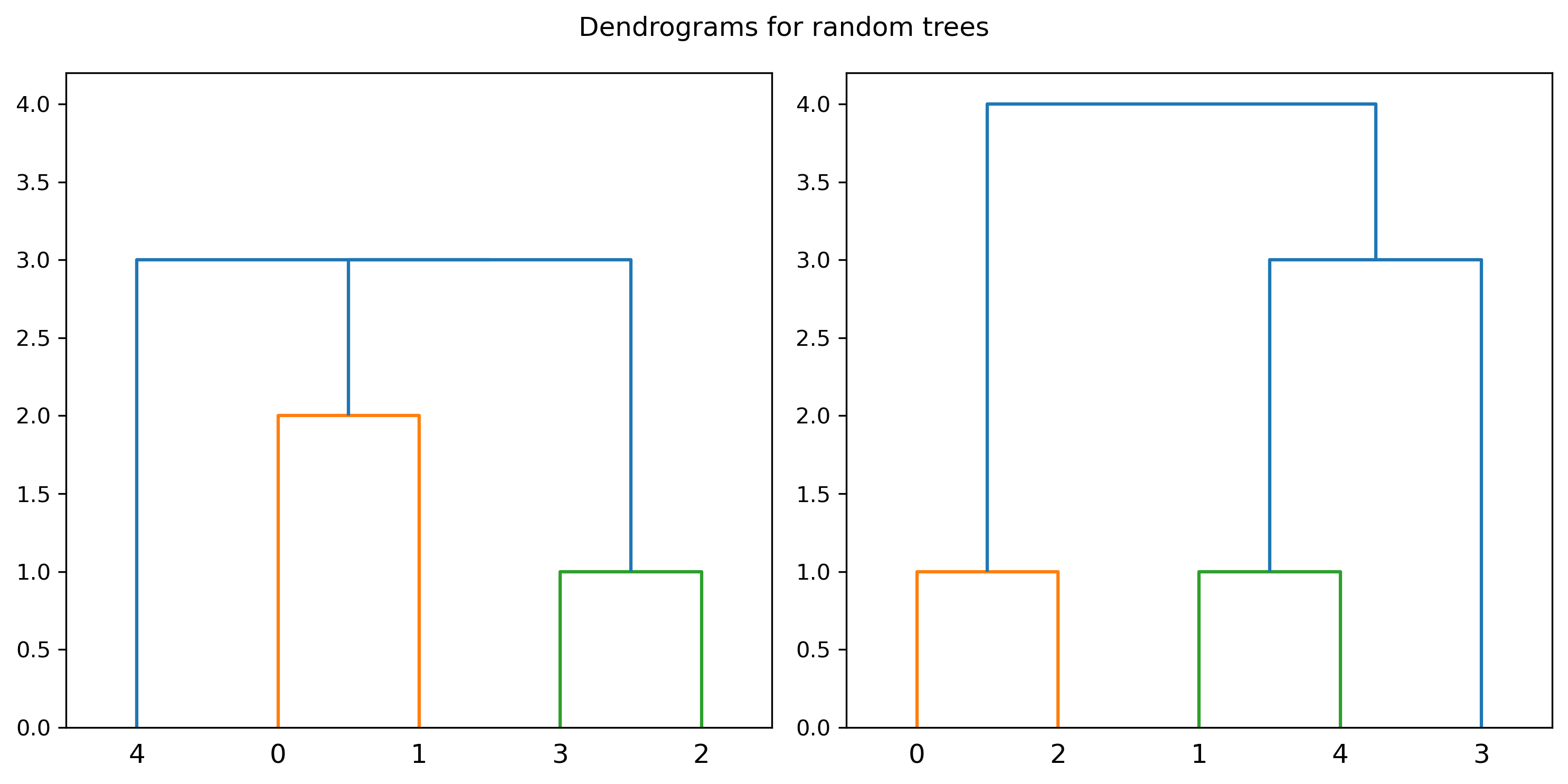}
    \caption{{\bf Dendrograms of random trees.} 
    The dendrograms of two random trees. The coloring is corresponding to certain clusters by \textsc{scipy}.
    }
    \label{fig:example_dendrogram}
\end{figure}

\begin{figure}
    \centering
    \includegraphics[width=0.7\textwidth]{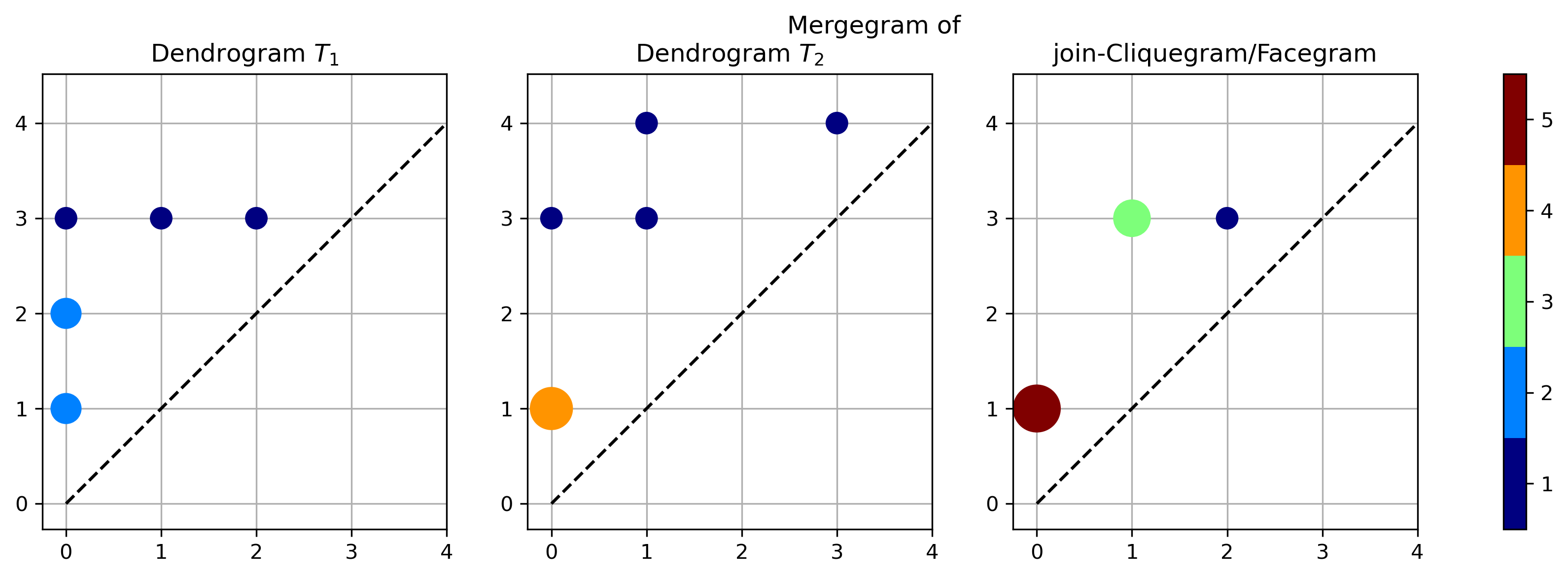}
    \caption{{\bf Mergegrams of random trees.} 
    The mergegrams of the two random trees (left and middle). Treegrams viewed as either cliquegrams or facegrams always give the same mergegram. Coloring as well as the size of the points show the multiplicity of each point. Mergegram of the join-cliquegram and of the join-facegram (right) of the two treegrams are the same in this case. Coloring is done by the multiplicity of each point.}
    \label{fig:example_mergegram_same}
\end{figure}

For these two dendrograms we can now compute their join-cliquegram as well as their join-facegram, in the respective lattices of cliquegrams and facegrams. In this case, both the join of cliquegrams and the join of facegrams coincide: we get three different points in the mergegram all with different multiplicities. 
On the other hand, the join-cliquegram and join-facegram may differ, depending on the subpartitions. An example are the dendrograms in Fig. \ref{fig:example_differentDendrograms}. The mergegrams of the treegrams and the distinct mergegrams of the join-cliquegram and join-facegram are shown in \ref{fig:example_differentTreegrams} and \ref{fig:example_differentMergegram} respectively.

\begin{figure}
    \centering
    \includegraphics[width=0.7\textwidth]{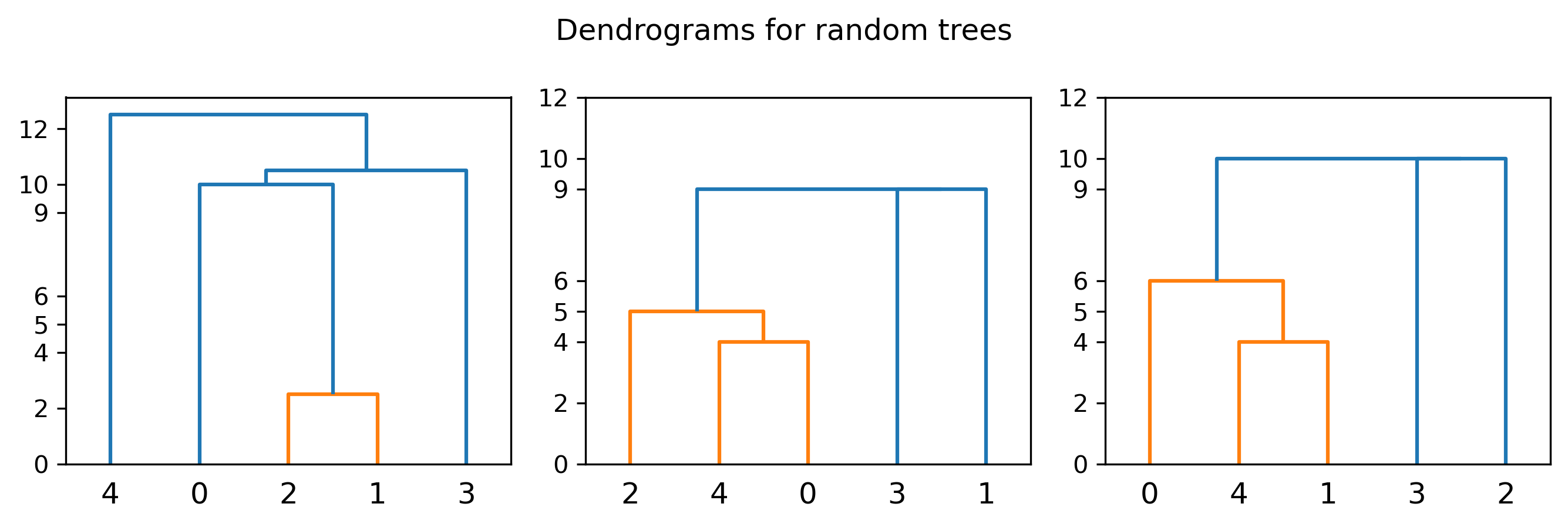}
    \caption{{\bf Three different dendrograms for a specific dataset.} 
    Three different dendrograms on the same taxa set.}
    \label{fig:example_differentDendrograms}
\end{figure}
\begin{figure}
    \centering
    \includegraphics[width=0.7\textwidth]{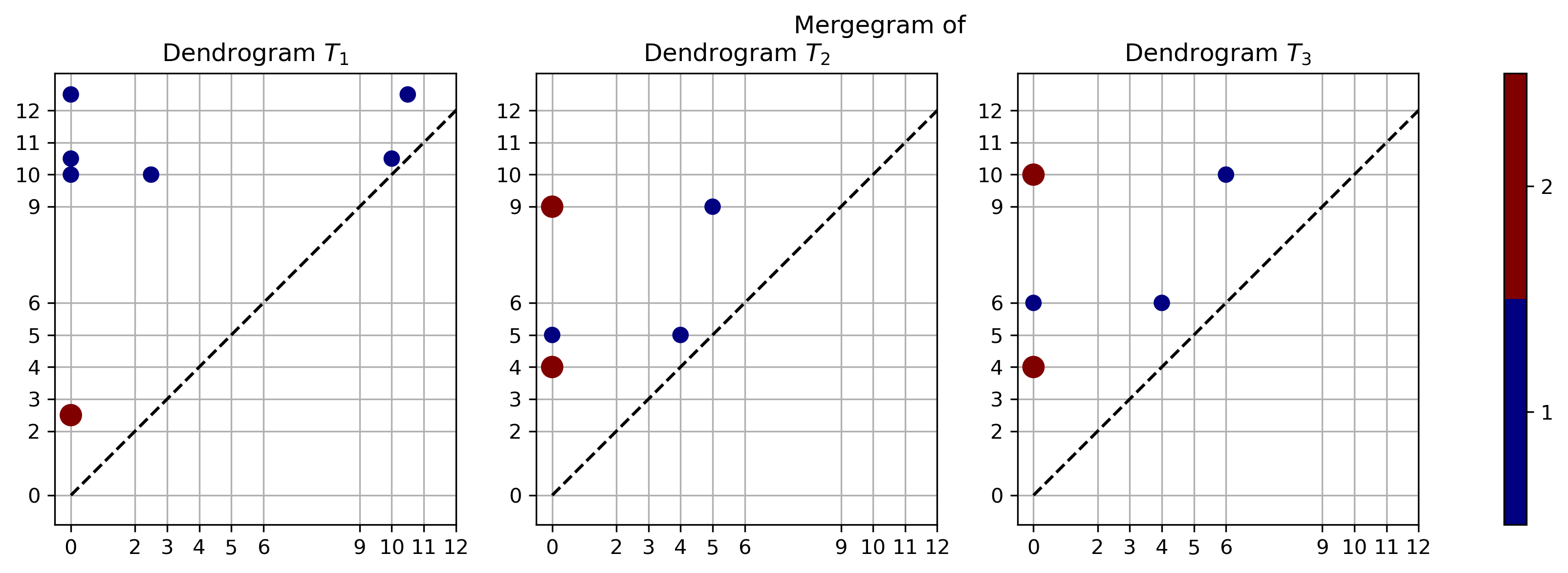}
    \caption{{\bf Mergegrams for three treegrams.} The mergegrams of the treegrams shown in \ref{fig:example_differentDendrograms}.}
    \label{fig:example_differentTreegrams}
\end{figure}

While the birth-/death-times are the same for the singletons for both cliquegram and facegram in Fig. \ref{fig:example_differentMergegram}, there is one more point in the mergegram of the join-cliquegram while the lifespans of these points are less (or equal) to the ones in the join-facegram.

\begin{figure}
    \centering
    \includegraphics[width=0.7\textwidth]{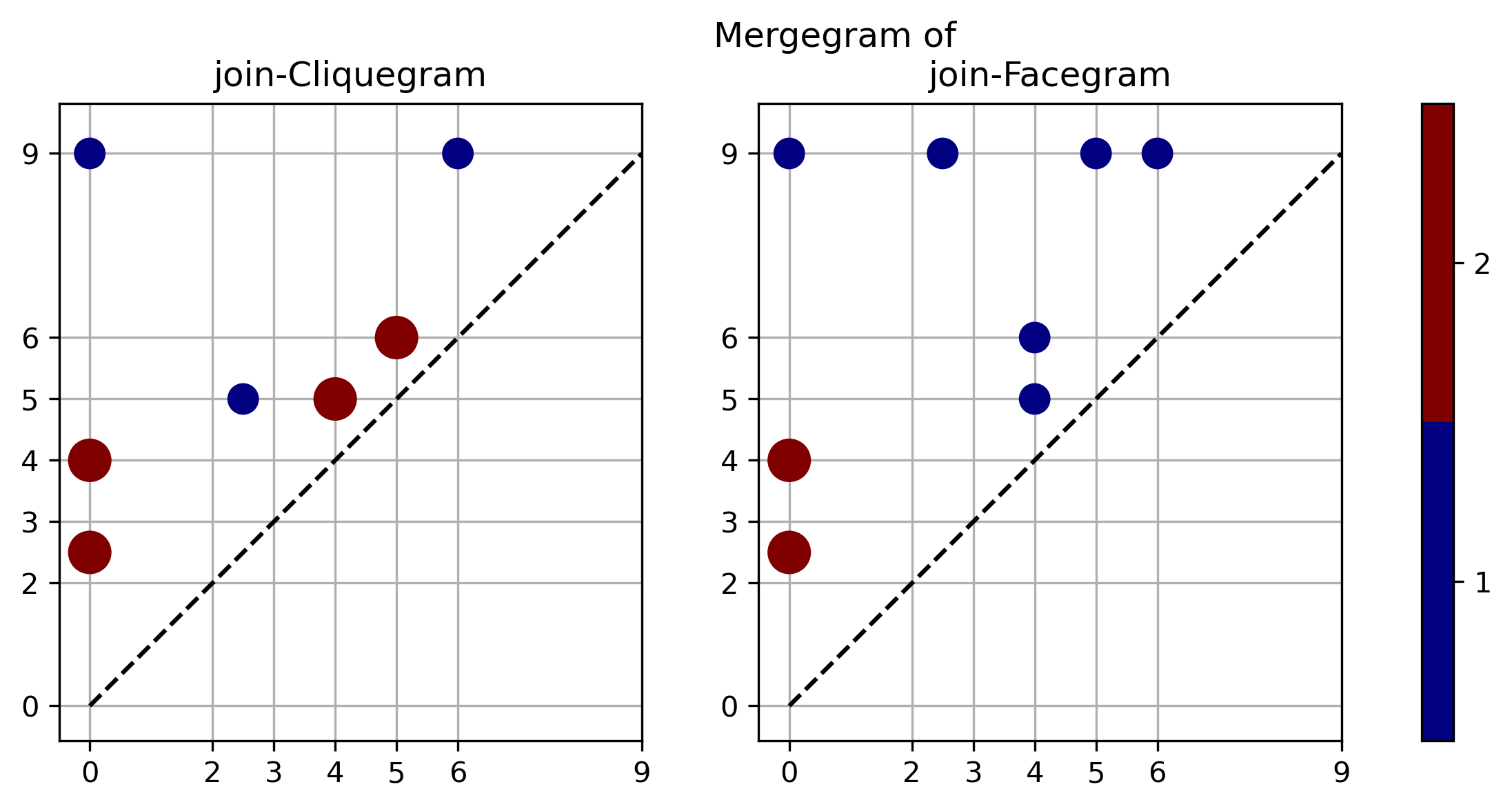}
    \caption{{\bf The Mergegram of the join in the cliquegram and facegram setting.} The mergegrams of the joins of the treegrams in \ref{fig:example_differentDendrograms}. We can see that the mergegrams are different for the cliquegram and facegram setting.}
    \label{fig:example_differentMergegram}
\end{figure}

\par\medskip
\textbf{Biological data set}
In the following, we use some trees coming from real phylogenetic trees in \cite{singhal2021congruence} to compute their join in the cliquegram and facegram settings. 
We have selected $68$ common taxa for $161$ dendrograms, two of which are shown in Fig. \ref{fig:realdata_dendrogram}.

\begin{table}
\centering
\begin{tabular}{cccc}
\bfseries mean & \bfseries standard deviation & \bfseries minimum & \bfseries maximum \\
\midrule
0.3499 & 0.1110 & 0.1117 & 0.9120
\end{tabular}
\caption{Summary statistics of the pairwise bottleneck distances between the mergegrams of different treegrams.}
\label{tab:stats_for_realtrees}
\end{table}

The mergegrams of the join-cliquegram and join-facegram only differ by 4 points which all have low persistence. 
Furthermore, there are only 3 points which have persistence (or lifetime) greater than 1, they are the one associated to the singletons 67, 66 as well as by the cluster formed by all the other taxa which is killed by $X$. Hence, we can directly see that those two taxa are outliers in all of the trees. 
The taxa correspond to UMMZ\_200325 and MVZ\_Herp\_189994 respectively. 
In \textit{Fig.~\ref{fig:realdata_dendrogram}}, an example dendrogram is shown where the two outlier taxa are marked red. In the plot of the mergegram of the join-facegram we exclude the three high persistence points, see \textit{Fig.~\ref{fig:readldata_mergegram}}. 

\begin{figure}
    \centering
    \includegraphics[width=1\textwidth]{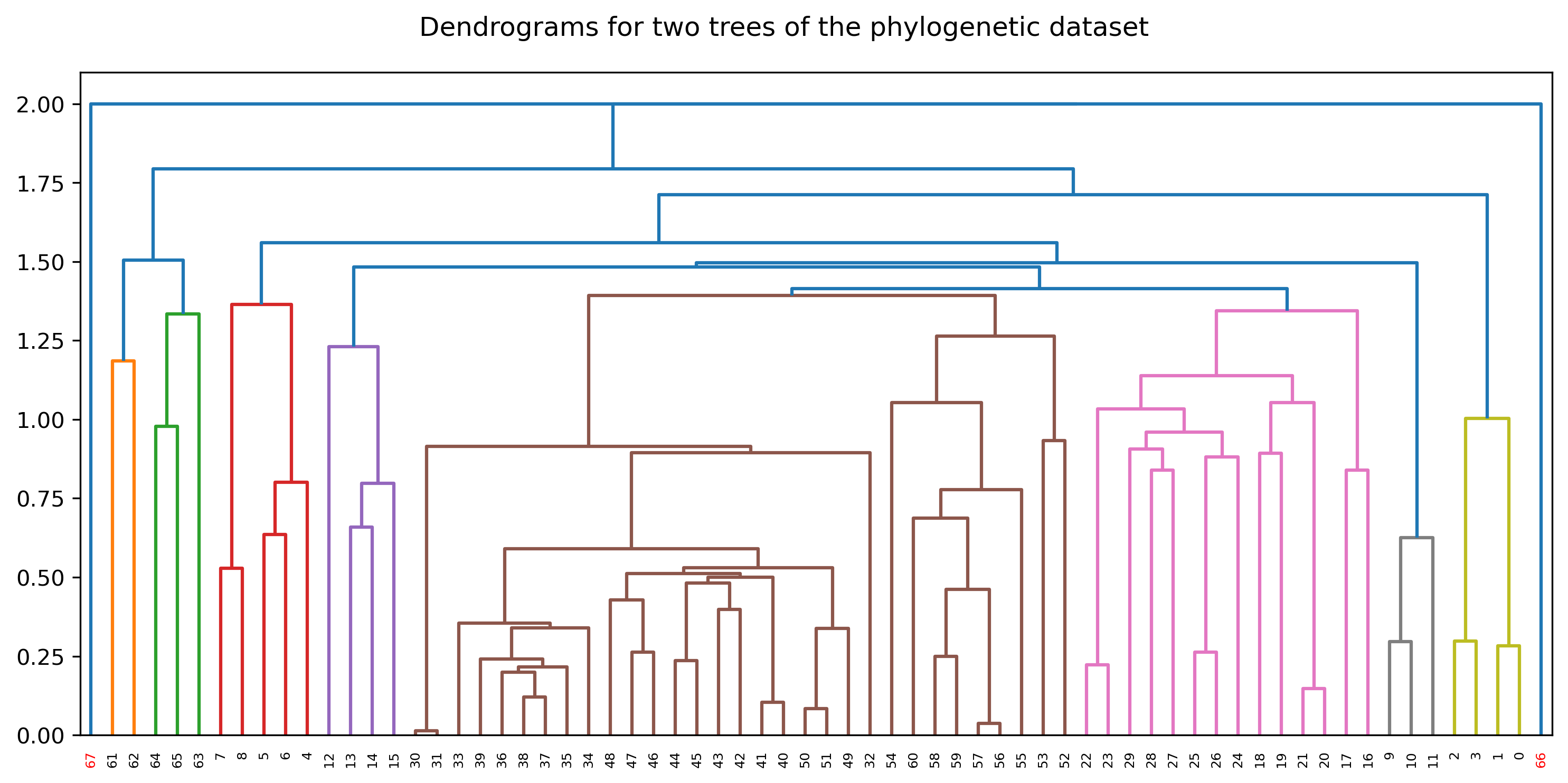}
    \caption{{\bf Example dendrogram for phylogentic tree dataset.} 
    The dendrograms of one trees from the 161 different trees available for this dataset. The two outlier taxa with numbers 66 and 67 are marked in red.}
    \label{fig:realdata_dendrogram}
\end{figure}

\begin{figure}
    \centering
    \includegraphics[width=0.5\textwidth]{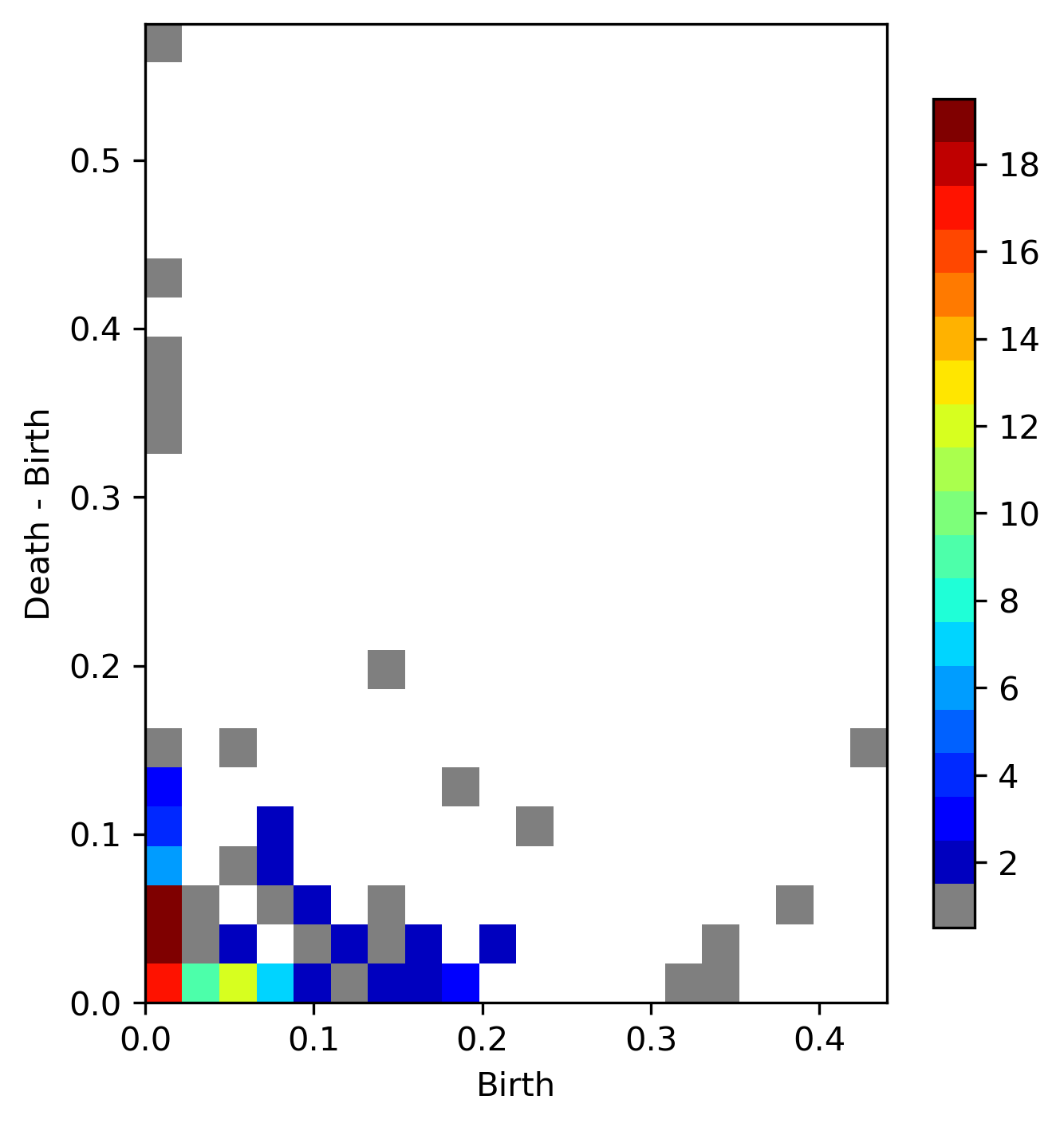}
    \caption{{\bf Mergegram of the join-facegram for phylogenetic tree dataset.}
    Heatmap of the 2d-Histogram plot taken for tuples $(b, d-b)$ for $b$ birth and $d$ death pairs of the original persistence diagram of the mergegrams of the join-facegram. We excluded three points in the diagram which correspond to the singletons 66, 67 as well as the cluster formed by all taxa not being 66, 67 in this plot.}
    \label{fig:readldata_mergegram}
\end{figure}

\par\medskip
\textbf{Bottleneck distance of mergegrams}
As mentioned before, we can interpret the mergegram as a persistence diagram and can therefore use distances like the bottleneck distance to compare two mergegrams with each other. 
To highlight one further property of the join of cliquegrams and facegrams, we consider the following setting:

Given a network with $21$ spanning trees over a common taxa set $X = \{ x_1, \ldots, x_n\}$ of $n$ taxa and their treegrams $\caT _X^1, \ldots , \caT _X^{21}$, compute the join-cliquegrams $\caC _X^k = \bigvee _{i=1}^{k} \caT _X^i, $
where $k = 1, \ldots , 21$. 
Now compute the bottleneck distances $d_{\mathrm{B}}$ between the mergegrams of these join spans with the mergegram of the join of all treegrams, i.e.
$$ d_{\mathrm{B}} (\mathbf{mgm}(\caC _X^k), \mathbf{mgm}(\caC _X^{21})).$$

Analogously, we can do the same construction for the join-facegram. 
Fig. \ref{fig:bottleneckdistance} shows the plots of these distances for different numbers of leaves of the network having $21$ spanning trees each. 
The bottleneck distances decrease until we reach a distance of zero at the end. This highlights the property of the join again of being the smallest cliquegram/facegram encoding the information given in the trees.

\begin{figure}[ht]
\centering
\includegraphics[width=1\textwidth]{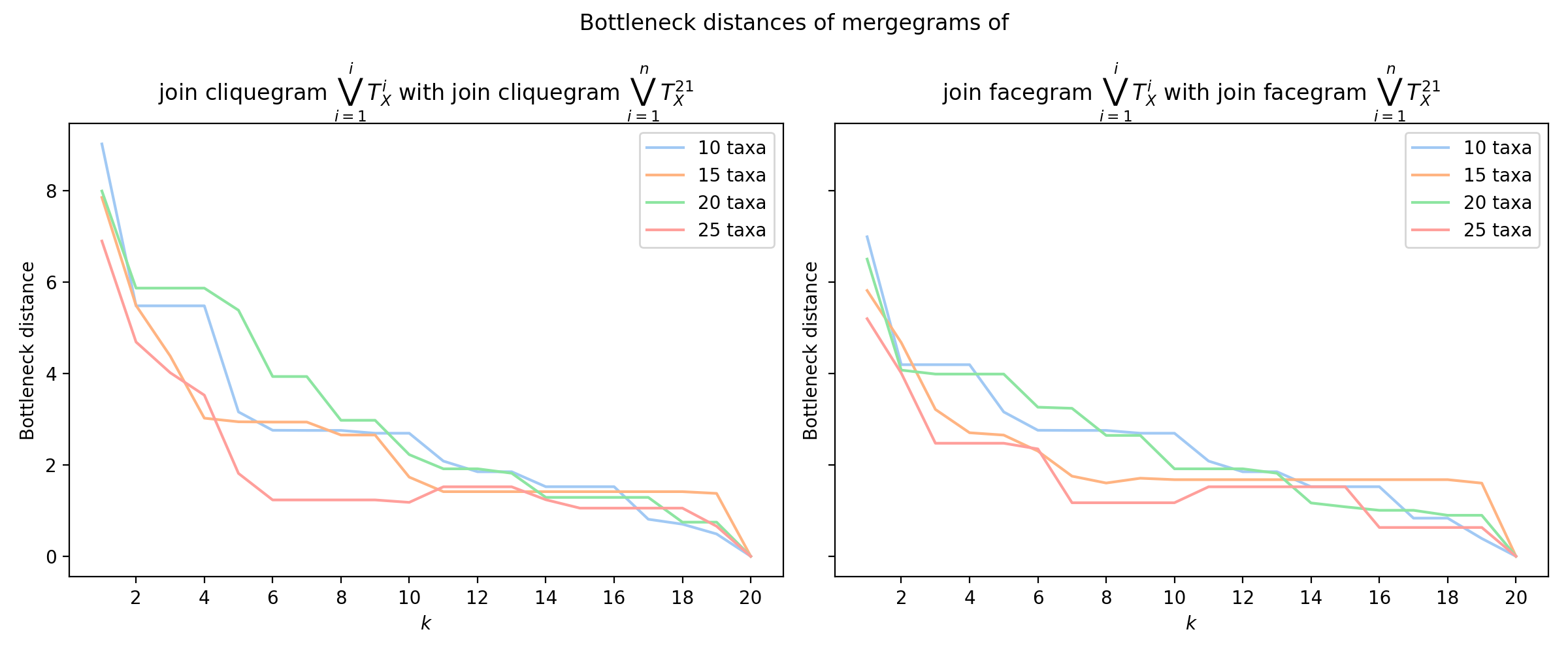}
\caption{{\bf Comparison between the joins of a subset and the all trees with bottleneck distance for the cliquegram and facegram setting.} 
Progression of bottleneck distances between the mergegram of the join of the first $n$ treegrams and the mergegram of the join of all 21 treegrams for different number of taxa for the case of cliquegrams and facegrams.}
\label{fig:bottleneckdistance}
\end{figure}

\par\medskip
\textbf{Runtimes}

To give a better intuition, we list the running times for the computations of the mergegram of the join-cliquegram and join-facegram for different datasets of treegrams in \textit{Tab.~\ref{tab:runtimes_short}}. The running times for the join-facegram are split into computation of the labeled mergegrams for the treegrams as well as computing the mergegram of the join-facegram with that information.

The datasets for the runtime computations are picked in increasing complexity of the resulting minimizer matrix, i.e. how many different new edges and thus new cliques appear when. The datasets consist of the phylogenetic trees from \cite{singhal2021congruence} (phylo. trees), as well as randomly generated trees whose treed distances are perturbations from each other (similar Trees), a random point cloud where the the points are in 5 clusters together with 10 random points and we compute a distance matrix for which we take out one tree per random point cloud using complete hierachical clustering (trees (Blobs)) as well as completely random trees (random trees). 
For the latter two datasets we consider different number of taxa as well as different numbers of trees. The larger table of this data can be found in the appendix \textit{Tab.~\ref{tab:runtimes_long}}. For a more thorough discussion of runtimes depend on some of the characteristics of the given data, see the code repository in \cite{ourgitcode}.

Computations have been done on a Apple M1 Pro (32GB RAM , 2021).
\begin{table}
\centering
\begin{tabular}{lcc|rrr}
\toprule
\textbf{dataset} & \textbf{number} & \textbf{number} & \textbf{cliquegram} & \textbf{treegrams} & \textbf{facegram} \\
 & \textbf{trees} & \textbf{taxa}  & \textbf{time (in s)} & \textbf{time (in s)} & \textbf{time (in s)}\\
\midrule
similar Trees & 50 & \multirow{2}{*}{40} & 0.03 & 0.07 & 0.03\\
 & 100 & & 0.03 & 0.14 & 0.09 \\
 \cline{2-6}
 & 50 & \multirow{2}{*}{60} & 14.94 & 0.12 & 0.05 \\
 & 100 & & 4.95 & 0.22 & 0.16 \\
 \hline
trees (Blobs) & 50 & \multirow{2}{*}{40} & 0.66 & 0.07 & 0.04 \\
 & 100 & & 0.80 & 0.13 & 0.10 \\
  \cline{2-6}
 & 50  & \multirow{2}{*}{60} & 1.99 & 0.10 & 0.10 \\
 & 100 & & 22.28 & 0.22 & 0.23 \\
 \hline
random trees & 50 & \multirow{2}{*}{40} & 1.90& 0.06 & 0.04 \\
& 100 & & 3.03 & 0.13 & 0.15 \\
 \cline{2-6}
& 50 & \multirow{2}{*}{60} & 221.89 & 0.10 & 0.09 \\
& 100 & & 241.43 & 0.21 & 0.30 \\
 \hline
phylo. Trees & 160 & 68 & 0.04& 0.49 & 0.02 \\
\bottomrule
\end{tabular}
\caption{Table listing the best runtime of the different algorithms provided to compute the mergegrams of the trees as well as the join-cliquegram and join-facegram of these trees for different datasets and number of trees and number of taxas. For total runtime of the computation of the mergegram of the join-facegram the time for the computation of all of the labeled mergegrams of the treegrams need to be added to the stated runtime.}
\label{tab:runtimes_short}
\end{table}

As stated in the table, we can see that increasing the number of trees influences the computations, but the number of taxa is by far the larger influence. In particular for the cliquegram, the number of trees only change the distancematrix in the form of the minimizer matrix slightly. As can be seen, the join-facegram together with the treegrams runtime are far below the runtimes for the cliquegrams for increasing number of taxa. The cliquegram is only quicker if we have a large number of trees with only slight differences as can be seen for the phylogenetic tree dataset as well as th similar tree dataset. The reason for the far increased runtime for 50 trees and 60 taxa in similar trees can be understood, when considering the number of unique values in the minimizer matrix, which for this random generation of trees is greater than for the case of 100 trees and 60 taxa, see \textit{Tab.~\ref{tab:runtimes_long}}.

\section{Future work}

Due to the versatility of the concepts introduced in this work, there are several avenues which warrant further investigation. In the following, we highlight some possible next steps. 

To highlight the usefulness of the methods in a more comprehensive way, an extensive analysis and comparison with current methods should be done for problems in biology as well as for general metric spaces highlighting where clique-and facegrams and their invariants fit in. 

Given a family of treegrams $\{ \caT^i _{X_i}\}_{i\in I}$ over different finite taxa sets $X_i$, we can define a join-facegram $\vee \caT^i_X$ via $X=\bigcup_{i\in I}X_i$ as well as extending each $\caT^i_{X_i} : \mathbf{pow}(X_i) \rightarrow \R$ to $\mathbf{pow}(X) \rightarrow \R$ by extending with a proper constant value. 
This constant value depends on the operation we are considering, for example for the join a candidate might be $\max _{i\in I} \max_{\sigma \subseteq X_i} \caT^i_{X_i}(\sigma)$.

As shown, when given a family of treegrams we can think of them as "points" in the lattice of cliquegrams or facegrams. 
This means that we can also cluster these treegrams using, for example, median-clustering by considering the cluster representatives to be the join of the treegrams in the respective clusters. 
Then we can use the bottleneck distance of these cluster representatives to compare them. 
In general and for applications such as these, other notions of distance than the bottleneck distance, like Wasserstein distances can be examined. 

The running times for computing the cliquegram and its mergegram rely on the exponential running time of the maximal clique finding algorithms. It should be further studied if we can improve the running times in practice by exploiting the fact that we consider all maximal cliques which appear during a filtration. 
Hence, instead of finding all maximal cliques in a graph extend individual taxa and find the maximal cliques containing these taxa for all the different instances during the (VR-)filtration, see e.g.~\cite{zomorodian2010}. 
Due to the critical importance of this for the time complexity, this directions warrant further research. 
Additionally, a parallel framework as well as approximating algorithms might also help to improve runtimes substantially. 

Data representations such as the Critical Simplex Diagrams allow us to efficiently store the cliquegram / facegram. These representations allows us to talk more about invariants on the clique-/facegrams directly without considering a mergegram invariant of them. 
These can be focusing on the structure of the underlying graph with vertices being the maximal cliques/simplices as edges reflecting the covering relation, i.e. in a poset we would say that $x$ covers $y$ if there is no $z$ with $x < z < y$. This structure can be embedded in the Hasse diagram and can be understood as a directed acyclic graph (DAG). 

Finally, the extensions of the ideas of developed in this work to a multi-filteration setting is of interest. For example, is the mergegram of a $k$-parameter filtration a multi-set of intervals in $\R^k$? Can the clique- and facegram still be interpreted as certain edge-labeled DAGs?

\section{Conclusion}

Nowadays, increasing computational power allows for more complex models of the data to be used in practice. 
Formalizing concepts in (computational) biology need to be balanced in terms of general usefulness of the models as well as the theoretical foundation of the models introduced. 
In this paper, we were motivated by the biological problem of the phylogenetic network reconstruction problem (PNRP) and focused subsequently on setting the groundwork to extend the theory of treegrams to the concepts of cliquegrams and facegrams. 
The goal was to establish new concepts, their descriptions, possible stable invariants and their computations while highlighting their connections with existing ideas. 
Throughout this process, we found interesting connections to current research and constructions in mathematical biology, combinatorial topology, as well as topological data analysis.
As shown, the extensions involve correspondences of different lattices which are either combinatorial or matrix-like in nature.
In particular, these correspondences can lead to speed-ups in the computations. 
Furthermore, we have shown that there exists an efficient invariant in the form of the mergegram with respect to the sizes of the clique-/facegrams which is proven to be stable under the bottleneck distance. 
The current framework shows promise due to its versatility for different applications and we are looking forward to explore further its potential for biological data analysis.

\section*{Code availability}
A python implementation can be found in \cite{ourgitcode}, where the computation of the mergegram of a cliquegram (from a distance matrix) as well as computation of the mergegram of the join-facegram are implemented. 
Additionally, it contains a basic implementation of the Critical Simplex Diagram as well as the jupyter notebooks for the analysis in \textit{Sec.~\ref{subsec:experiments}}.

\section*{Acknowledgements}
The authors gratefully thank Facundo M\'emoli and Woojin Kim for their beneficial mathematical feedback, and the anonymous reviewers for their feedback, comments, and suggestions. The authors also thank Bernd Sturmfels and Ruriko Yoshida for their helpful comments, as well as Krzysztof Bartoszek for his help in collecting the benchmark biological data. PD, JS and AS were supported by the Dioscuri program initiated by the Max Planck Society, jointly managed
with the National Science Centre (Poland), and mutually funded by the
Polish Ministry of Science and Higher Education and the German Federal
Ministry of Education and Research.

\bibliography{references}

\clearpage
\appendix

\section{Postponed proofs}
Below, we provide detailed combinatorial topological proofs of some of the concepts we used. For a detailed introduction in combinatorial algebraic topology, we refer to the book of Kozlov \cite{kozlov2008}.
\begin{proposition}
\label{prop:cliq-clut-iso}
Let $X$ be a set of taxa. 
There is a natural isomorphism 
$$\mathbf{Cliq}(X)\stackrel[\mathbf{Cliq}]{\mathbf{Grph}}{\rightleftarrows}\mathbf{Grph}(X)$$
between the lattice $\mathbf{Cliq}(X)$ of all clique-sets over $X$ and the lattice $\mathbf{Grph}(X)$ of all graphs over $X$, where the maps are defined as follows
\begin{itemize}
    \item For any clique-set $\S_X$ over $X$, the graph $\mathbf{Grph}(\S_X):=(V_X,E_X)$ over $X$ is given by 
\begin{itemize}
    \item the vertex set $V_X:=\{x\in\X \mid x\in C\in \S_X\}$, and 
    \item the edge set $E_X:=\{\{x,x'\}\subset \X\mid \{x,x'\}\in C\in \S_X\text{, }x\neq x'\}$.
\end{itemize}
    \item For any graph $G_X$ over $X$, the clique-set $\mathbf{Cliq}(G_X)$ is the set of maximal cliques of $G_X$.
\end{itemize}
\end{proposition}

\begin{proof}[Proof of \textit{Prop.~\ref{prop:cliq-clut-iso}}
]
It is straightforward to check that the maps $\mathbf{Grph}$ and $\mathbf{Cliq}$ are order preserving.
Now, we check that the pair $(\mathbf{Grph},\mathbf{Cliq})$ is a bijection of sets:
\begin{itemize}
    \item We have:
    \begin{align*}
    \mathbf{Cliq}(\mathbf{Grph}(\S_X))&=\left\{C\subset X\mid C\text{ is a maximal clique of the graph }\mathbf{Grph}(\S_X)\right\}\\
    &=\S_X.
\end{align*}
    \item We have $\mathbf{Grph}(\mathbf{Cliq}(G_X))=G_X$. Indeed:
\begin{align*}
V(\mathbf{Grph}(\mathbf{Cliq}(G_X))):&=\{x\in\X \mid x\in C\in \mathbf{Cliq}(G_X)\}=V(G_X)\text{, and }\\
E(\mathbf{Grph}(\mathbf{Cliq}(G_X))):&=\{\{x,x'\}\subset \X\mid \{x,x'\}\in C\in \mathbf{Cliq}(G_X)\text{, }x\neq x'\}\\
        &=\{\{x,x'\}\subset \X\mid \{x,x'\}\in E(G_X)\text{, }x\neq x'\}\\
&=E(G_X).
\end{align*}
\end{itemize}

\end{proof}

\begin{proof}[Proof of \textit{Prop.~\ref{prop:equivalence}}]
First, we check that the pair $(\Phi,\Psi)$ is a bijection of sets. 
\begin{align*}
    \Psi(\Phi(N_X))(x,x')&=\min\{t\in\R\mid x,x'\text{ belong to some clique in }\Phi(N_X)\}\\
    &=\min\{t\in\R\mid N_X(x,x')\leq t\}\\
    &=N_X(x,x').
\end{align*}
\begin{align*}
\Phi(\Psi(\C _X))(t)&=\bigvee\{\{x,x'\}\mid \Psi(\C _X)(x,x')\leq t\}\\
&=\bigvee\{\{x,x'\}\mid x,x' \text{ belong to some clique in }\C _X(s)\text{, for some }s\leq t\}\\
&=\bigvee\{\{x,x'\}\mid x,x' \text{ belong to some clique in }\C _X(t)\}\hspace{1em} (\text{since }\C _X(s)\leq \C _X(t))\\
&=\C _X(t).
\end{align*}
Now, we show that the bijection pair $(\Phi,\Psi)$ preserves the respective orders of the lattices.
\begin{itemize}
    \item $\Phi$ is order reversing: Suppose that $N_X,N'_X$ are two phylogenetic networks over $X$ such that $N_X\geq N'_X$. Let $t\in\R$. Then, for any $x,x'\in X$, we have $N_X(x,x')\leq t\Rightarrow N'_X(x,x')\leq t$. Thus, $\Phi(N_X)(t)\leq \Phi(N'_X)(t)$. 
    \item $\Psi$ is order preserving: Suppose that $\C _X,\C '_X$ are two cliquegrams over $X$ such that $\C _X\leq \C '_X$. Let $x,x'\in X$. Then, for any $t\in\R$, if $x,x'$ belong to some clique in $\C _X(t)$, $x,x'$ must belong to some clique in $\C '_X(t)$.
\end{itemize}
\end{proof}

\begin{proposition}
\label{prop:clut-simp}
Let $X$ be a set of taxa. 
There is a natural isomorphism 
$$\mathbf{Face}(X)\stackrel[\mathbf{Face}]{\mathbf{Simp}}{\rightleftarrows}\mathbf{Simp}(X)$$
between the lattice $\mathbf{Face}(X)$ of all face-sets of $X$ and the lattice $\mathbf{Simp}(X)$ of simplicial complexes over $X$, where the maps are defined as follows
\begin{itemize}
    \item For any face-set $\C _X$ over $X$, the simplicial complex $\mathbf{Simp}(\C _X)$ over $X$ is given by $$\mathbf{Simp}(\C _X):=\{S\subset X\mid S\subset C\text{, for some }C\in \C _X\}.$$
    \item For any simplicial complex $\S_X$ over $X$, the face-set \sloppy$\mathbf{Face}(\S_X):=\{\text{maximal faces of }\S_X\}$.
\end{itemize}
\end{proposition}
\begin{proof}
Proof is straightforward and omitted.
\end{proof}

\begin{remark}
By the first two bullets of Ex.~\ref{ex:subpartition}, we see that \textit{Prop.~\ref{prop:cliq-clut-iso}} generalizes the isomorphism between partitions on $X$ and equivalence relations on $X$.
\end{remark}

\begin{proof}[Proof of \textit{Prop.~\ref{thm:equivalencetwo}}]
First, we check that the pair $(\Phi,\Psi)$ is a bijection of sets. 
\begin{align*}
    \Psi(\Phi(\F _X))(S)&=\min\{t\in\R\mid S\text{ is contained in some face in }\Phi(\F _X)(t)\}\\
    &=\min\{t\in\R\mid \F _X(S)\leq t\}\\
    &=\F _X(S).
\end{align*}
\begin{align*}
\Phi(\Psi(\C _X))(t)&=\bigvee\{S\in\mathbf{pow}(X)\mid \Psi(\C _X)(S)\leq t\}\\
&=\bigvee\{S\in\mathbf{pow}(X)\mid S \in \C _X(s)\text{, for some }s\leq t\}\\
&=\bigvee\{S\in\mathbf{pow}(X)\mid S \in \C _X(t)\}\hspace{1em} (\text{since }\C _X(s)\leq \C _X(t))\\
&=\C _X(t).
\end{align*}
Now, we show that the bijection pair $(\Phi,\Psi)$ preserves the respective orders of the lattices.
\begin{itemize}
    \item $\Phi$ is order reversing: Suppose that $\F _X,\F '_X$ are two filtrations over $X$ such that $\F _X\geq \F '_X$. Let $t\in\R$. Then, for any $S\in \mathbf{pow}(X)$, we have $\F _X(S)\leq t\Rightarrow \F '_X(S)\leq t$. Thus, $\Phi(\F _X)(t)\leq \Phi(\F '_X)(t)$. 
    \item $\Psi$ is order preserving: Suppose that $\C _X,\C '_X$ are two facegrams over $X$ such that $\C _X\leq \C '_X$. Let $S\subset X$. Then, for any $t\in\R$, if $S$ is contained in some face in $\C _X(t)$, then $S$ must be contained in some face in $\C '_X(t)$.
\end{itemize}
\end{proof}

\section{Computational aspects}

\subsubsection{Running times}\label{subsec:App_runtimes}

Here we list a larger table for the datasets in \textit{Sec. \ref{subsec:experiments}}. 
The Algorithms listed with * are slight variations, which try to incorporate, e.g., in the case of the cliquegram, a choice between finding all maximal cliques or those which contain only the new edges. 
Also listed in the table are some of the properties of the datasets, which influence the runtimes.

\begin{table}[t]
\resizebox{!}{0.79\textheight}{\rotatebox{90}{
\begin{tabular}{lcccc|rrr|rrrr|rrr}
\toprule
\textbf{dataset} & \textbf{number} & \textbf{number} & \textbf{no. unique}
& $\varnothing$ new edges
& \multicolumn{7}{c}{\textbf{time (in s) for mergegram computation of the (join)}} & 
\multicolumn{3}{c}{\textbf{size mergegram}} \\
 & \textbf{trees} & \textbf{taxa} & \textbf{(minimizer)} & \textbf{(aux. Graph)} 
 & $\mathbf{\caC}$ (Alg1) & $\mathbf{\caC}$ (Alg1*) & $\mathbf{\caC}$ (Alg2) 
 & $\mathbf{\caT}$ & $\mathbf{\caF}$ (Alg1) & $\mathbf{\caF}$ (Alg1*) & $\mathbf{\caF}$ (Alg2) 
 & $\mathbf{\caC}$ & $\mathbf{\caT}$ & $\mathbf{\caF}$ \\
\midrule
similar Trees & 10 & 20 & 19 & 21.05 & 0.00 & 0.00 & 0.01 & 0.01 & 0.00 & 0.00 & 0.00 & 368 & 31 & 109 \\
& 50 & 20 & 11 & 36.36 & 0.00 & 0.00 & 0.00 & 0.03 & 0.01 & 0.01 & 0.01 & 118 & 30 & 332 \\
& 100 & 20 & 9 & 44.44 & 0.00 & 0.00 & 0.00 & 0.07 & 0.02 & 0.03 & 0.04 & 126 & 30 & 546 \\
& 500 & 20 & 8 & 50.00 & 0.00 & 0.00 & 0.00 & 0.35 & 0.36 & 0.26 & 0.28 & 87 & 30 & 1399 \\
& 10 & 40 & 29 & 55.17 & 0.13 & 0.10 & 0.26 & 0.01 & 0.00 & 0.00 & 0.01 & 11235 & 61 & 229 \\
& 50 & 40 & 19 & 84.21 & 0.07 & 0.03 & 0.08 & 0.07 & 0.04 & 0.03 & 0.04 & 4375 & 58 & 760 \\
& 100 & 40 & 18 & 88.89 & 0.03 & 0.03 & 0.09 & 0.14 & 0.16 & 0.09 & 0.09 & 3973 & 58 & 1461 \\
& 500 & 40 & 13 & 123.08 & 0.04 & 0.04 & 0.12 & 0.67 & 2.74 & 2.20 & 2.66 & 4777 & 58 & 5932 \\
& 10 & 60 & 41 & 87.80 & 6.19 & 4.48 & 13.45 & 0.02 & 0.01 & 0.01 & 0.02 & 501022 & 87 & 325 \\
& 50 & 60 & 31 & 116.13 & 17.59 & 14.94 & 40.20 & 0.12 & 0.11 & 0.05 & 0.11 & 1586455 & 85 & 1181 \\
& 100 & 60 & 24 & 150.00 & 4.95 & 6.75 & 16.75 & 0.22 & 0.41 & 0.16 & 0.29 & 575393 & 84 & 2228 \\
& 500 & 60 & 16 & 225.00 & 1.72 & 2.27 & 6.35 & 1.12 & 8.61 & 3.77 & 5.02 & 298664 & 85 & 9924 \\
 \hline
trees (Blobs)& 10 & 20 & 116 & 3.45 & 0.02 & 0.01 & 0.02 & 0.01 & 0.00 & 0.00 & 0.00 & 521 & 39 & 165 \\
& 50 & 20 & 175 & 2.29 & 0.04 & 0.01 & 0.02 & 0.03 & 0.01 & 0.01 & 0.02 & 836 & 39 & 508 \\
& 100 & 20 & 180 & 2.22 & 0.04 & 0.01 & 0.02 & 0.07 & 0.03 & 0.11 & 0.03 & 756 & 39 & 770 \\
& 500 & 20 & 188 & 2.13 & 0.04 & 0.01 & 0.02 & 0.38 & 0.27 & 0.31 & 0.33 & 938 & 39 & 2030 \\
& 10 & 40 & 228 & 7.02 & 0.14 & 0.08 & 0.09 & 0.01 & 0.00 & 0.01 & 0.01 & 824 & 79 & 319 \\
& 50 & 40 & 506 & 3.16 & 1.24 & 0.66 & 1.10 & 0.07 & 0.04 & 0.04 & 0.04 & 32040 & 79 & 998 \\
& 100 & 40 & 618 & 2.59 & 2.23 & 0.80 & 1.14 & 0.13 & 0.12 & 0.10 & 0.13 & 48797 & 79 & 1556 \\
& 500 & 40 & 745 & 2.15 & 3.13 & 0.85 & 1.07 & 0.66 & 1.16 & 1.07 & 1.99 & 59760 & 79 & 4341 \\
& 10 & 60 & 383 & 9.40 & 0.42 & 0.23 & 0.26 & 0.02 & 0.01 & 0.01 & 0.02 & 1294 & 119 & 525 \\
& 50 & 60 & 821 & 4.38 & 3.89 & 1.99 & 2.79 & 0.10 & 0.13 & 0.10 & 0.10 & 42213 & 119 & 1653 \\
& 100 & 60 & 1105 & 3.26 & 38.52 & 22.28 & 36.18 & 0.22 & 0.37 & 0.27 & 0.23 & 909429 & 119 & 2676 \\
& 500 & 60 & 1635 & 2.20 & 125.67 & 44.07 & 54.79 & 1.06 & 4.10 & 2.78 & 7.25 & 2326738 & 119 & 7797 \\
 \hline
random trees & 10 & 20 & 140 & 2.86 & 0.03 & 0.01 & 0.02 & 0.01 & 0.00 & 0.00 & 0.00 & 984 & 39 & 181 \\
& 50 & 20 & 189 & 2.12 & 0.05 & 0.02 & 0.03 & 0.03 & 0.02 & 0.02 & 0.02 & 1693 & 39 & 636 \\
& 100 & 20 & 190 & 2.11 & 0.05 & 0.02 & 0.02 & 0.07 & 0.04 & 0.06 & 0.07 & 1389 & 39 & 1065 \\
& 500 & 20 & 191 & 2.09 & 0.04 & 0.01 & 0.02 & 0.31 & 0.55 & 0.63 & 0.42 & 958 & 39 & 3495 \\
& 10 & 40 & 351 & 4.56 & 1.25 & 1.21 & 2.01 & 0.01 & 0.01 & 0.01 & 0.01 & 72254 & 79 & 406 \\
& 50 & 40 & 730 & 2.19 & 4.70 & 1.90 & 2.23 & 0.06 & 0.09 & 0.05 & 0.04 & 129486 & 79 & 1546 \\
& 100 & 40 & 774 & 2.07 & 7.19 & 3.03 & 3.57 & 0.13 & 0.31 & 0.15 & 0.15 & 221021 & 79 & 2755 \\
& 500 & 40 & 781 & 2.05 & 6.86 & 2.59 & 3.05 & 0.65 & 5.25 & 2.91 & 2.61 & 191907 & 79 & 10529 \\
& 10 & 60 & 565 & 6.37 & 42.63 & 52.02 & 97.05 & 0.02 & 0.01 & 0.01 & 0.02 & 3530433 & 119 & 628 \\
& 50 & 60 & 1561 & 2.31 & 427.16 & 221.89 & 282.25 & 0.10 & 0.25 & 0.09 & 0.09 & 11138659 & 119 & 2542 \\
& 100 & 60 & 1716 & 2.10 & 545.29 & 241.43 & 271.23 & 0.21 & 0.84 & 0.30 & 0.43 & 12928819 & 119 & 4535 \\
& 500 & 60 & 1771 & 2.03 & 882.45 & 427.09 & 459.36 & 1.01 & 14.95 & 8.85 & 7.54 & 22754257 & 119 & 17178 \\
 \hline
phylo. trees & 160 & 68 & 82 & 56.39 & 0.07 & 0.04 & 0.05 & 0.49 & 0.02 & 0.02 & 0.11 & 144 & 134 & 144 \\
\bottomrule
\end{tabular}}}
\caption{Table listing the best runtime of the different algorithms provided to compute the mergegrams of the trees $\caT$ as well as the join-cliquegram $\caC$ and join-facegram $\caF$ of these trees for different datasets and number of trees and number of taxas. The number of unique values in the minimizer together with the average number of new edges per iteration of constructing the auxiliary graph for the cliquegram give a rough estimate of how complex the computations for the cliquegram will be. The sizes of the mergegrams for the treegrams are the mean size of each treegram.}
\label{tab:runtimes_long}
\end{table}

As already noted previously, the join-facegram computations are largely benefitting from the speed-up by using the informations from the labeled mergegrams of the treegrams. 
Comparing the different algorithms, we can see that the algorithm 2 for the join-cliquegram perfoms worse than both versions of algorithm 1. 
This is a result of the fact, that algorithm 2 is more optimized for distance matrices which have nearly all different values as well as a potential parallel implementation. 
Similary, algorithm 2 for the facegram also relies on the computations of the matrices directly and can be more easily parallelized. For single core perfomance it perfoms worse than algorithm 1. In general, it should be noted how large the differences in the number of points in the mergegram of the join is compared to the facegram.

\subsection{Maximal number of cliques in a cliquegram}

The maximal number of maximal simplices in a facegram over $X$ with $n=|X|$ can be $2^n = \sum_{k=1}^n {\binom{n}{k}}$. 
For a cliquegram, due to the fact that a maximal clique appears as soon as all of its faces of codimension 1 are added at one time $t$. Hence, 

\begin{proposition}\label{prop:maximalnumbercliquegram}
Let $\caC_X$ a cliquegram over $X$ with $n=|X|$. The maximal number of unique maximal cliques in $\caC_X$ is bounded above by 
\begin{equation}\label{eq:number_max_cliques}
N_{cl} =  2^n -1 - \sum_{k=2}^{n-1}\left\lceil {\binom{n}{k+1}} / (n-k) \right\rceil .
\end{equation}
This bound is tight. In particular, we have $O(N_{cl}) = O(2^n)$.
\end{proposition}

\begin{proof}
The statement of the proposition follows directly from the following observations 
\begin{enumerate}
\item the number of different $(k-1)$-simplices is $\binom{n}{k}$  for $k=1, \ldots n$,
\item any $k-1$-simplex is a face of $n-k$ many $k$-simplices,
\item the maximal number of $(k-1)$-simplices in $\caC_X$ is ${\binom{n}{k}} - \lceil {\binom{n}{k+1}} / (n-k) \rceil$,
\end{enumerate}

and the fact that 
\[ 
n+1 + \sum_{k=2}^{n-1} {\binom{n}{k}} - \left\lceil {\binom{n}{k+1}} / (n-k) \right\rceil
= N_{cl}
\]

(1) and (2) are clear if we consider a $(k+1)$-simplex to be a subset of cardinality $k$ in $X$ and that there are only $n-k$ elements left in $X$ for a $(k-1)$-simplex to become a $k$-simplex. 
These then imply (3) by computing how many $(k-1)$-simplices we would minimally need to be the "last" subsets of codimension $1$ to enter the filtration and thus add the $k$-simplex they are contained in.

We can find lower and upper bounds for the expression \textit{Eq.~\ref{eq:number_max_cliques}}, namely,

\begin{align*}
    2(2^n-n-2)/(n+1) & = \sum_{k=2}^{n-1} {\binom{n}{k}} / (n-k+1)
    < \sum_{k=2}^{n-1} \left\lceil {\binom{n}{k+1}} / (n-k) \right\rceil \\
 & < \sum_{k=2}^{n-1} \left(1 +  {\binom{n}{k+1}} / (n-2) \right) 
     = \frac{1}{n-2} \left( 2^n + n^2 - 5n +2 \right).
\end{align*} 

For the total numbers, we have for the lower bound
\[ O(2^n - 1 - 2(2^n-n-2)/(n+1)) = O(2^n(1-2/(n+1))) = O(2^n), \]
and for the upper case, 
\[O(2^n - 1 - (2^n + n^2 - 5n +2)/(n-2)) = O( 2^n (1- 1/(n-2))) = O(2^n),\]
therefore, $O(N_{cl}) = O(2^n)$ 
\end{proof}

\end{document}